\documentclass[11pt,english]{article}
\usepackage{lmodern}

\usepackage[T1]{fontenc}
\usepackage[latin9]{inputenc}
\usepackage{geometry}
\usepackage{units}
\usepackage{dsfont}
\usepackage{amsmath}
\usepackage{amsthm}
\usepackage{amssymb}

\makeatletter
\numberwithin{figure}{section}
\numberwithin{equation}{section}
\theoremstyle{plain}
\newtheorem{thm}{\protect\theoremname}[section]
\theoremstyle{plain}
\newtheorem{cor}[thm]{\protect\corollaryname}
\theoremstyle{plain}
\newtheorem{lem}[thm]{\protect\lemmaname}
\theoremstyle{plain}
\newtheorem{prop}[thm]{\protect\propositionname}
\theoremstyle{remark}
\newtheorem{rem}[thm]{\protect\remarkname}
\theoremstyle{definition}
\newtheorem{defn}[thm]{\protect\definitionname}
\theoremstyle{remark}
\newtheorem*{acknowledgement*}{\protect\acknowledgementname}
\theoremstyle{remark}
\newtheorem{claim}[thm]{\protect\claimname}

\usepackage[T1]{fontenc}
\usepackage{float,psfrag,epsfig,color,url,hyperref}
\usepackage{algorithm,algorithmic}
\usepackage{graphicx,relsize}
\usepackage{amssymb,amsfonts,amsmath,amsthm,amscd,dsfont,mathrsfs,mathtools,microtype,nicefrac,pifont}
\usepackage{upgreek}
\usepackage[dvipsnames]{xcolor}
\usepackage{epstopdf,bbm,enumitem}
\usepackage{dsfont,tikz}
\usetikzlibrary{arrows.meta,positioning,calc,decorations.pathreplacing}
\usepackage[mathscr]{euscript}
\usepackage[toc,page]{appendix}
\hypersetup{
  colorlinks,
  linkcolor={red!50!black},
  citecolor={blue!50!black},
  urlcolor={blue!80!black}
}
\usepackage{etoolbox}
\patchcmd{\thebibliography}
  {\settowidth}
  {\setlength{\itemsep}{0pt plus 0.1pt}\settowidth}
  {}{}
\apptocmd{\thebibliography}
  {\small}
  {}{}
\allowdisplaybreaks
\usepackage{custom}
\makeatletter
\renewcommand{\paragraph}{%
  \@startsection{paragraph}{4}%
  {\z@}{1.25ex \@plus 1ex \@minus .2ex}{-1em}%
  {\normalfont\normalsize\bfseries}%
}
\makeatother

\makeatother

\usepackage{babel}
\providecommand{\acknowledgementname}{Acknowledgement}
\providecommand{\claimname}{Claim}
\providecommand{\corollaryname}{Corollary}
\providecommand{\definitionname}{Definition}
\providecommand{\lemmaname}{Lemma}
\providecommand{\propositionname}{Proposition}
\providecommand{\remarkname}{Remark}
\providecommand{\theoremname}{Theorem}

\begin{document}
\def\balign#1\ealign{\begin{align}#1\end{align}}
\def\baligns#1\ealigns{\begin{align*}#1\end{align*}}
\def\balignat#1\ealign{\begin{alignat}#1\end{alignat}}
\def\balignats#1\ealigns{\begin{alignat*}#1\end{alignat*}}
\def\bitemize#1\eitemize{\begin{itemize}#1\end{itemize}}
\def\benumerate#1\eenumerate{\begin{enumerate}#1\end{enumerate}}

\newenvironment{talign*}
 {\let\displaystyle\textstyle\csname align*\endcsname}
 {\endalign}
\newenvironment{talign}
 {\let\displaystyle\textstyle\csname align\endcsname}
 {\endalign}

\def\balignst#1\ealignst{\begin{talign*}#1\end{talign*}}
\def\balignt#1\ealignt{\begin{talign}#1\end{talign}}

\let\originalleft\left
\let\originalright\right
\renewcommand{\left}{\mathopen{}\mathclose\bgroup\originalleft}
\renewcommand{\right}{\aftergroup\egroup\originalright}

\def\Gronwall{Gr\"onwall\xspace}
\def\Holder{H\"older\xspace}
\def\Ito{It\^o\xspace}
\def\Nystrom{Nystr\"om\xspace}
\def\Schatten{Sch\"atten\xspace}
\def\Matern{Mat\'ern\xspace}

\def\tinycitep*#1{{\tiny\citep*{#1}}}
\def\tinycitealt*#1{{\tiny\citealt*{#1}}}
\def\tinycite*#1{{\tiny\cite*{#1}}}
\def\smallcitep*#1{{\scriptsize\citep*{#1}}}
\def\smallcitealt*#1{{\scriptsize\citealt*{#1}}}
\def\smallcite*#1{{\scriptsize\cite*{#1}}}

\def\blue#1{\textcolor{blue}{{#1}}}
\def\green#1{\textcolor{green}{{#1}}}
\def\orange#1{\textcolor{orange}{{#1}}}
\def\purple#1{\textcolor{purple}{{#1}}}
\def\red#1{\textcolor{red}{{#1}}}
\def\teal#1{\textcolor{teal}{{#1}}}

\def\mbi#1{\boldsymbol{#1}} 
\def\mbf#1{\mathbf{#1}}
\def\mrm#1{\mathrm{#1}}
\def\tbf#1{\textbf{#1}}
\def\tsc#1{\textsc{#1}}

\def\mbiA{\mbi{A}}
\def\mbiB{\mbi{B}}
\def\mbiC{\mbi{C}}
\def\mbiDelta{\mbi{\Delta}}
\def\mbif{\mbi{f}}
\def\mbiF{\mbi{F}}
\def\mbih{\mbi{g}}
\def\mbiG{\mbi{G}}
\def\mbih{\mbi{h}}
\def\mbiH{\mbi{H}}
\def\mbiI{\mbi{I}}
\def\mbim{\mbi{m}}
\def\mbiP{\mbi{P}}
\def\mbiQ{\mbi{Q}}
\def\mbiR{\mbi{R}}
\def\mbiv{\mbi{v}}
\def\mbiV{\mbi{V}}
\def\mbiW{\mbi{W}}
\def\mbiX{\mbi{X}}
\def\mbiY{\mbi{Y}}
\def\mbiZ{\mbi{Z}}

\def\textsum{{\textstyle\sum}} 
\def\textprod{{\textstyle\prod}} 
\def\textbigcap{{\textstyle\bigcap}} 
\def\textbigcup{{\textstyle\bigcup}} 

\def\reals{\mathbb{R}} 
\def\integers{\mathbb{Z}} 
\def\rationals{\mathbb{Q}} 
\def\naturals{\mathbb{N}} 
\def\complex{\mathbb{C}} 

\def\what#1{\widehat{#1}}

\def\twovec#1#2{\left[\begin{array}{c}{#1} \\ {#2}\end{array}\right]}
\def\threevec#1#2#3{\left[\begin{array}{c}{#1} \\ {#2} \\ {#3} \end{array}\right]}
\def\nvec#1#2#3{\left[\begin{array}{c}{#1} \\ {#2} \\ \vdots \\ {#3}\end{array}\right]} 

\def\maxeig#1{\lambda_{\mathrm{max}}\left({#1}\right)}
\def\mineig#1{\lambda_{\mathrm{min}}\left({#1}\right)}

\def\Re{\operatorname{Re}} 
\def\indic#1{\mbb{I}\left[{#1}\right]} 
\def\logarg#1{\log\left({#1}\right)} 
\def\polylog{\operatorname{polylog}}
\def\maxarg#1{\max\left({#1}\right)} 
\def\minarg#1{\min\left({#1}\right)} 
\def\Earg#1{\E\left[{#1}\right]}
\def\Esub#1{\E_{#1}}
\def\Esubarg#1#2{\E_{#1}\left[{#2}\right]}
\def\bigO#1{\mathcal{O}\left(#1\right)} 
\def\littleO#1{o(#1)} 
\def\P{\mbb{P}} 
\def\Parg#1{\P\left({#1}\right)}
\def\Psubarg#1#2{\P_{#1}\left[{#2}\right]}
\def\Trarg#1{\Tr\left[{#1}\right]} 
\def\trarg#1{\tr\left[{#1}\right]} 
\def\Var{\mrm{Var}} 
\def\Vararg#1{\Var\left[{#1}\right]}
\def\Varsubarg#1#2{\Var_{#1}\left[{#2}\right]}
\def\Cov{\mrm{Cov}} 
\def\Covarg#1{\Cov\left[{#1}\right]}
\def\Covsubarg#1#2{\Cov_{#1}\left[{#2}\right]}
\def\Corr{\mrm{Corr}} 
\def\Corrarg#1{\Corr\left[{#1}\right]}
\def\Corrsubarg#1#2{\Corr_{#1}\left[{#2}\right]}
\newcommand{\info}[3][{}]{\mathbb{I}_{#1}\left({#2};{#3}\right)} 
\newcommand{\staticexp}[1]{\operatorname{exp}(#1)} 
\newcommand{\loglihood}[0]{\mathcal{L}} 


\providecommand{\arccos}{\mathop\mathrm{arccos}}
\providecommand{\dom}{\mathop\mathrm{dom}}
\providecommand{\diag}{\mathop\mathrm{diag}}
\providecommand{\tr}{\mathop\mathrm{tr}}
\providecommand{\card}{\mathop\mathrm{card}}
\providecommand{\sign}{\mathop\mathrm{sign}}
\providecommand{\conv}{\mathop\mathrm{conv}} 
\def\rank#1{\mathrm{rank}({#1})}
\def\supp#1{\mathrm{supp}({#1})}

\providecommand{\minimize}{\mathop\mathrm{minimize}}
\providecommand{\maximize}{\mathop\mathrm{maximize}}
\providecommand{\subjectto}{\mathop\mathrm{subject\;to}}

\def\openright#1#2{\left[{#1}, {#2}\right)}

\ifdefined\nonewproofenvironments\else
\ifdefined\ispres\else
 
\fi
\makeatletter
\@addtoreset{equation}{section}
\makeatother
\def\theequation{\thesection.\arabic{equation}}

\newcommand{\cmark}{\ding{51}}

\newcommand{\xmark}{\ding{55}}

\newcommand{\eq}[1]{\begin{align}#1\end{align}}
\newcommand{\eqn}[1]{\begin{align*}#1\end{align*}}
\renewcommand{\Pr}[1]{\mathbb{P}\left( #1 \right)}
\newcommand{\Ex}[1]{\mathbb{E}\left[#1\right]}

\newcommand{\matt}[1]{{\textcolor{Maroon}{[Matt: #1]}}}
\newcommand{\kook}[1]{{\textcolor{blue}{[Kook: #1]}}}
\definecolor{OliveGreen}{rgb}{0,0.6,0}
\newcommand{\sv}[1]{{\textcolor{OliveGreen}{[Santosh: #1]}}}

\global\long\def\on#1{\operatorname{#1}}%

\global\long\def\bw{\mathsf{Ball\ walk}}%
\global\long\def\sw{\mathsf{Speedy\ walk}}%
\global\long\def\gw{\mathsf{Gaussian\ walk}}%
\global\long\def\ps{\mathsf{Proximal\ sampler}}%
\global\long\def\dw{\mathsf{Dikin\ walk}}%

\global\long\def\chr{\mathsf{Coordinate\ Hit\text{-}and\text{-}Run}}%
\global\long\def\har{\mathsf{Hit\text{-}and\text{-}Run}}%
\global\long\def\gc{\mathsf{Gaussian\ cooling}}%
\global\long\def\ino{\mathsf{\mathsf{In\text{-}and\text{-}Out}}}%
\global\long\def\tgc{\mathsf{Tilted\ Gaussian\ cooling}}%
\global\long\def\PS{\mathsf{PS}}%
\global\long\def\psunif{\mathsf{PS}_{\textup{unif}}}%
\global\long\def\psexp{\mathsf{PS}_{\textup{exp}}}%
\global\long\def\psann{\mathsf{PS}_{\textup{ann}}}%
\global\long\def\psgauss{\mathsf{PS}_{\textup{Gauss}}}%
\global\long\def\eval{\mathsf{Eval}}%
\global\long\def\mem{\mathsf{Mem}}%
\global\long\def\unif{\textup{unif}}%

\global\long\def\O{O}%
\global\long\def\Otilde{\widetilde{O}}%
\global\long\def\Omtilde{\widetilde{\Omega}}%

\global\long\def\E{\mathbb{E}}%
\global\long\def\Z{\mathbb{Z}}%
\global\long\def\P{\mathbb{P}}%
\global\long\def\N{\mathbb{N}}%

\global\long\def\R{\mathbb{R}}%
\global\long\def\Rd{\mathbb{R}^{d}}%
\global\long\def\Rdd{\mathbb{R}^{d\times d}}%
\global\long\def\Rn{\mathbb{R}^{n}}%
\global\long\def\Rnn{\mathbb{R}^{n\times n}}%

\global\long\def\psd{\mathbb{S}_{+}^{d}}%
\global\long\def\pd{\mathbb{S}_{++}^{d}}%

\global\long\def\defeq{\stackrel{\mathrm{{\scriptscriptstyle def}}}{=}}%

\global\long\def\veps{\varepsilon}%
\global\long\def\lda{\lambda}%
\global\long\def\vphi{\varphi}%
\global\long\def\K{\mathcal{K}}%

\global\long\def\half{\frac{1}{2}}%
\global\long\def\nhalf{\nicefrac{1}{2}}%
\global\long\def\texthalf{{\textstyle \frac{1}{2}}}%
\global\long\def\ltwo{L^{2}}%

\global\long\def\ind{\mathds{1}}%
\global\long\def\op{\mathsf{op}}%
\global\long\def\ch{\mathsf{Ch}}%
\global\long\def\kls{\mathsf{KLS}}%
\global\long\def\ts{\mathsf{Ts}}%
\global\long\def\hs{\textup{HS}}%

\global\long\def\cpi{C_{\mathsf{PI}}}%
\global\long\def\clsi{C_{\mathsf{LSI}}}%
\global\long\def\cch{C_{\mathsf{Ch}}}%
\global\long\def\clch{C_{\mathsf{logCh}}}%
\global\long\def\cexp{C_{\mathsf{exp}}}%
\global\long\def\cgauss{C_{\mathsf{Gauss}}}%

\global\long\def\chooses#1#2{_{#1}C_{#2}}%

\global\long\def\vol{\on{vol}}%

\global\long\def\law{\on{law}}%

\global\long\def\tr{\on{tr}}%

\global\long\def\diag{\on{diag}}%

\global\long\def\diam{\on{diam}}%

\global\long\def\poly{\on{poly}}%

\global\long\def\polylog{\on{polylog}}%

\global\long\def\Diag{\on{Diag}}%

\global\long\def\inter{\on{int}}%

\global\long\def\esssup{\on{ess\,sup}}%

\global\long\def\proj{\on{Proj}}%

\global\long\def\e{\mathrm{e}}%

\global\long\def\id{\mathrm{id}}%

\global\long\def\supp{\on{supp}}%

\global\long\def\spanning{\on{span}}%

\global\long\def\rows{\on{row}}%

\global\long\def\cols{\on{col}}%

\global\long\def\rank{\on{rank}}%

\global\long\def\T{\mathsf{T}}%

\global\long\def\bs#1{\boldsymbol{#1}}%

\global\long\def\eu#1{\EuScript{#1}}%

\global\long\def\mb#1{\mathbf{#1}}%

\global\long\def\mbb#1{\mathbb{#1}}%

\global\long\def\mc#1{\mathcal{#1}}%

\global\long\def\mf#1{\mathfrak{#1}}%

\global\long\def\ms#1{\mathscr{#1}}%

\global\long\def\mss#1{\mathsf{#1}}%

\global\long\def\msf#1{\mathsf{#1}}%

\global\long\def\textint{{\textstyle \int}}%
\global\long\def\Dd{\mathrm{D}}%
\global\long\def\D{\mathrm{d}}%
\global\long\def\grad{\nabla}%
 
\global\long\def\hess{\nabla^{2}}%
 
\global\long\def\lapl{\triangle}%
 
\global\long\def\deriv#1#2{\frac{\D#1}{\D#2}}%
 
\global\long\def\pderiv#1#2{\frac{\partial#1}{\partial#2}}%
 
\global\long\def\de{\partial}%
\global\long\def\lagrange{\mathcal{L}}%
\global\long\def\Div{\on{div}}%

\global\long\def\Gsn{\mathcal{N}}%
 
\global\long\def\BeP{\textnormal{BeP}}%
 
\global\long\def\Ber{\textnormal{Ber}}%
 
\global\long\def\Bern{\textnormal{Bern}}%
 
\global\long\def\Bet{\textnormal{Beta}}%
 
\global\long\def\Beta{\textnormal{Beta}}%
 
\global\long\def\Bin{\textnormal{Bin}}%
 
\global\long\def\BP{\textnormal{BP}}%
 
\global\long\def\Dir{\textnormal{Dir}}%
 
\global\long\def\DP{\textnormal{DP}}%
 
\global\long\def\Expo{\textnormal{Expo}}%
 
\global\long\def\Gam{\textnormal{Gamma}}%
 
\global\long\def\GEM{\textnormal{GEM}}%
 
\global\long\def\HypGeo{\textnormal{HypGeo}}%
 
\global\long\def\Mult{\textnormal{Mult}}%
 
\global\long\def\NegMult{\textnormal{NegMult}}%
 
\global\long\def\Poi{\textnormal{Poi}}%
 
\global\long\def\Pois{\textnormal{Pois}}%
 
\global\long\def\Unif{\textnormal{Unif}}%

\global\long\def\bpar#1{\bigl(#1\bigr)}%
\global\long\def\Bpar#1{\Bigl(#1\Bigr)}%

\global\long\def\abs#1{|#1|}%
\global\long\def\babs#1{\bigl|#1\bigr|}%
\global\long\def\Babs#1{\Bigl|#1\Bigr|}%

\global\long\def\snorm#1{\|#1\|}%
\global\long\def\bnorm#1{\bigl\Vert#1\bigr\Vert}%
\global\long\def\Bnorm#1{\Bigl\Vert#1\Bigr\Vert}%

\global\long\def\sbrack#1{[#1]}%
\global\long\def\bbrack#1{\bigl[#1\bigr]}%
\global\long\def\Bbrack#1{\Bigl[#1\Bigr]}%

\global\long\def\sbrace#1{\{#1\}}%
\global\long\def\bbrace#1{\bigl\{#1\bigr\}}%
\global\long\def\Bbrace#1{\Bigl\{#1\Bigr\}}%

\global\long\def\Abs#1{\left\lvert #1\right\rvert }%
\global\long\def\Par#1{\left(#1\right)}%
\global\long\def\Brack#1{\left[#1\right]}%
\global\long\def\Brace#1{\left\{  #1\right\}  }%

\global\long\def\inner#1{\langle#1\rangle}%
 
\global\long\def\binner#1#2{\left\langle {#1},{#2}\right\rangle }%

\global\long\def\norm#1{\lVert#1\rVert}%
\global\long\def\onenorm#1{\norm{#1}_{1}}%
\global\long\def\twonorm#1{\norm{#1}_{2}}%
\global\long\def\infnorm#1{\norm{#1}_{\infty}}%
\global\long\def\fronorm#1{\norm{#1}_{\text{F}}}%
\global\long\def\nucnorm#1{\norm{#1}_{*}}%
\global\long\def\staticnorm#1{\|#1\|}%
\global\long\def\statictwonorm#1{\staticnorm{#1}_{2}}%

\global\long\def\mmid{\mathbin{\|}}%

\global\long\def\otilde#1{\widetilde{O}(#1)}%
\global\long\def\wtilde{\widetilde{W}}%
\global\long\def\wt#1{\widetilde{#1}}%

\global\long\def\KL{\msf{KL}}%
\global\long\def\dtv{d_{\textrm{\textup{TV}}}}%
\global\long\def\FI{\msf{FI}}%
\global\long\def\tv{\msf{TV}}%
\global\long\def\TV{\msf{TV}}%

\global\long\def\cov{\on{cov}}%
\global\long\def\var{\on{Var}}%
\global\long\def\ent{\on{Ent}}%

\global\long\def\cred#1{\textcolor{red}{#1}}%
\global\long\def\cblue#1{\textcolor{blue}{#1}}%
\global\long\def\cgreen#1{\textcolor{green}{#1}}%
\global\long\def\ccyan#1{\textcolor{cyan}{#1}}%

\global\long\def\iff{\Leftrightarrow}%
 
\global\long\def\textfrac#1#2{{\textstyle \frac{#1}{#2}}}%

\title{Zeroth-order Logconcave Sampling\date{}\author{Yunbum Kook\\ Georgia Tech\\  \texttt{yb.kook@gatech.edu} \and Santosh S. Vempala\\ Georgia Tech\\ \texttt{vempala@gatech.edu}}}
\maketitle
\begin{abstract}
We study the zeroth-order query complexity of sampling from a general
logconcave distribution: given access to an evaluation oracle for
a convex function $V:\mathbb{R}^{d}\rightarrow\mathbb{R}\cup\{\infty\}$,
output a point from a distribution within $\varepsilon$-distance
to the density proportional to $e^{-V}$. A long line of work provides
efficient algorithms for this problem in TV distance, assuming a pointwise
warm start (i.e., in $\infty$-R\'enyi divergence), and using annealing
to generate such a warm start. Here, we address the natural and more
general problem of using a $q$-R\'enyi divergence warm start to
generate a sample that is $\varepsilon$-close in $q$-R\'enyi divergence.
Our first main result is an algorithm with this end-to-end guarantee
with state-of-the-art complexity for $q=\widetilde{\Omega}(1)$. Our
second result shows how to generate a $q$-R\'enyi divergence warm
start directly via annealing, by maintaining $q$-R\'enyi divergence
throughout, thereby obtaining a streamlined analysis and improved
complexity. Such results were previously known only under the stronger
assumptions of smoothness and access to first-order oracles. We also
show a lower bound for Gaussian annealing by disproving a geometric
conjecture about quadratic tilts of isotropic logconcave distributions. 

Central to our approach, we establish hypercontractivity of the heat
adjoint and translate this into improved mixing time guarantees for
the Proximal Sampler. The resulting analysis of both sampling and
annealing follows a simplified and natural path, directly tying convergence
rates to isoperimetric constants of the target distribution.  

\thispagestyle{empty}

\newpage{}
\end{abstract}
\tableofcontents{}

\thispagestyle{empty}\newpage{}

\section{Introduction}

\setcounter{page}{1}We study the zeroth-order query complexity of
sampling from an \emph{arbitrary} logconcave distribution $\pi$ with
an \emph{arbitrary} output guarantee, assuming access to an evaluation
oracle for $-\log\pi$ up to a fixed unknown additive constant (equivalently,
an evaluation oracle for a convex function $V$, with the goal of
sampling from the distribution with density proportional to $e^{-V}$).
This is a general setting for logconcave sampling---for example,
it includes uniform sampling from convex bodies, which will serve
as a running example throughout this paper. Results in this general
setting provide a theoretical baseline, highlighting the benefits
of additional assumptions often made in other sampling work, such
as log-smoothness (i.e., $-\nabla\log\pi$ is Lipschitz) or access
to first-order oracles (i.e., $-\nabla\log\pi$). Indeed, much of
our work is motivated by gaps between this general case and more favorable
settings where first-order oracles are available for unconstrained
and log-smooth distributions. In practical terms, zeroth-order samplers
are natural default algorithms, since evaluation oracles typically
require only function value computations, making them the most general
and easiest to implement. 

We now make the problem setup more explicit, focusing first on uniform
sampling. Let $\mathcal{K}\subset\mathbb{R}^{d}$ be a convex body
(i.e., a convex and compact set with nonempty interior) containing
the $d$-dimensional unit ball $B_{1}$. Given access to the membership
oracle for $\mathcal{K}$ (which answers whether a queried point belongs
to $\mathcal{K}$) \cite{GLS93geometric}, we ask: how can we sample
uniformly from $\mathcal{K}$, and how many oracle queries are required?
More formally, given a target accuracy $\varepsilon>0$ and a probability
divergence (or distance) $\mathsf{D}$---such as the total variation
($\tv$) distance or $\chi^{2}$-divergence---what is the query complexity
of generating a sample whose law is $\varepsilon$-close to the uniform
distribution $\pi$ over $\mathcal{K}$ in $\mathsf{D}$?

In this work, we propose a uniform sampler that allows us to establish
the first \emph{balanced end-to-end} complexity guarantee, meaning
that the R\'enyi divergence orders of the initial distribution and
the output distribution are the same. This result parallels guarantees
found in sampling research in the log-smooth setting. Next, we provide
a simple annealing scheme that, for the first time, generates a warm
start in $q$-R\'enyi divergence. To this end, we prove hypercontractivity
under simultaneous heat flow, under a logarithmic Sobolev inequality.
Finally, we extend these results to arbitrary logconcave distributions
under evaluation oracles. Alongside, we present a streamlined analysis
leveraging tools from diverse mathematical fields. Combining these
contributions, we establish the improved zeroth-order query complexity
of sampling from an arbitrary logconcave distribution with $\eu R_{q}$-guarantees.

\subsection{Background and results}

In this section, $\pi_{\mathcal{K}}$ denotes the uniform distribution
over a convex body $\mathcal{K}$ with diameter $D$, and $\pi$ refers
to a logconcave distribution over $\mathbb{R}^{d}$. For clarity,
we recall the $q$-R\'enyi divergence, defined as $\eu R_{q}(\mu\mmid\nu):=\frac{q}{q-1}\log\,\norm{\nicefrac{\D\mu}{\D\nu}}_{L^{q}(\nu)}$
for $q>1$, and its monotonicity in $q$, and $2\,\norm{\cdot}_{\tv}^{2}\leq\KL=\lim_{q\downarrow1}\eu R_{q}\leq\eu R_{2}=\log(1+\chi^{2})\leq\chi^{2}\leq\eu R_{\infty}=\log\esssup_{\nu}\nicefrac{\D\mu}{\D\nu}$.
See \S\ref{subsec:Preliminaries} for more details. When a sampler
is initialized with a distribution $\nu_{0}$ to sample from a target
distribution $\nu$, we define the warmness parameter as $M_{q}=\norm{\D\nu_{0}/\D\nu}_{L^{q}(\nu)}$. 

\subsubsection{Uniform sampling from a warm start}

\paragraph{Classical algorithms: Ball walk and Hit-and-run.}

The $\bw$ and $\har$ are well-studied random walks (i.e., Markov
chains) for uniform sampling. These algorithms are analyzed through
the \emph{conductance} framework, where a lower bound on conductance
yields a mixing-time guarantee. Establishing such a lower bound involves
two ingredients: one-step coupling (i.e., the closeness of transition
kernels at two nearby points) and a Cheeger isoperimetric inequality
for the target distribution. The Cheeger constant of a probability
measure $\pi$ over $\Rd$ is defined as
\[
C_{\ch}(\pi):=\inf_{S:\pi(S)\le\frac{1}{2}}\frac{\pi(\de S)}{\pi(S)}\,,
\]
where $S\subset\Rd$ is any open set with smooth boundary, $S_{\eta}=\{x:d(x,S)\leq\eta\}$,
and $\pi(\de S):=\liminf_{\eta\downarrow0}\frac{\pi(S_{\eta})-\pi(S)}{\eta}$.

The $\bw$ was introduced by Lov\'asz \cite{lovasz90compute}; in
each step, it samples a random point uniformly from a ball of fixed
radius around the current point, accepting it if the point lies within
$\mathcal{K}$. Later, Lov\'asz and Simonovits established that the
$\bw$ satisfies $\eu R_{\infty}\to\tv$ with complexity $\poly(d,D,M_{\infty}/\veps)$
\cite[Remark on p.\,398]{LS93random}, meaning the initial warmness
should be given in $\eu R_{\infty}$-divergence and the final output
is $\varepsilon$-close to $\pi_{\mathcal{K}}$ in $\tv$-distance.
They also analyzed an indirect approach that uses the $\bw$ (with
a Metropolis filter) to sample from a logconcave distribution and
then post-processes the sample to bring it closer to $\pi_{\mathcal{K}}$.
The resulting algorithm, also of type $\eu R_{\infty}\to\tv$, was
shown to have query complexity $\poly(d,D,\log\nicefrac{M_{\infty}}{\veps})$.

This approach was further improved by Kannan, Lov\'asz, and Simonovits
\cite{KLS97random}, who introduced the \emph{$\sw:\eu R_{\infty}\to\tv$.}
Similar to \cite{LS93random}, it begins by sampling from a biased
logconcave variant $\pi'$ (of $\pi_{\K}$), which is invariant under
the following per-step transition: the next point $x_{i+1}$ is sampled
uniformly from $\K\cap B_{\delta}(x_{i})$ via rejection sampling,
using the uniform distribution over $B_{\delta}(x_{i})$ as the proposal.
Then, a post-processing step is used to reduce the bias. They showed
that the mixing time of this chain is $\O(d^{2}\cch^{-2}(\pi')\log\nicefrac{M_{\infty}}{\veps})$
and that each rejection sampling step requires $\O(M_{\infty})$ queries,
leading to an overall complexity of $\Otilde(M_{\infty}d^{2}D^{2}\log\nicefrac{1}{\veps})$
(due to $\cch(\pi)\gtrsim D^{-1}$ \cite{LS93random}). Hereafter,
by following convention, the $\bw$ refers to the $\sw$ with the
post-processing step. Later, progress on the Kannan-Lov\'asz-Simonovits
(KLS) conjecture \cite{KLS95isop} improved this bound to $\cch^{2}(\pi)\gtrsim(\norm{\cov\pi}\log d)^{-1}$
\cite{Klartag23log}. As a result, the complexity of the $\bw$ is
now $\Otilde(M_{\infty}d^{2}\norm{\cov\pi}\log\nicefrac{1}{\veps})$.

The $\har$ algorithm \cite{Smith84HAR} was first analyzed by Lov\'asz
\cite{Lovasz99hit} via $s$-conductance \cite{LS90mixing}, with
complexity $\Otilde(M_{\infty}d^{2}D^{2}/\varepsilon^{2})$. Lov\'asz
and Vempala \cite{LV06hit} showed that the conductance of $\har$
is in fact $\Omega(1/dD)$ resulting in a mixing time of $\Otilde(d^{2}D^{2}\log\nicefrac{M_{2}}{\veps})$.
Each step of $\har$ selects a random chord passing the current point
and moves to a random point on the chord. Here, uniform sampling from
the chord is implemented via rejection sampling with complexity $\O(\log D)$,
where the proposal distribution is uniform over a slightly longer
chord (found by binary search). They established a novel Cheeger isoperimetric
inequality in terms of the cross-cut ratio (rather than Euclidean
distance) and showed $\har:\eu R_{2}\to\eu R_{2}$ with complexity
$\Otilde(d^{2}D^{2}\log\nicefrac{M_{2}}{\veps})$, making it the first
random walk with a balanced guarantee (and to date the only walk with
logarithmic dependence on both the initial warmness and the target
error). From an $\O(1)$-warm start in $\eu R_{\infty}(\geq\eu R_{2})$,
the $\bw$ currently has better complexity than $\har$.

Despite these and many other advances over the past four decades,
a unified theory for the complexity of this fundamental problem is
missing. First, the $\bw$ takes an indirect approach to uniform
sampling: it samples from a biased distribution to obtain a mixing
guarantee, followed by a post-processing step to reduce the bias.
Second, conductance-based analyses require an algorithmic modification
making the chains \emph{lazy} (i.e., the chain remains at its current
point with some probability). It would be desirable to develop a sampling
algorithm that avoids these additional processing steps. Finally,
can we replace the ``unbalanced'' guarantee of the $\bw$ (i.e., $\eu R_{\infty}\to\tv$)
to an arguably more natural $\eu R_{q}\to\eu R_{q}$ guarantee for
any desired $q$?  This question becomes even more natural when considering
recent developments in sampling under log-smoothness.

\paragraph{Sampling under smoothness.}

Research in sampling has made substantial progress by analyzing discretizations
of continuous-time processes. A central example is the \emph{Langevin
diffusion} $(X_{t})_{t\geq0}$ in $\Rd$, described by the stochastic
differential equation
\[
\D X_{t}=-\nabla V(X_{t})\,\D t+\sqrt{2}\,\D B_{t}\,,\qquad X_{0}\sim\nu_{0}\,,
\]
where $V:\Rd\to\R$ is a potential function satisfying mild regularity
conditions, and $B_{t}$ denotes Brownian motion. The law $\nu_{t}$
of $X_{t}$ converges to the stationary distribution $\nu\propto e^{-V}$
as $t\to\infty$. A spectral analysis of the infinitesimal generator
for this process yields a connection between the rate of convergence
and the \emph{Poincar\'e constant of $\nu$} (see \cite[\S4]{BGL14analysis}).
A probability measure $\pi$ on $\Rd$ is said to satisfy a Poincar\'e
inequality\emph{ }with constant $C$ if for any locally Lipschitz
function $f\in L^{2}(\pi)$,
\begin{equation}
\var_{\pi}f:=\int\Bpar{f-\int f\,\D\pi}^{2}\,\D\pi\leq C\int\abs{\nabla f}^{2}\,\D\pi\,,\tag{\ensuremath{\msf{PI}}}\label{eq:pi}
\end{equation}
and the smallest such $C$ is called the Poincar\'e constant $\cpi(\pi)$.
The seminal work of Jordan, Kinderlehrer, and Otto \cite{JKO98variational}
interpreted Langevin diffusion as the $2$-Wasserstein gradient flow
of $\KL(\cdot\mmid\nu)$, revealing a connection to the \emph{Logarithmic
Sobolev Inequality} (LSI), which is stronger than the Poincar\'e
inequality. A probability measure $\pi$ on $\Rd$ is said to satisfy
a logarithmic Sobolev inequality with constant $C$ if for any locally
Lipschitz function $f\in L^{2}(\pi)$,
\begin{equation}
\ent_{\pi}(f^{2}):=\int f^{2}\log f^{2}\,\D\pi-\int f^{2}\,\D\pi\cdot\log\int f^{2}\,\D\pi\leq2C\int\abs{\nabla f}^{2}\,\D\pi\,,\tag{\ensuremath{\msf{LSI}}}\label{eq:lsi}
\end{equation}
and the smallest such $C$ is referred to as the log-Sobolev constant
$\clsi(\pi)$. Under \eqref{eq:pi} and \eqref{eq:lsi} respectively,
the Langevin process converges exponentially:
\begin{equation}
\chi^{2}(\nu_{t}\mmid\nu)\leq\exp\bpar{-\frac{2t}{\cpi(\nu)}}\,\chi^{2}(\nu_{0}\mmid\nu)\,,\quad\text{and}\quad\KL(\nu_{t}\mmid\nu)\leq\exp\bpar{-\frac{2t}{\clsi(\nu)}}\,\KL(\nu_{0}\mmid\nu)\,.\label{eq:Langevin-contraction}
\end{equation}
These bounds were extended to $\eu R_{q}$-divergence in \cite{CLL19exponential,VW23rapid}.

With the assumption of bounded smoothness (i.e., $\beta:=\norm{\hess V}<\infty$),
a short-time discretization of the Langevin process yields an algorithm
that preserves the balanced guarantees of the continuous process.
Specifically, it has been shown to converge in $\KL$ (i.e., $\eu R_{1}\to\eu R_{1}$)
with complexity $\poly(d/\veps,\beta,\clsi(\nu),\log\KL(\nu_{0}\mmid\nu))$
\cite{VW23rapid}. This result was further extended to $\eu R_{q}$-divergence
\cite{CEL24analysis} achieving a query complexity of $\poly(qd/\veps,\beta,\cpi(\nu),\eu R_{q}(\nu_{0}\mmid\nu))$
or $\poly(qd/\veps,\beta,\clsi(\nu),\log\eu R_{q}(\nu_{0}\mmid\nu))$.
We note that the dependence on $\nicefrac{1}{\veps}$ can be made
polylogarithmic by incorporating a Metropolis filter (e.g., $\msf{MALA}$
\cite{CLA21opitmal,WSC22minimax,CG23simple}) or by considering $\msf{MHMC}$
\cite{CV22optimal,CDW20fast,CG23when}). For a comprehensive review,
we refer readers to the monograph \cite{chewi25log}.

The complexity of the $\bw$ can be described as $\poly(qd/\veps,\beta,\cpi(\nu),\eu R_{q}(\nu_{0}\mmid\nu))$.
In fact, for logconcave probability measures $\pi$, it holds that
\[
\frac{1}{4}\leq\frac{C_{\ch}^{-2}(\pi)}{\cpi(\pi)}\leq9\,,
\]
where the first inequality is due to Cheeger \cite{Chee70lower},
and the second follows from the Buser--Ledoux inequality \cite{Buser82iso,Led04spectral}.
Thus, $\cpi$ could be replaced with $\cch^{-2}$, although it remains
unclear how to interpret the smoothness parameter $\beta$ in the
context of convex-body sampling. 

\paragraph{The proximal sampler.}

Continuing this line of work based on algorithmic diffusion, the $\ps$
($\PS$) \cite{LST21structured} has led to substantial progress on
sampling. For a target distribution $\pi^{X}\propto e^{-V}$ over
$\Rd$, $\PS$ considers the augmented distribution $\pi(x,y)\propto\pi^{X}(x)\,\gamma_{h}(y-x)$
over $\Rd\times\Rd$, where $\gamma_{h}$ denotes the centered Gaussian
with covariance $hI_{d}$. Each iteration consists of two steps: for
step size $h>0$, 
\[
\mbox{(i) }y_{k+1}\sim\pi^{Y|X=x_{k}}=\mc N(x_{k},hI_{d}),\quad\mbox{ and }\quad\mbox{(ii) }x_{k+1}\sim\pi^{X|Y=y_{k+1}}.
\]
When $\pi^{X}$ is $\beta$-log-smooth, rejection sampling for the
backward step $x_{i+1}\sim\pi^{X|Y=y_{i+1}}$ can be performed using
$\O(1)$ evaluation queries and one proximal query\footnote{The proximal query for a function $V:\Rd\to\R$ and step size $h>0$
is the function defined as $x\in\Rd\mapsto\arg\min\{V(\cdot)+\frac{1}{2h}\,\norm{\cdot-x}^{2}\}$. } for \emph{any} $y_{i+1}$ and $h\asymp(\beta d)^{-1}$. 

Further refinements were obtained by \cite{CCSW22improved}, which
showed that one iteration of $\PS$ can be interpreted as the composition
of two continuous-time stochastic processes---Brownian motion and
its time-reversal. Since Brownian motion naturally defines a heat
semigroup $(P_{t})_{t\geq0}$ from $L^{2}(\nu*\gamma_{t})$ to $L^{2}(\nu)$,
as $P_{t}f:=f*\gamma_{t}$, and the time-reversal also admits a semigroup
$(Q_{t})_{t\geq0}$ from $L^{2}(\nu)$ to $L^{2}(\nu*\gamma_{t})$,
defined as $Q_{t}f:=P_{t}(f\nu)/P_{t}\nu$ \cite{KP23spectral}, this
allowed \cite{CCSW22improved} to carry out semigroup calculus, leading
to contraction results similar to \eqref{eq:Langevin-contraction}.
In particular, this framework yields query complexity of $\O(qd\beta\cpi(\nu)\,\eu R_{q}(\nu_{0}\mmid\nu)\log\nicefrac{1}{\veps})$
for achieving $\veps$-closeness in $\eu R_{q}$ (i.e., $\PS:\eu R_{q}\to\eu R_{q}$
under \eqref{eq:pi} and smoothness). See \S\ref{subsec:intro-semigroup}
for further details on these semigroups.

\paragraph{Proximal sampler for uniform sampling.}

Convex body sampling does not satisfy bounded smoothness. Kook, Vempala,
and Zhang \cite{KVZ24INO} adapted these ideas to convex body sampling
by introducing the $\ino$ algorithm. They extended the contraction
result to constrained distributions and developed a new efficient
implementation for sampling from $\pi^{X|Y=y}\propto\gamma_{h}(\cdot-y)|_{\K}$.
They introduced a threshold parameter $\tau$: if the number of attempts
in the rejection sampling subroutine with proposal $\mc N(y,hI_{d})$
exceeds $\tau$, then $\ino$ halts and declares failure. To control
this failure probability $\eta\in(0,1)$, they carefully set the parameters
$h,\eta,\tau$ so that with high probability, the backward step can
be completed using $\Otilde(M_{\infty}\polylog\nicefrac{dD}{\eta\veps})$
queries. Consequently, they established $\ino:\eu R_{\infty}\to\eu R_{q}$
with query complexity $\Otilde(qM_{\infty}d^{2}\Lambda\polylog\nicefrac{D}{\eta\veps})$
for $\Lambda:=\norm{\cov\pi}$, generalizing $\bw$'s output guarantee.
Later, Kook and Zhang \cite{KZ25Renyi} established a balanced guarantee
of $\ino:\eu R_{\infty}\to\eu R_{\infty}$ using \eqref{eq:lsi},
with query complexity $\Otilde(M_{\infty}d^{2}D^{2}\polylog\nicefrac{1}{\eta\veps})$.

Nevertheless, some algorithmic tricks (e.g., the failure probability)
still remain in $\ino$. Moreover, the required initial warmness
remains the stringent $\eu R_{\infty}$. Kook and Vempala \cite{KV25faster}
relaxed the warmness to $\eu R_{c}$ with $c=2+\polylog\nicefrac{dD}{\eta\veps}=\Otilde(1)$
for achieving $\veps$-closeness in $\eu R_{2}$ (i.e., $\ino:\eu R_{c}\to\eu R_{2}$
under \eqref{eq:pi}), with query complexity $\Otilde(M_{c}d^{2}\Lambda\polylog\nicefrac{1}{\eta\veps})$
. This result raised the question and gave hope for fully balancing
the warmness and output guarantee.

\paragraph{Proximal sampler with restart for balanced guarantees (\S\ref{sec:ps-unif}).}

We propose the sampler, $\psunif$---essentially $\ino$ with failure
replaced by \emph{restart}: starting with $x_{0}\sim\pi_{0}^{X}$,
step size $h$, and threshold $\tau$, 
\begin{itemize}
\item {[}Forward{]} $y_{k+1}\sim\pi^{Y|X=x_{k}}=\gamma_{h}(\cdot-x_{k})$.
\item {[}Backward{]} $x_{k+1}\sim\pi^{X|Y=y_{k+1}}\propto\gamma_{h}(\cdot-y_{k+1})\,\ind_{\K}(\cdot)$.
To this end, draw $x_{k+1}\sim\mc N(y_{k+1},hI_{d})$ repeatedly until
$x_{k+1}\in\K$. If this rejection loop exceeds $\tau$ trials, declare
failure and restart scratch from a new $x_{0}\sim\pi_{0}^{X}$.
\end{itemize}
One \emph{iteration} refers to the successful completion of both steps.
Note that each check $x_{k+1}\in\K$ incurs a single query to the
membership oracle.

This algorithm always succeeds and has a complexity guarantee from
an $O(1)$-warm start, where the R\'enyi divergence orders of initial
warmness and output distribution guarantee are the same (i.e., $\psunif:\eu R_{q}\to\eu R_{q}$
under \eqref{eq:pi}). This is our first main result. Hereafter,
we use $B_{1}$, $B_{2}$, or $B_{2}^{d}$ to denote a $\ell_{2}$-unit
ball in $\Rd$.
\begin{thm}
[Uniform sampling from a warm start]\label{thm:unif-warm} Let $\pi$
be the uniform distribution over a convex body $\K\subset\Rd$ given
by a membership oracle with $B_{1}\subset\K$ and $\Lambda=\norm{\cov\pi}$.
Given $\veps\in(0,1/100)$, an initial distribution $\pi_{0}$, and
$q\geq2\vee\Omtilde(\log(d^{2}\Lambda\log\frac{1}{\veps}))$ such
that $M_{q}=\norm{\nicefrac{\D\pi_{0}}{\D\pi}}_{L^{q}(\pi)}\leq10$,
consider the algorithm that runs $\psunif$ initialized at $\pi_{0}$
and restarts if $\psunif$ fails before completing $N=\Otilde(qd^{2}\Lambda\log^{2}\frac{1}{\veps})$
iterations. There exist choices of $h$ and $\tau$ such that this
algorithm restarts with probability at most $1/50$, and its output
$Z$ satisfies $\eu R_{q}(\law Z\mmid\pi)\leq\veps$, using $\Otilde(qd^{2}\Lambda\log^{6}\frac{1}{\veps})$
\textup{total membership queries in expectation}.
\end{thm}

\subsubsection{Improved warm-start generation via annealing}

$\ino$ exhibits polynomial dependence on $M_{\Otilde(1)}$, and a
``cold'' start (e.g., the uniform distribution over a small ball)
yields $M_{\Otilde(1)}\lesssim\exp(d,D)$. Hence, efficient warm-start
generation is necessary to reduce the complexity overhead incurred
by a cold start, so warm-start generation algorithms have been developed
alongside sampling algorithms. 

\paragraph{Warm-start generation in the well-conditioned setting.}

The necessity of warmness is not limited to this general setting;
its importance also arises in the well-conditioned setting (i.e.,
$-\log\pi$ is $\alpha$-strongly convex and $\beta$-smooth). In
fact, $\msf{MALA}$ from an $\O(1)$-warm start in $\eu R_{2}$ can
take a step size $h\asymp(\beta d^{1/2})^{-1}$, achieving the first-order
complexity of $\Otilde(\kappa d^{1/2}\polylog\nicefrac{1}{\veps}$)
in the well-conditioned setting. In contrast, from a feasible start,
it can only take a smaller step size $h\asymp(\beta d)^{-1}$, resulting
in $\Otilde(\kappa d\polylog\nicefrac{1}{\veps})$ complexity, where
$\kappa:=\beta/\alpha$ denote the condition number of $-\log\pi$.
Recently, \cite{AC24faster} showed that combining the Underdamped
Langevin Algorithm with $\PS$ can generate an $\O(1)$-warm start
in $\eu R_{2}$-divergence using $\Otilde(\kappa d^{1/2})$ first-order
queries, thus incurring no overhead for $\msf{MALA}$ as a result.

\paragraph{Simulated annealing.}

In our general setup, a classical strategy for warm-start generation
is an \emph{annealing} scheme, which involves a sequence of probability
distributions $\{\mu_{i}\}_{i\leq m}$ to interpolate between an easy-to-sample
initial distribution $\mu_{0}$ and the target distribution $\pi_{\K}$.
The key idea is to ensure that each pair of consecutive distributions
$\mu_{i}$ and $\mu_{i+1}$ is $\O(1)$-close in some appropriate
metric (e.g., $\eu R_{\infty}$-divergence), thereby enabling efficient
transitions by a sampler along the sequence.

In the context of convex-body sampling, several annealing schemes
have been proposed: (i) uniform distributions over truncated convex
bodies intersected with balls of increasing radii~\cite{DFK91random,LS90mixing,LS93random,KLS97random},
(ii) exponential distributions of the form $e^{-cx}$ with decreasing
$c$ over a convex set, derived from the so-called ``pencil'' construction~\cite{LV06simulated},
and (iii) Gaussian distributions with increasing variance, truncated
to the body~\cite{CV15bypass,CV18Gaussian,KZ25Renyi,KV25faster}.

Among these, all approaches except~\cite{KZ25Renyi} provide a warm
start $\mu$ satisfying $\norm{\mu-\pi_{\K}}_{\tv}\leq\varepsilon$,
while \cite{KZ25Renyi} guarantees $\eu R_{\infty}$-warmness with
complexity $\Otilde(d^{2}R^{2})$ for $R^{2}=\E_{\pi_{\K}}[\abs{\cdot}^{2}]$.
\cite{KV25faster} show that $\tv$-warmness can be achieved with
the better complexity of $\Otilde(d^{2}R^{3/2}\Lambda^{1/4}\polylog\nicefrac{1}{\veps})$
. Currently, these are the only known guarantees, and there is no
bound for intermediate warmness results, between $\tv$ and $\eu R_{\infty}$
(i.e., either very weak or strong warmness).

\paragraph{Towards genuine warmness.}

Why do we even care about $\eu R_{q}$-warmness? Given the gap in
known complexities between the two extreme settings discussed above,
bridging this gap is theoretically intriguing in its own right. In
our general setup, it is not clear whether generating $\eu R_{q}$-warmness
without significant overhead is feasible, let alone whether one can
tighten the complexities of warm-start generation and sampling from
a warm start.

Moreover, $\tv$-warmness arguably serves only as a ``pseudo'' warm-start
for downstream sampling, since subsequent sampling guarantees using
this warmness are only in $\tv$-distance due to the necessity of
applying triangle inequalities. In practice, this limitation leads
to accumulation of sampling errors, resulting in somewhat cumbersome
post-processing, such as replacing the original target accuracy $\veps$
by a reduced accuracy $\veps/m$, where $m$ is the number of times
the triangle inequality is applied. While this does not significantly
affect downstream guarantees for applications (since $\varepsilon^{-1}$
is inside a logarithm), it is a bit unsatisfactory and makes the analysis
more complicated.

One might instead generate a $\eu R_{\infty}$-warm start and then
apply the known $\eu R_{\infty}\to\eu R_{q}$ guarantee (see Figure~\ref{fig:complexity-summary}).
However, generating $\eu R_{\infty}$-warmness incurs a higher query
complexity than $\tv$-warmness, making it an unnecessarily expensive
route to an $\eu R_{q}$-guarantee. Lastly, when combined with $\psunif$,
which achieves balanced guarantees, a $\eu R_{q}$-warmness result
would yield clean \emph{end-to-end} guarantees for uniform sampling
in any desired $\eu R_{q}$-divergence.

\begin{figure}[t]
\centering \begin{tikzpicture}[>=Stealth]
\useasboundingbox (-2, -8) rectangle (9.5, 0.5); 
\node (Whdr) at (-1, 0) {\large\bfseries Warmness};
\node (Ohdr) at (10, 0) {\large\bfseries Output guarantee};
 
 
\node (Linf) at (-1, -1.5) {\large $\eu R_\infty$};
\node[anchor=north, text width=6cm, align=center] at (-3, -1) {%
    {\small \cite{KV25sampling}\ $d^2 \bar R^2$}\\
    {\small \cite{KZ25Renyi}$^\dagger$\ $d^2 R^2$}};
 
\node (Lq) at (-1, -4) {\large $\eu R_{q}$};
\node[anchor=north, text width=6cm, align=center] at (-3, -3.6) {%
    {\small\bfseries\color{BrickRed} New:\ $qd^2 \bar R^{3/2}\, \bar\lambda^{1/4}$}};
 
\node (Lc) at (-1, -5.5) {\large $\eu R_{O(1)}$};
 
\node (Ltv) at (-1, -7) {\large $\tv$};
\node[anchor=north, text width=6cm, align=center] at (-3, -6.5) {%
    {\small \cite{KV25faster}\ $d^2 \bar R^{3/2}\, \bar\lambda^{1/4}$}\\
    {\small \cite{CV18Gaussian}$^\dagger$\ $d^2 R^2$}};
 
 
\node (Rinf) at (10, -1.5) {\large $\eu R_\infty$};
\node (Rq)   at (10, -4) {\large $\eu R_{q}$};
\node (R2)   at (10, -6) {\large $\eu R_{2}$};
\node (Rtv)  at (10, -7)  {\large $\tv$};
 
 
\draw[->, thick] (Linf) -- (Rinf)
    node[fill=white, inner sep=1.5pt, pos=0.5, above=2pt]
        {\small \cite{KV25sampling}\ $d^2 \bar R^2$}
    node[fill=white, inner sep=1.5pt, pos=0.5, below=2pt]
        {\small \cite{KZ25Renyi}$^\dagger$\ $d^2 R^2$};
 
\draw[->, thick] ($(Linf.east)+(0.15,-0.25)$) -- ($(Rq.west)+(-0.15,0.15)$);
\node[fill=white, inner sep=1.5pt] at (7, -2.5)
    {\small \cite{KV25sampling}\ $qd^2 \bar\lambda$};
\node[fill=white, inner sep=1.5pt] at (7, -3.05)
    {\small \cite{KVZ24INO, KVZ26INO}$^\dagger$\ $qd^2 \lambda$};
 
\draw[->, very thick, BrickRed] (Lq) -- (Rq)
    node[fill=white, inner sep=2pt, pos=0.2, above=2pt]
        {\small\bfseries\color{BrickRed} New:\ $qd^2 \bar\lambda$};
 
\draw[->, thick] (Lc) -- (R2)
    node[fill=white, inner sep=1.5pt, pos=0.3, below=2pt, sloped]
        {\small \cite{KV25faster}\ $d^2 \bar{\lambda}$};
 
\draw[->, thick, blue] ($(Linf.east)+(0.15,-0.45)$) -- ($(Rtv.west)+(-0.15,0.15)$)
    node[fill=white, inner sep=1.5pt, pos=0.57, below=2pt, sloped]
        {\small $\msf{Ball\ walk}^\dagger$\ $d^2 \lambda$};
 \end{tikzpicture}

\caption{\label{fig:complexity-summary}Zeroth-order query complexities for
logconcave sampling. Arrows indicate the sampling complexity from
an $O(1)$-warm start. Complexities listed next to each left node
indicate warmness-generation costs. The dagger sign ($\dagger$) indicates
results for \emph{uniform} sampling over convex bodies (membership
oracle). Results without $\dagger$ hold for general logconcave distributions
(evaluation oracle). Here, $\bar{R}:=R\vee1$, $\bar{\lambda}:=\lambda\vee1$,
$\lambda:=\protect\norm{\protect\cov\pi}$, and $R^{2}:=\mathbb{E}_{\pi}[|\cdot|^{2}]$
(so $\protect\lda\protect\leq R^{2}$). Note that $\protect\tv$-warmness
can be treated as $\protect\eu R_{\infty}$-warmness at the cost of
collapsing downstream guarantees to $\protect\tv$-distance.}
\end{figure}

\paragraph{$\protect\tv$-collapse.}

A central difficulty in obtaining $\eu R_{q}$-warmness through annealing
is that we do not yet know how to relay $\eu R_{q}$-warmness without
compromise, leading nearly all previous work to encounter this $\tv$-collapse
problem. To illustrate this, consider the annealing distributions
$\mu_{0},\mu_{1},\mu_{2}$, where a sampler transitions from $\mu_{0}$
toward $\mu_{1}$. Markov-chain samplers for convex bodies are inherently
\emph{approximate}, producing a measure $\bar{\mu}_{1}$ close to
but not exactly $\mu_{1}$, quantified via $\eu R_{q}(\bar{\mu}_{1}\mmid\mu_{1})\leq\veps$
for some $q>1$. In the next phase, the sampler moves toward $\mu_{2}$
starting from the approximate measure $\bar{\mu}_{1}$, not $\mu_{1}$.
Ideally, the complexity for this phase would directly depend on $\eu R_{q}(\bar{\mu}_{1}\mmid\mu_{2})$.
Unfortunately, the lack of a triangle inequality for $\eu R_{q}$-divergence
prevents us from ensuring $\eu R_{q}(\bar{\mu}_{1}\mmid\mu_{2})=O(1)$,
even if $\eu R_{q}(\mu_{1}\mmid\mu_{2})=O(1)$.

To circumvent this issue, prior studies have resorted to using the
triangle inequality (or a coupling argument) for $\tv$-distance.
Specifically, although $\bar{\mu}_{1}$ differs from $\mu_{1}$, annealing
algorithms proceeds as if the initial distribution were $\mu_{1}$,
not $\bar{\mu}_{1}$, when running a sampler toward $\mu_{2}$. Denoting
the transition kernel of this sampler by $P$ and the output distribution
of this sampler by $\bar{\mu}_{2}=\bar{\mu}_{1}P$, we can deduce
that 
\[
\norm{\bar{\mu}_{2}-\mu_{2}}_{\tv}\leq\norm{\bar{\mu}_{2}-\mu_{1}P}_{\tv}+\norm{\mu_{1}P-\mu_{2}}_{\tv}\leq\norm{\bar{\mu}_{1}-\mu_{1}}_{\tv}+\norm{\mu_{1}P-\mu_{2}}_{\tv}\,,
\]
where the second inequality follows from the data-processing inequality
\eqref{eq:DPI}. Since $\eu R_{q}$-divergence implies $\tv$-distance,
and each term on the RHS is bounded by $\veps$, this leads to the
accumulation of error over $m$ phases: $\norm{\bar{\mu}_{m}-\mu_{m}}_{\tv}\leq m\veps$.
This necessitates accuracy adjustments like replacing $\veps$ with
$\veps/m$.

\paragraph{Boosting the R\'enyi divergence order.}

The above discussion highlights the fundamental bottleneck arising
from the absence of a triangle inequality for R\'enyi divergences.
One promising direction to overcome this is to leverage a weak triangle
inequality \eqref{eq:weak-triangle-Renyi}, namely, $\eu R_{q}(\bar{\mu}_{1}\mmid\mu_{2})\lesssim\eu R_{2q-1}(\bar{\mu}_{1}\mmid\mu_{1})+\eu R_{2q}(\mu_{1}\mmid\mu_{2})$.
As \cite{KV25faster} has already designed annealing schemes that
can ensure $\eu R_{2q}(\mu_{1}\mmid\mu_{2})=O(1)$, one might consider
starting with a $\eu R_{2q-1}$-guarantee when sampling from $\mu_{1}$.
However, this effectively amounts to requiring a $\eu R_{\Theta(2^{m}q)}$-guarantee
for $\mu_{1}$-sampling when there are $m$ phases. Since $\PS$ has
linear dependence on $q$, and prior annealing approaches require
$m\gtrsim d^{1/2}$, such a ``backtracking'' analysis would potentially
incur an $\Omega(2^{\sqrt{d}}q)$-multiplicative overhead in the final
complexity.

Given the failure of this na\"ive approach, a pertinent question
is whether one can ensure $\eu R_{2q-1}(\bar{\mu}_{1}\mmid\mu_{1})=O(1)$
given $\eu R_{q}(\bar{\mu}_{1}\mmid\mu_{1})<\veps$. Since $2q-1>q$
for $q>1$, this seemingly goes in the wrong direction --- opposite
the monotonic property of $\eu R_{q}$-divergence. This suggests an
algorithmic question: \emph{can we boost the R\'enyi-divergence order
algorithmically} (i.e., can we transform $\eu R_{q}(\mu\mmid\pi)=O(1)$
into a stronger bound $\eu R_{2q-1}(\mu'\mmid\pi)=O(1)$)? Solving
this question would enhance our theoretical toolkit, providing the
final missing piece towards a comprehensive approach for convex body
sampling and its applications.

\paragraph{Hypercontractivity of semigroups.}

Looking back at diffusion for inspiration, perhaps the most relevant
mathematical notion to tackle this question is Gross' hypercontractivity
of semigroups \cite{Gross1975logarithmic} (see Proposition~\ref{prop:Gross-hypercontractivity}):
for a semigroup $(P_{t})_{t\geq0}$ reversible with respect to an
invariant measure $\pi$, it is known that \eqref{eq:lsi} for $\pi$
is equivalent to the property that $\|P_{t}f\|_{L^{q(t)}(\pi)}\leq\|f\|_{L^{p}(\pi)}$
for any function $f\in L^{p}(\pi)$, where $q(t)$ is an increasing
function with $q(0)=p$. For instance, the continuous-time Langevin
diffusion satisfies this property, but it is not immediately clear
whether such hypercontractivity is preserved upon discretization.
Recent advancements, such as \cite{CEL24analysis}, showed that discretized
Langevin processes retain hypercontractivity under LSI when the target
$\pi$ is unconstrained and log-smooth.

Can we extend hypercontractivity algorithmically beyond such specialized
settings? In particular, sampling with hard constraints---such as
sampling from a convex body---or more general smoothness assumptions
(e.g., $B_{1}\subseteq\mathcal{K}$) have resisted straightforward
solutions. One might consider using reflected Langevin diffusions,
but discretizations like the projected Langevin algorithm \cite{BEL18sampling}
require access to a projection oracle for $\mathcal{K}$, and existing
analytical guarantees similarly collapse to $\tv$-distance (beside
incurring much higher complexity).

\paragraph{Hypercontractivity of the heat adjoint (\S\ref{sec:hypercontractivity}).}

To address this boosting question, we explore an alternative strategy
involving a different Markov semigroup that can be algorithmically
realized via $\PS$. Specifically, we aim to establish hypercontractivity
for the $L^{2}(\nu)$-adjoint $Q_{t}$ of the heat semigroup $P_{t}$,
under \eqref{eq:lsi} for the target measure $\nu$. Formally, we
state the following hypercontractivity result for $Q_{t}$.
\begin{thm}
[Hypercontractivity under heat flow]\label{thm:hypercontractivity-heat-flow}
Let $\nu$ be a probability measure satisfying \eqref{eq:lsi} with
$\clsi(\nu)<\infty$. Let $\nu_{t}=\nu*\gamma_{t}$ for $t\geq0$,
and define $q:[0,\infty)\to(1,\infty)$ by
\[
\frac{q(t)-1}{p-1}=1+\frac{t}{\clsi(\nu)}\quad\text{for any }p\in(1,\infty)\,.
\]
Then, the $L^{2}(\nu)$-adjoint $Q_{t}$ of the heat semigroup satisfies
\[
\norm{Q_{t}f}_{L^{q(t)}(\nu_{t})}\leq\norm f_{L^{p}(\nu)}\quad\text{for all }p\in(1,\infty),t\geq0,\text{ and }f\in L^{p}(\nu)\,.
\]
\end{thm}

\subparagraph{(1) Proof sketch.}

Overall, we follow a proof sketch in \cite[Theorem 5.2.3]{BGL14analysis}
along with technical tweaks. To begin, for any compactly supported,
smooth, positive function $f$, we define the probability measure
$\D\mu\propto f\,\D\pi$. Denoting the convolved measures $\mu_{t}:=\mu*\gamma_{t}$
and $\pi_{t}:=\pi*\gamma_{t}$, we introduce the functionals $\mc F(t,q):=\int(\nicefrac{\D\mu_{t}}{\D\pi_{t}})^{q}\,\D\pi_{t}$
and $\mc G(t):=\frac{1}{q(t)}\,\log\mc F(t,q(t))$. Our goal is to
construct a suitable function $q(t)$ such that $\de_{t}\mathcal{G}\leq0$.
Our proof proceeds by differentiating $\mc G$ in $t$, which involves
computing $\de_{t}\mc F$ and $\de_{q}\mc F$. The time derivative
$\de_{t}\mc F$ follows from the general de Bruijn identity \cite{KO25strong},
while the exponent derivative $\de_{q}\mc F$ gives rise to $\ent_{\pi_{t}}$,
where the log-Sobolev constant $\clsi(\pi_{t})$ comes into picture.

Combining these ingredients, we derive a differential inequality for
$\mathcal{G}(t)$, and $\de_{t}\mc G\leq0$ holds provided that $\clsi(\pi_{t})\leq(q(t)-1)/(q(t)-1)'$.
Leveraging $\clsi(\pi_{t})\leq\clsi(\pi)+t$ \cite{chafai04entropies},
we solve $\clsi(\pi)+t\leq(q(t)-1)/(q(t)-1)'$, which holds when $q(t)-1=(1+\nicefrac{t}{\clsi(\pi)})\,(p-1)$.
Altogether, these arguments establish the hypercontractivity of the
heat adjoint $Q_{t}$. Lastly, we extend this result beyond the space
of compactly supported, positive, smooth functions by employing standard
approximation techniques in analysis, such as the dominated convergence
theorem and Fatou's lemma, thereby generalizing the statement to $L^{p}(\pi)$-space.\vspace{-9pt}

\subparagraph{(2) Consequences.}

An immediate consequence follows by selecting $f=\frac{\D\mu}{\D\nu}$
for a probability measure $\mu\ll\nu$. Since $Q_{t}f=\frac{P_{t}(f\nu)}{P_{t}\nu}$,
this yields $\norm{\D\mu_{t}/\D\nu_{t}}_{L^{q(t)}(\nu_{t})}\leq\norm{\D\mu/\D\nu}_{L^{p}(\nu)}$,
which implies $\eu R_{q(t)}(\mu_{t}\mmid\nu_{t})\leq\eu R_{p}(\mu\mmid\nu).$
This result has algorithmic implications for $\PS$. Below, the constant
$2$ can be replaced by any number larger than $1$.
\begin{cor}
[Hypercontractivity of proximal sampler]\label{cor:PS-mixing-LSI}
Let $\pi_{k}^{X}$ be the law of the $k$-th iterate returned by the
$\ps$ with step size $h>0$, initial $\pi_{0}^{X}$, and target $\pi^{X}$.
Given $1<p\leq q$, we have $\eu R_{q}(\pi_{N}^{X}\mmid\pi^{X})\leq\eu R_{p}(\pi_{0}^{X}\mmid\pi^{X})$
after $N\lesssim h^{-1}\clsi(\pi^{X})\log_{2}\frac{q-1}{p-1}$ iterations.
Moreover, given $\veps>0$ and $q\geq q_{0}>1$, we have $\eu R_{q}(\pi_{N}^{X}\mmid\pi^{X})\leq\veps$
after
\[
N\lesssim h^{-1}\clsi(\pi^{X})\log\frac{q\eu R_{2\wedge q_{0}}(\pi_{0}^{X}\mmid\pi^{X})}{(2\wedge q_{0}-1)\,\veps}\quad\text{iterations\,.}
\]
\end{cor}

This result translates hypercontractivity into an algorithmic outcome,
enabling efficient implementation of warm-start generation. Moreover,
it leads to $\Otilde(\kappa d\log\frac{q}{\veps})$-mixing of $\PS$
in the well-conditioned setting (Proposition~\ref{prop:mixing-PS-smooth}),
improving the linear dependence of $\PS$ on $q$ \cite{CCSW22improved}.

\paragraph{Streamlined warm-start generation (\S\ref{sec:Unif-warmstart}).}

We simplify the design and analysis of a known Gaussian-annealing
scheme for warm-start generation, incorporating the hypercontractivity
established for $\PS$. This streamlined approach offers an elegant
and clean picture for warm-start generation via annealing schemes.

As with $\psunif$, we propose and analyze its Gaussian variant in
\S\ref{subsec:PS-gauss}, denoted by $\psgauss$, specifically tailored
for truncated Gaussian distributions $\gamma_{\sigma^{2}}|_{\K}$.
Leveraging the hypercontractivity result for $\PS$ and the bound
$\clsi(\gamma_{\sigma^{2}}|_{\K})\leq\sigma^{2}$, we derive an enhanced
guarantee for $\psgauss$. Below, $B_{1}(0)$ denotes the $\ell_{2}$
unit ball centered at the origin.
\begin{thm}
\label{thm:gauss-warm} Let $\K\subset\Rd$ be a convex body with
$B_{1}(0)\subset\K$ given by a membership oracle, and let $\mu\propto\gamma_{\sigma^{2}}\cdot\ind_{\K}$
be the Gaussian truncated to $\K$. Given $\veps\in(0,1/100)$, initial
distribution $\mu_{0}$, any $q\geq2\vee\Omtilde(\log(d^{2}\sigma^{2}\log\frac{1}{\veps}))$
such that $q\geq q_{0}:=2\vee\Otilde(\log(d^{2}\sigma^{2}\log\frac{q}{\veps}\log\frac{1}{\veps}))$
with $M_{q_{0}}=\norm{\frac{\D\mu_{0}}{\D\mu}}_{L^{q_{0}}(\mu)}\leq10$,
consider the algorithm that runs $\psgauss$ initialized at $\mu_{0}$
and restarts if $\psgauss$ fails before completing $N=\Otilde(d^{2}\sigma^{2}\log\frac{q}{\veps})$
iterations. There exist choices of $h$ and $\tau$ such that this
algorithm restarts with probability at most $1/50$, and its output
$Z$ satisfies $\eu R_{q}(\law Z\mmid\mu)\leq\veps$, using $\Otilde(d^{2}\sigma^{2}\log\frac{q}{\veps}\log^{5}\frac{1}{\veps})$
\textup{membership queries in expectation}.
\end{thm}

Equipped with $\psgauss$, our warm-start generation process is very
simple: (1) sample from $\mu_{1}\propto\pi_{\K}\gamma_{d^{-1}}$ by
rejection sampling with the Gaussian proposal $\gamma_{d^{-1}}$,
and (2) while $\sigma_{i}^{2}\lesssim qR\Lambda^{1/2}\log^{1/2}d$,
move from $\bar{\mu}_{i}$ toward $\mu_{i+1}\propto\pi_{\K}\gamma_{\sigma_{i+1}^{2}}$
using $\psgauss$, where $\eu R_{\Otilde(1)}(\bar{\mu}_{i}\mmid\mu_{i})=O(1)$
and $\sigma_{i+1}^{2}=\sigma_{i}^{2}\,(1+\nicefrac{\sigma_{i}}{R})$.

This simpler algorithm eliminates pre-processing steps, such as truncation
of the convex body, and post-processing, such as adjusting the target
accuracy. Moreover, a key practical advantage of this streamlined
method is the introduction of early-stopping at $\sigma^{2}\approx qR\Lambda^{1/2}\log^{1/2}d$.
Consequently, faster sampling phases beyond $\sigma^{2}\gtrsim R\Lambda^{1/2}\log^{1/2}d$
needed in \cite{KV25faster} become redundant. On the theoretical
side, hypercontractivity enables the propagation of $\eu R_{q}$-warmness
throughout the annealing scheme, leading to the following comprehensive
result for warm-start generation:
\begin{thm}
\label{thm:complexity-Renyi-warm} Let $\pi$ be the uniform distribution
over a convex body $\K\subset\Rd$ given by a membership oracle with
$B_{1}(0)\subset\K$, $\Lambda=\norm{\cov\pi}$, and $R^{2}=\E_{\pi}[\abs{\cdot}^{2}]$,
. For any $\veps_{\unif}\in(0,1/100)$ and $q\geq2\vee\Omtilde\bpar{\log\frac{dR}{\veps_{\unif}}}$,
there exists an algorithm that outputs a sample $X^{*}$ such that
$\eu R_{q}(\law X^{*}\mmid\pi)\leq2$, using $\Otilde(qd^{2}R^{3/2}\Lambda^{1/4})$
membership queries in expectation.
\end{thm}

Combining this with Theorem~\ref{thm:unif-warm}, we obtain the query
complexity of convex-body sampling from scratch, with $\eu R_{q}$-divergence
output guarantees for any $q\geq1$.
\begin{cor}
\label{cor:unif-comp-scratch} Let $\pi$ be the uniform distribution
over a convex body $\K\subset\Rd$ given by a membership oracle with
$B_{1}(0)\subset\K$, $\Lambda=\norm{\cov\pi}$, and $R^{2}=\E_{\pi}[\abs{\cdot}^{2}]$.
For any $\veps\in(0,1/100)$ and $q\geq1$, there exists an algorithm
that outputs a sample $X^{*}$ such that $\eu R_{q}(\law X^{*}\mmid\pi)\leq\veps$,
using $\Otilde(qd^{2}R^{3/2}\Lambda^{1/4}\log\frac{1}{\veps}+qd^{2}\Lambda\log^{7}\frac{1}{\veps})$
membership queries in expectation.
\end{cor}

\paragraph{Streamlined analysis of Gaussian annealing.}

In our analysis of Gaussian annealing (\S\ref{subsec:FI-to-closeness}),
we present a principled and streamlined approach to bounding the closeness
between successive annealing distributions. Whereas prior works relied
on localization arguments (with long calculations) to establish such
closeness, we show that this property follows more cleanly from LSI,
combined with a mollification argument that allows for a limiting
process.

To formalize this idea, we recall a R\'enyi version of LSI established
in \cite{VW23rapid}, which states that $\eu R_{q}(\mu\mmid\pi)\leq\nicefrac{q\clsi(\pi)}{2}\,\msf{RFI}_{q}(\mu\mmid\pi)$
for any $q>1$ and smooth $\mu/\pi$, where $\msf{RFI}_{q}(\mu\mmid\pi)=q\,\E_{\pi}[(\frac{\D\mu}{\D\pi})^{q}\abs{\nabla\log\frac{\D\mu}{\D\pi}}^{2}]/\E_{\pi}[(\frac{\D\mu}{\D\pi})^{q}]$.
When bounding $\eu R_{q}(\gamma_{\sigma^{2}}|_{\K}\mmid\gamma_{\sigma^{2}(1+\alpha)}|_{\K})$
for a convex body $\K$ and some $\alpha>0$, a na\"ive calculation---ignoring
the truncation to $\K$---suggests that one can take $\alpha\asymp\nicefrac{\sigma}{qR}$
to make it $\O(1)$. This heuristic choice recovers the scaling behavior
established in earlier results.

To rigorously account for the effects of truncation to $\K$, we employ
a mollification argument using a standard mollifier $\eta$ supported
on the unit ball (see \eqref{eq:mollifier}). For $\varepsilon>0$,
we define $\eta_{\varepsilon}(x):=\varepsilon^{-d}\,\eta(x/\varepsilon)$,
and convolve both truncated Gaussians with $\eta_{\varepsilon}$.
Since convolution with $\eta_{\varepsilon}$ yields a smooth approximation
$f_{\varepsilon}:=f*\eta_{\varepsilon}$ that converges to $f$ almost
everywhere as $\varepsilon\to0$ for any locally integrable function
$f$, this mollification step allows us to apply the R\'enyi-LSI
to the regularized densities and pass to the limit, thereby recovering
the desired closeness bound.

An additional benefit of this approach is that it enables a clean
justification for early-stopping in the annealing schedule. Unlike
known annealing algorithms, we can terminate annealing when the variance
$\sigma^{2}$ reaches approximately $qR\Lambda^{1/2}$, at which point
$\eu R_{q}(\gamma_{\sigma^{2}}|_{\K}\mmid\pi_{\K})=\O(1)$. This follows
directly from a R\'enyi version of PI (see \eqref{eq:Renyi-PI}),
combined with the mollification argument described above. Overall,
this approach simplifies the analysis and unifies the reasoning behind
closeness and early-stopping in Gaussian annealing.

\subsubsection{Logconcave sampling under an evaluation oracle}

Our uniform-sampling results extend naturally to general logconcave
distributions $\pi\propto e^{-V}$, given access to an \emph{evaluation
oracle} for a convex function $V:\Rd\to\R\cup\{\infty\}$. In this
setting, the regularity condition $B_{1}\subset\K$ translates to
a condition on a level set of $V$, namely, that the \emph{ground
set} \emph{$\msf L_{\pi,g}:=\{x\in\Rd:V(x)-\min V\leq10d\}$} contains
a unit ball. This generalization is justified by the fact that $\msf L_{\pi,g}$
takes up a constant fraction of the $\pi$-measure (see \cite[Lemma 2.6]{KV25sampling}).
We emphasize that any logconcave distribution satisfies this condition
after rescaling.

The logconcave sampling problem has been studied as an algorithmic
problem since Applegate and Kannan~\cite{AK91sampling}, who used
the $\msf{Grid}\text{ }\msf{walk}:\eu R_{\infty}\to\tv$, establishing
polynomial complexity in the dimension under the assumption of Lipschitzness
of the function over the support. Later, Lov\'asz and Vempala \cite{LV07geometry}
generalized classical algorithms such as the $\bw$ and $\har$ to
this setting by incorporating a Metropolis filter. They showed that
after a suitable pre-processing of $\pi$, $\bw,\har:\eu R_{\infty}\to\tv$
with query complexity $\Otilde(\varepsilon^{-4}d^{2}R^{2}M_{\infty}^{4}\log\nicefrac{M_{\infty}}{\veps})$.
Furthermore, \cite{LV06fast} showed that the complexity of $\har:\eu R_{2}\to\eu R_{2}$
can be improved to $\Otilde(d^{2}R^{2}\polylog\nicefrac{M_{2}}{\veps})$.

Recent work \cite{KV25sampling,KV25faster} revisited this problem
by adapting $\PS$ to the general logconcave setting. Although $\PS$
can in principle be extended to any distribution where $\pi^{X|Y=y}\propto\pi^{X}(\cdot)\,\gamma_{h}(\cdot-y)$
can be sampled (e.g., when given access to a proximal oracle in the
log-smooth setting), it is unclear how to implement this step for
arbitrary $\pi^{X}$. To address this, \cite{KV25sampling} introduced
the exponential lifting technique: for $\pi^{X}\propto e^{-V}$, consider
the lifted distribution 
\begin{equation}
\D\pi^{X,T}(x,t)\propto e^{-dt}\,\ind_{\K}(x,t)\,\D x\D t\,,\quad\text{where }\K:=\{(x,t)\in\Rd\times\R:V(x)\leq dt\}\,.\tag{\ensuremath{\msf{exp\text{-}lifting}}}\label{eq:exp-lifting}
\end{equation}
Here, the $X$-marginal of $\pi^{X,T}$ is exactly $\pi^{X}$, and
$\K$ is convex whenever $\pi^{X}$ is logconcave. Thus, logconcave
sampling reduces to the simpler problem of sampling from a logconcave
distribution with a linear potential over a convex set. We then apply
$\PS$ to this exponential distribution, denoted by $\psexp$: for
$z=(x,t)\in\R^{d+1}$ and $y\in\R^{d+1}$, it alternates $y\sim\pi^{Y|Z=z}=\mc N(z,hI_{d+1})$
and $z\sim\pi^{Z|Y=y}\propto\mc N(y-h\alpha,hI_{d+1})|_{\K}$ for
$\alpha=de_{d+1}$, where the second step is implemented via rejection
sampling, using the proposal $\mc N(y-h\alpha,hI_{d+1})$ with threshold
$\tau$. 

In \cite{KV25sampling}, it was shown that $\psexp:\eu R_{\infty}\to\eu R_{q}$
with query complexity $\Otilde(qM_{\infty}d^{2}(\Lambda\vee1)\polylog\nicefrac{1}{\eta\veps})$
and $\psexp:\eu R_{\infty}\to\eu R_{\infty}$ with complexity $\Otilde(M_{\infty}d^{2}(d\vee R^{2})\polylog\nicefrac{1}{\eta\veps})$.
As in the uniform sampling case, the order of initial warmness was
relaxed to $\eu R_{c}$ with $c=\Otilde(1)$ in \cite{KV25faster},
where $\psexp:\eu R_{c}\to\eu R_{2}$ achieves query complexity $\Otilde(M_{c}d^{2}(\Lambda\vee1)\polylog\nicefrac{1}{\eta\veps})$.
However, it remained open to balance the divergence order of initial
warmness and final output distribution guarantee.

\paragraph{Extension to logconcave sampling (\S\ref{sec:complexity-LC}).}

Our $\psexp:\eu R_{q}\to\eu R_{q}$ with restart aligns the divergence
order of initial warmness with that of the final error guarantee,
from an $O(1)$-warm start, at an additive cost of $qd^{2}$ queries.
This refines a long line of work on sampling in the classical setting.
See \S\ref{subsec:LC-sampling-from-warmstart} for further details.
\begin{thm}
[Zeroth-order logconcave sampling from a warm start]\label{thm:lc-warm}
For a convex function $V:\Rd\to\R$, given by an evaluation oracle,
let $\pi\propto e^{-V}$ be the logconcave distribution over $\Rd$
with $B_{1}\subset\msf L_{\pi,g}$ and $\Lambda=\norm{\cov\pi}$.
Given $\veps>0$, an initial distribution $\pi_{0}$, and $q\geq2\vee\Omtilde(\log\{d^{2}\,(\Lambda\vee1)\log\frac{1}{\veps}\})$
such that $M_{q}=\norm{\frac{\D\pi_{0}}{\D\pi}}_{L^{q}(\pi)}\leq10$,
consider the algorithm that runs $\psexp$ initialized at $\pi_{0}$
and restarts if $\psexp$ fails before completing $N=\Otilde(qd^{2}\,(\Lambda\vee1)\log^{2}\frac{1}{\veps})$
iterations. There exists choices of $h$ and $\tau$ such that this
algorithm restarts with probability at most $1/50$, and its output
$Z$ satisfies $\eu R_{q}(\law Z\mmid\pi)\leq\veps$, using $\Otilde(qd^{2}\,(\Lambda\vee1)\log^{3}\frac{1}{\veps})$
\textup{evaluation queries in expectation}.
\end{thm}

We remark on a connection to the well-conditioned setting. In this
case, the ground set contains a ball of radius $\Theta((d/\beta)^{1/2})$
centered at $\arg\min V$. Since $\Lambda\leq\alpha^{-1}$ by the
Brascamp--Lieb or Lichnerowicz inequality, $\psexp$ uses $\Otilde(q\,(\kappa d+d^{2})\log\nicefrac{1}{\veps})$
queries from an $\O(1)$-warm start in $\eu R_{q}$. With the additional
$qd^{2}$ queries, this matches the known $\Otilde(q\kappa d\log\nicefrac{1}{\veps})$-complexity
results for zeroth-order samplers in this setting which additionally
use a proximal oracle. See \S\ref{subsec:Related-work} for further
details.

Annealing approaches for uniform sampling have also been extended
to arbitrary logconcave distributions~\cite{LV06fast,KV25sampling,KV25faster}.
In the general case, the best-known complexity for warm-start generation
is comparable to that of the convex-body case for $\tv$ and $\eu R_{\infty}$-warmness,
with algorithms requiring additional $d^{2}\polylog\nicefrac{1}{\varepsilon}$
queries. In \S\ref{subsec:LC-warmness-generation}, we extend our
refined warm-start generation to general logconcave distributions.
\begin{thm}
\label{thm:complexity-Renyi-warm-LC} For a convex function $V:\Rd\to\R$,
given by an evaluation oracle, let $\pi\propto e^{-V}$ be the logconcave
distribution over $\Rd$ with $B_{1}(0)\subset\msf L_{\pi,g}$, $\Lambda=\norm{\cov\pi}$,
and $R^{2}=\E_{\pi}[\abs{\cdot}^{2}]$. For any $\veps_{\exp}\in(0,\frac{1}{10})$
and $q\geq2\vee\Omtilde(\log\frac{dD}{\veps_{\exp}})$, there exists
an annealing algorithm that outputs a sample $X^{*}$ such that $\eu R_{q}(\law X^{*}\mmid\pi)\leq2$,
using $\Otilde(d^{2.5}+qd^{2}\,(R^{3/2}\vee1)(\Lambda^{1/4}\vee1))$
evaluation queries in expectation.
\end{thm}

Combining the two theorems above, we establish the query complexity
of sampling from an arbitrary logconcave distribution in $\Rd$ with
$\eu R_{q}$-divergence guarantees, assuming only an evaluation oracle.
\begin{cor}
\label{cor:general-LC-comp} For a convex function $V:\Rd\to\R$,
given by an evaluation oracle, let $\pi\propto e^{-V}$ be the logconcave
distribution over $\Rd$ with $B_{1}(0)\subset\msf L_{\pi,g}$, $\Lambda=\norm{\cov\pi}$,
and $R^{2}=\E_{\pi}[\abs{\cdot}^{2}]$. For any $\veps\in(0,\frac{1}{10})$
and $q\geq1$, there exists an algorithm that outputs a sample $X^{*}$
such that $\eu R_{q}(\law X^{*}\mmid\pi)\leq\veps$, using $\Otilde(d^{2.5}+qd^{2}R^{3/2}\,(\Lambda^{1/4}\vee1)\log\frac{1}{\veps}+qd^{2}\,(\Lambda\vee1)\log^{4}\frac{1}{\veps})$
evaluation queries in expectation.
\end{cor}

\subsubsection{Obstruction to further acceleration}

A natural question is whether the annealing complexity in Corollary~\ref{cor:general-LC-comp}
can be improved beyond $\Otilde(qd^{2}R^{3/2}\Lambda^{1/4})$. For
isotropic logconcave distributions in particular, given the $d^{3}$
query complexity of \cite{KV25sampling} for $\eu R_{\infty}$-guarantees
and our $qd^{2.75}$ complexity for $\eu R_{q}$-guarantees, it is
natural to ask whether the $d^{2.5}$ \emph{iteration} complexity
of the $\sw$ established in \cite{LV24eldan} can be matched as a
query complexity, at least for $\tv$-guarantees (note that the speedy
walk is an abstraction that, in general, does not have an efficient
implementation).

Towards this goal, \cite{KV25faster} conjectured that for the uniform
distribution over any isotropic convex body $\pi$ in $\Rd$, the
Gaussian-tilted distribution $\pi\gamma_{t}$ satisfies $\norm{\cov\pi\gamma_{t}}=\Otilde(1)$
for all $t>0$, a property called $\Otilde(1)$-\emph{quadratic-tilt
stability}. Under this conjecture, one can afford more aggressive
annealing steps of size $\sigma_{i+1}^{2}=\sigma_{i}^{2}(1+\sigma_{i}^{2}/R)$
(instead of $\sigma_{i+1}^{2}=\sigma_{i}^{2}(1+\sigma_{i}/R)$), leading
to an improved warm-start generation complexity of $\Otilde(d^{2.5})$
for isotropic logconcave distributions (see \S\ref{subsec:Acceleration-conj}).
We verify in \S\ref{sec:Quadratic-tilt-stability} that structured
bodies---including the hypercube, the simplex, and any isotropic
convex body of revolution---satisfy $O(1)$-quadratic-tilt stability,
confirming the desired $\Otilde(d^{2.5})$ query complexity of sampling
from these structured distributions.

However, in \S\ref{subsec:counterexample} we disprove the conjecture
in general, establishing a fundamental obstruction to this line of
acceleration. Specifically, adapting a construction of Bizeul \cite{Bizeul26logsobolev},
we show that there exists an isotropic convex body whose Gaussian-tilted
distribution exhibits a peak variance of $d^{1/3}$ (Proposition~\ref{prop:counter-example}),
disproving the conjecture and proving a lower bound for the Gaussian
annealing approach that has been central to this line of work since
\cite{CV14cubic,CV15bypass,CV18Gaussian}.

\subsection{Preliminaries\label{subsec:Preliminaries}}

\paragraph{Notation.}

For $t>0$, we reserve $\gamma_{t}$ for the centered Gaussian distribution
with covariance matrix $tI_{d}$. The indicator function of a set
$S\subseteq\Rd$ is denoted by $\ind_{S}(x):=[x\in S]$, and $\mu|_{S}$
refers to a distribution $\mu$ truncated to $S$ (i.e., $\mu|_{S}\propto\mu\cdot\ind_{S}$).
For two probability measures $\mu,\pi$, we use $\mu\pi$ to denote
the new distribution with density proportional to $\mu\pi$.

Both $a\lesssim b$ and $a=\O(b)$ mean $a\le cb$ for a universal
constant $c>0$. $a=\Omega(b)$ means $a\gtrsim b$, and $a\asymp b$
means $a=\O(b)$ and $a=\Omega(b)$. Lastly, $a=\Otilde(b)$ means
$a=O(b\polylog b)$. For a positive semi-definite matrix $\Sigma$,
$\norm{\Sigma}$ denote the operator norm of $\Sigma$.

\paragraph{Log-concavity.}

We call a function $f:\Rd\to[0,\infty)$ \emph{logconcave} if $-\log f$
is convex in $\Rd$, and a probability measure $\pi$ (or distribution)
logconcave if it has a logconcave density function with respect to
the Lebesgue measure. We abuse notation by using the same symbol for
a distribution and density. We assume that any logconcave distributions
considered in this work are non-degenerate; otherwise, we could simply
work on an affine subspace on which degenerate distributions are supported.
For $t\geq0$, $\pi$ is called \emph{$t$-strongly logconcave} if
$-\log\pi$ is $t$-strongly convex (i.e., $-\log\pi-\frac{t}{2}\,\abs{\cdot}^{2}$
is convex). Clearly, log-concavity is preserved under multiplication,
and a classical result by Pr\'ekopa and Leindler ensures that the
convolution also preserves log-concavity. A distribution is called
\emph{isotropic} if its barycenter is at the origin and has the identity
covariance matrix. Note that as the logconcave distributions decay
exponentially fast at infinity, they have finite moments of all orders.

\paragraph{Probability divergences and distances.}

Let $\mu$ and $\nu$ denote probability measures over $\Rd$, and
$\vphi:\R_{+}\to\R$ be a convex function with $\vphi(1)=0$. The
\emph{$\vphi$-divergence} of $\mu$ toward $\nu$ with $\mu\ll\nu$
is defined as
\[
D_{\vphi}(\mu\mmid\nu):=\int\vphi\bpar{\frac{\D\mu}{\D\nu}}\,\D\nu\,.
\]
The \emph{total variation} ($\tv$) distance, the \emph{$\KL$-divergence,
and $\chi^{q}$-divergence} for $q\in(1,\infty)$ can be recovered
through $\vphi(x)=\half\,\abs{x-1}$, $x\log x$, and $x^{q}-1$,
respectively. The \emph{$q$-R\'enyi divergence} is defined as
\begin{equation}
\eu R_{q}(\mu\mmid\nu):=\frac{1}{q-1}\,\log\bpar{1+\chi^{q}(\mu\mmid\nu)}=\frac{1}{q-1}\,\log\,\Bnorm{\frac{\D\mu}{\D\nu}}_{L^{q}(\nu)}^{q}\,.\label{eq:div-equivalence}
\end{equation}
The \emph{R\'enyi-infinity divergence} is defined as $\eu R_{\infty}(\mu\mmid\nu):=\log\esssup_{\nu}\frac{\D\mu}{\D\nu}$.
A weak triangle inequality holds for R\'enyi divergence: for any
$q>1,\lda\in(0,1)$, and probability measures $\mu,\nu,\pi$, it holds
that
\begin{equation}
\eu R_{q}(\mu\mmid\pi)\leq\frac{q-\lda}{q-1}\,\eu R_{\frac{q}{\lda}}(\mu\mmid\nu)+\eu R_{\frac{q-\lda}{1-\lda}}(\nu\mmid\pi)\,.\label{eq:weak-triangle-Renyi}
\end{equation}
For $\lda=q/(2q-1)$, we have $\eu R_{q}(\mu\mmid\pi)\leq\frac{2q}{2q-1}\,\eu R_{2q-1}(\mu\mmid\nu)+\eu R_{2q}(\nu\mmid\pi)$.
We refer readers to \cite{vH14renyi,Mironov17renyi} for more properties
of the R\'enyi divergence.

\section{Uniform sampling from convex bodies with balanced R\'enyi divergence\label{sec:ps-unif}}

We begin by describing our approach for the general proximal sampler
($\PS$) and then focus specifically on $\psunif$ for uniform sampling
in \S\ref{subsec:uniform-query-complexity}, which was proposed and
analyzed in \cite{KVZ24INO}. $\PS$ with step size $h$ alternates
between two steps: starting with $k=0$ and $x_{0}\sim\pi_{0}^{X}$,
(i) sample $y_{k+1}\sim\pi^{Y|X=x_{k}}$ and (ii) sample $x_{k+1}\sim\pi^{X|Y=y_{k+1}}$.
Precisely, let $P_{h}$ and $Q_{h}$ be the transition kernels of
the forward and backward steps, respectively, defined as 
\[
P_{h}(x,\D y)=\pi^{Y|X=x}=\mc N(x,hI_{d})\,,\qquad Q_{h}(y,\D x)=\pi^{X|Y=y}\propto\pi^{X}(\cdot)\,\gamma_{h}(\cdot-y)\,.
\]
Assume that there exists a suitable proposal distribution $\mu_{y}$
such that $(\mu_{y}\mid\text{accepted})=Q_{h}(y,\cdot)$. For uniform
sampling, we take $\mu_{y}=\mc N(y,hI_{d})$ and accept the proposal
if it lies inside $\K$. Given this setup, we implement the backward
step via rejection sampling with proposal $\mu_{y_{k+1}}$. If the
rejection loop exceeds $\tau$ trials, we declare failure and restart
from scratch with a new sample from $\pi_{0}^{X}$.

Kook and Vempala \cite{KV25faster} obtained the query complexity
of $\PS$ for obtaining a uniform sample whose law is $\veps$-close
to the target $\pi^{X}$ in $\eu R_{2}$-divergence, assuming the
warmness is given in $\eu R_{q}$ with $q\geq\Otilde(1)$. We refine
their analysis to obtain the query complexity of $\psunif:\eu R_{q}\to\eu R_{q}$
given $\norm{\D\pi_{0}^{X}/\D\pi^{X}}_{L^{q}(\pi^{X})}\leq10$ (i.e.,
an $\O(1)$-warm start in $\eu R_{q}$). To this end, we decouple
the analysis into two components: (1) a mixing analysis and (2) the
query complexity of the backward step. The first part determines how
many iterations $N$ are required for the sample $X_{N}$ to be $\veps$-close
to $\pi^{X}$. The second determines how many queries $\PS$ uses
in each iteration. Therefore, the total query complexity of the algorithm
is the product of these two components.

\subsection{Mixing analysis\label{subsec:mixing}}

The data-processing inequality (DPI) is a fundamental tool in information
theory. For any probability measures $\mu,\nu$, Markov kernel $P$,
$\vphi$-divergence $D_{\vphi}$, and $q\in(1,\infty]$, it holds
that
\begin{equation}
D_{\vphi}(\mu P\mmid\nu P)\leq D_{\vphi}(\mu\mmid\nu)\,,\qquad\quad\text{and}\qquad\quad\eu R_{q}(\mu P\mmid\nu P)\leq\eu R_{q}(\mu\mmid\nu)\,,\label{eq:DPI}
\end{equation}
where $\mu P:=\int P(x,\cdot)\,\mu(\D x)$ denotes the distribution
obtained by applying one step of $P$ starting from $\mu$. This implies
that for $D\in\{D_{\vphi},\eu R_{q}\}$ and any $k\geq0$,
\[
D(\pi_{k+1}^{X}\mmid\pi^{X})=D(\pi_{k+1}^{Y}Q_{h}\mmid\pi^{Y}Q_{h})\leq D(\pi_{k+1}^{Y}\mmid\pi^{Y})=D(\pi_{k}^{X}P_{h}\mmid\pi^{X}P_{h})\leq D(\pi_{k}^{X}\mmid\pi^{X})\,.
\]
However, this alone is not sufficient to establish a mixing rate.
It turns out that adding an independent Gaussian random variable to
$X_{k}$ (i.e., convolving $\law X_{k}$ with $\gamma_{h}$) causes
$\law X_{k}$ to contract directly toward $\pi^{X}*\gamma_{h}=\pi^{Y}$,
with the contraction rate governed by $\cpi(\pi^{X})$ or $\clsi(\pi^{X})$.

This phenomenon, known as the \emph{strong data-processing inequality}
(SDPI) for the Gaussian channel \cite{AG76spreading}, has been extensively
studied in information theory \cite{PW16dissipation,CPW18strong}.
Given a $\vphi$-divergence $D$, the contraction coefficient is defined
as 
\[
\eta_{D}(\mu,t):=\sup_{\nu:\,D(\nu\mmid\mu)\in\R_{>0}}\frac{D(\nu_{t}\mmid\mu_{t})}{D(\nu\mmid\mu)}\,,
\]
where $\nu_{t}:=\nu*\gamma_{t}$ denotes the probability measure obtained
by convolving $\nu$ with $\gamma_{t}$. Note that $0\leq\eta_{D}(\mu,t)\leq1$
by the DPI.

Bounding $\eta_{D}$ relies on the well-known \emph{de Bruijn identity},
$\de_{t}\KL(\nu_{t}\mmid\mu_{t})=-\half\,\FI(\nu_{t}\mmid\mu_{t})$,
where the relative Fisher information is defined as $\FI(\nu\mmid\mu):=\E_{\nu}[\abs{\nabla\log\frac{\D\nu}{\D\mu}}^{2}]$;
see rigorous proofs by Barron \cite{Barron84monotonic} and Klartag
and Ordentlich \cite{KO25strong} (for its generalizations). Other
derivations can be found in \cite{guo09relative,CCSW22improved},
under the assumption that the probability measures involved are sufficiently
smooth to justify differentiation under the integral sign and the
vanishing of boundary terms in integration by parts.

Combining this identity with, for example, \eqref{eq:lsi} for $\mu$
(i.e., $\KL(\nu_{t}\mmid\mu_{t})\leq\nicefrac{\clsi(\mu_{t})}{2}\,\FI(\nu_{t}\mmid\mu_{t})$)
and $\clsi(\mu_{t})\leq\clsi(\mu)+t$ \cite{chafai04entropies}, one
obtains a differential inequality in $\KL$; solving it yields 
\[
\eta_{\KL}(\mu,t)\leq\bpar{1+\frac{t}{\clsi(\mu)}}^{-1}\,.
\]
Under similar regularity assumptions, the SDPI for R\'enyi divergence
under \eqref{eq:pi}, \eqref{eq:lsi}, and Lata\l a--Oleszkiewicz
inequality was studied in \cite{CCSW22improved} and further extended
to constrained probability measures in \cite{KVZ24INO}. These results
were rigorously unified in \cite{KO25strong} under a $\vphi$-Sobolev
inequality \cite{chafai04entropies} (a generalization of \eqref{eq:pi}
and \eqref{eq:lsi}), with regularity assumptions explicitly identified.
We refer the reader to \cite[Regularity conditions]{KO25strong} for
a detailed account.

Specifically, \cite{CCSW22improved} showed that under the regularity
assumptions, for any $q\geq2$ and $\nu\ll\mu$, the following bounds
hold:
\begin{equation}
\eu R_{q}(\nu_{t}\mmid\mu_{t})\leq\begin{cases}
\eu R_{q}(\nu\mmid\mu)-\frac{\log(1+t/\cpi(\mu))}{q} & \text{if }\eu R_{q}(\nu\mmid\mu)\geq1\,,\\
\frac{\eu R_{q}(\nu\mmid\mu)}{(1+t/\cpi(\mu))^{1/q}} & \text{if }\eu R_{q}(\nu\mmid\mu)<1\,,
\end{cases}\label{eq:Reniy-contraction-PI}
\end{equation}
and $\chi^{2}(\nu_{t}\mmid\mu_{t})\leq(1+t/\cpi(\mu))^{-1}\,\chi^{2}(\nu\mmid\mu)$.
Applying this result to the forward step, along with the DPI for the
backward step, implies that $\PS$ achieves $\eu R_{q}(\pi_{N}^{X}\mmid\pi^{X})\leq\veps$
with
\[
N\lesssim qh^{-1}\cpi(\pi^{X})\log\frac{M_{q}}{\veps}\,,\quad\text{where }M_{q}=\Bnorm{\frac{\D\pi_{0}^{X}}{\D\pi^{X}}}_{L^{q}(\pi^{X})}\,.
\]

We provide a version of this result in terms of $\chi^{q}$-divergence
(generalizing the $\chi^{2}$-result of \cite{CCSW22improved}), as
working with $\chi^{q}$-divergence yields cleaner expressions and
a direct bound on the contraction coefficient $\eta_{\chi^{q}}$ under
\eqref{eq:pi}. Our proof follows the argument sketched above, combining
the general de Bruijn identity in terms of $\chi^{q}$-divergence
with \eqref{eq:pi}. Below, we note that any logconcave probability
measure satisfies \eqref{eq:pi} with $\cpi<\infty$, as established
in \cite{KLS95isop,bobkov99isoperimetric}.
\begin{lem}
\label{lem:chiq-SDPI} Let $\mu$ be a logconcave probability measure
over $\Rd$. For any $q\geq2$, it holds that
\[
\eta_{\chi^{q}}(\mu,t)\leq\bpar{1+\frac{t}{\cpi(\mu)}}^{-2/q}\quad\text{for }t>0\,.
\]
\end{lem}

To prove this result, we recall the general version of de Bruijn's
identity from \cite[Proposition 1.12]{KO25strong}.
\begin{prop}
[General de Bruijn's identity]\label{prop:KO-contraction} Let $\mu$
and $\nu$ be probability measures on $\Rd$ such that $\chi^{q}(\nu\mmid\mu)<\infty$
and $\E_{\mu}[\abs{\cdot}^{4}]<\infty$. Then, for any $t>0$ and
$q>1$,
\[
\de_{t}\chi^{q}(\nu_{t}\mmid\mu_{t})=-\frac{q\,(q-1)}{2}\,\E_{\mu_{t}}\Bbrack{\bpar{\frac{\D\nu_{t}}{\D\mu_{t}}}^{q}\,\babs{\nabla\log\frac{\D\nu_{t}}{\D\mu_{t}}}^{2}}=-\frac{q-1}{2}\,\Bnorm{\frac{\D\nu_{t}}{\D\mu_{t}}}_{L^{q}(\mu_{t})}^{q}\,\msf{RFI}_{q}(\nu_{t}\mmid\mu_{t})\,,
\]
where $\msf{RFI}_{q}(\nu\mmid\mu)=q\,\E_{\mu}[(\frac{\D\nu}{\D\mu})^{q}\,\abs{\nabla\log\frac{\D\nu}{\D\mu}}^{2}]/\E_{\mu}[(\frac{\D\nu}{\D\mu})^{q}]$
denotes the $q$-R\'enyi Fisher information.
\end{prop}

We also recall a R\'enyi version of \eqref{eq:pi} from \cite[Lemma 9]{VW23rapid}
that for $q\geq2$ and probability measure $\nu$ and $\mu$ with
$\nu\ll\mu$ and smooth $\nu/\mu$,
\begin{equation}
1-\exp\bpar{-\eu R_{q}(\nu\mmid\mu)}\leq\frac{q\cpi(\mu)}{4}\,\msf{RFI}_{q}(\nu\mmid\mu)\,.\tag{\ensuremath{\msf{R\text{-}PI}}}\label{eq:Renyi-PI}
\end{equation}

\begin{proof}
[Proof of Lemma~\ref{lem:chiq-SDPI}] We combine the general de
Bruijn identity with the R\'enyi PI. For the former, we note that
$\E_{\mu}[\abs{\cdot}^{4}]<\infty$ since \eqref{eq:pi} ensures exponential
integrability \cite{BU83inequality,GM83topological} (i.e., $\int\exp(s\,\abs{\cdot})\,\D\mu<\infty$
for $s^{2}<4/\cpi(\mu)$; see \cite[Proposition 4.4.2]{BGL14analysis}).
For the latter, $\nu_{t}$ and $\mu_{t}$ are smooth by an elementary
property of the heat kernel.

Using $\chi^{q}+1=\norm{\frac{\D\nu_{t}}{\D\mu_{t}}}_{L^{q}(\mu_{t})}^{q}:=L^{q}\ge1$
and combining the two, we deduce that 
\[
\de_{t}\chi^{q}(\nu_{t}\mmid\mu_{t})\leq-\frac{2}{q\cpi(\mu_{t})}\,(q-1)\,L^{q}(1-L^{-\frac{q}{q-1}})\,.
\]
Since $(q-1)\cdot x\,(1-x^{-1/(q-1)})\geq x-1$ for $x\geq1$, we
obtain 
\[
\de_{t}\chi^{q}(\nu_{t}\mmid\mu_{t})\leq-\frac{2}{q\cpi(\mu_{t})}\,(L^{q}-1)=-\frac{2}{q\cpi(\mu_{t})}\,\chi^{q}(\nu_{t}\mmid\mu_{t})\,.
\]
Using $\cpi(\mu_{t})\leq\cpi(\mu)+t$ and solving the differential
inequality yields 
\[
\chi^{q}(\nu_{t}\mmid\mu_{t})\leq\bpar{1+\frac{t}{\cpi(\mu)}}^{-2/q}\,\chi^{q}(\nu\mmid\mu)\,,
\]
which completes the proof.
\end{proof}
\begin{rem}
[R\'enyi version] Under the same condition $\E_{\mu}[\abs{\cdot}^{4}]<\infty$,
one can rederive \eqref{eq:Reniy-contraction-PI} by following the
argument of \cite{CCSW22improved}. Using the identity $\eu R_{q}=\frac{1}{q-1}\,\log(1+\chi^{q})$
together with the chain rule and de Bruijn's identity, we obtain:
\[
\de_{t}\eu R_{q}(\nu_{t}\mmid\mu_{t})=\frac{1}{q-1}\,\frac{\de_{t}\chi^{q}(\nu_{t}\mmid\mu_{t})}{L^{q}}=-\half\,\msf{RFI}_{q}(\nu_{t}\mmid\mu_{t})\,.
\]
When applying the R\'enyi PI, the inequality $1-e^{-x}\geq1/2$ for
$x\geq1$ is used when $\eu R_{q}\geq1$, and $1-e^{-x}\geq x/2$
for $x\leq1$ when $\eu R_{q}\leq1$. These control the dissipation
of the R\'enyi divergence in each regime.
\end{rem}

The mixing time of $\PS$ follows as a direct consequence of the contraction
result above. We ignore via the DPI the additional contraction due
to the backward step, although the backward step is also known to
exhibit similar contraction under suitable regularity assumptions
\cite{CCSW22improved}.
\begin{prop}
\label{prop:mixing-general} Let $q\geq2$ and $\E_{\pi^{X}}[\abs{\cdot}^{4}]<\infty$.
Then, the law $\pi_{k}^{X}$ of $X_{k}$ returned by the $\ps$ satisfies
$\chi^{q}(\pi_{k+1}^{X}\mmid\pi^{X})\leq\eta_{\chi^{q}}(\pi^{X},h)\,\chi^{q}(\pi_{k}^{X}\mmid\pi^{X})$.
Given $\veps\in(0,1)$, we can ensure that $\chi^{q}(\pi_{N}^{X}\mmid\pi^{X})\leq\veps$
after
\[
N\lesssim qh^{-1}\cpi(\pi^{X})\log\frac{\chi^{q}(\pi_{0}^{X}\mmid\pi^{X})}{\veps}\quad\text{iterations\,.}
\]
Moreover, $\eu R_{q}(\pi_{N}^{X}\mmid\pi^{X})\leq\veps$ after
\[
N\lesssim q\,\bpar{1\vee h^{-1}\cpi(\pi^{X})}\log\frac{M_{q}}{\veps}\quad\text{iterations\,.}
\]
\end{prop}

We adapt this result to the setting of logconcave probability measures.
\begin{lem}
\label{lem:psunif-mixing}For any $\veps\in(0,1)$, $q\geq2$, and
logconcave probability measure $\pi^{X}$ over $\Rd$, we can ensure
$\eu R_{q}(\pi_{N}^{X}\mmid\pi^{X})\leq\veps$ after
\[
N\lesssim q\,\bpar{1\vee h^{-1}\norm{\cov\pi^{X}}\log d}\log\frac{M_{q}}{\veps}\quad\text{iterations\,.}
\]
\end{lem}

\begin{proof}
Any logconcave probability measures satisfy \eqref{eq:pi} (and thus
$\E_{\pi^{X}}[\abs{\cdot}^{4}]<\infty$). The claim then follows from
the proposition above with $\cpi(\pi^{X})\lesssim\norm{\cov\pi^{X}}\log d$
for logconcave probability measures \cite{Klartag23log}.
\end{proof}

\subsection{Per-step analysis\label{subsec:uniform-query-complexity}}

We now focus on the case of $\pi^{X}\propto\ind_{\K}$, assuming $B_{1}(0)\subset\K$
by translation. Hereafter, we assume that the starting distribution
$\pi_{0}^{X}$ is $M_{q}$-warm with respect to $\pi^{X}$. By the
DPI, the same holds between $\pi_{1}^{Y}=\pi_{0}^{X}*\gamma_{h}$
and $\pi^{Y}=\pi^{X}*\gamma_{h}$ (i.e., $\norm{\D\pi_{1}^{Y}/\D\pi^{Y}}_{L^{q}(\pi^{Y})}<\infty$),
and this property is preserved throughout subsequent iterations (i.e.,
$\eu R_{q}(\pi_{k}^{X}\mmid\pi^{X})\leq\eu R_{q}(\pi_{0}^{X}\mmid\pi^{X})$).
For brevity, we will use the $L^{q}$-warmness, $\chi^{q}$-warmness,
or $\eu R_{q}$-warmness interchangeably, as justified by \eqref{eq:div-equivalence}.

\subsubsection{High-level idea of analysis}

Let us first analyze $\psunif$ as a one-attempt algorithm that may
declare failure if the inner rejection sampler exceeds the threshold
$\tau$. Fix an iteration horizon $N\in\mathbb{N}$, step size $h>0$,
and inner cap $\tau\in\mathbb{N}$.

Let $S_{i}$ denote the \emph{survival} event that failure has not
occurred up to and including step $i$, with $S_{0}=\Omega$ and full
success event $S:=S_{N}$. Let $C_{i}\ge0$ be the number of membership
queries used by the backward step at step $i$. The total query cost
is then $C:=\sum_{i=1}^{N}\ind[S_{i-1}]\,C_{i}$. Note that
\[
\E C=\sum_{i=1}^{N}\E\bbrack{\ind[S_{i-1}]\,C_{i}}\leq\sum_{i=1}^{N}\E C_{i}\,,
\]
so it suffices to bound $\E C_{i}$. Similarly, if $F_{i}$ denotes
the event of failure \emph{at} step $i$, and $q_{i}$ denotes the
cap-hit probability at step $i$, then 
\[
\P(F_{i})=\E\bbrack{\ind[S_{i-1}]\,q_{i}}\leq\E q_{i},\qquad\P(S^{c})\le\sum_{i=1}^{N}\E q_{i}(Y_{i})\,.
\]
These inequalities allow us to carry out the main per-step analysis
under the clean unconditional intermediate laws (e.g., $Y_{i}\sim\pi_{i}^{Y}$),
without explicitly tracking the conditional distributions at each
step.

\paragraph{(1) Bias from failure.}

It may not be immediately clear whether the transition kernel of $\PS$
is preserved under the algorithmic modification of imposing a retry
threshold $\tau$. For completeness, we verify that the backward step
is indeed unchanged by this modification.
\begin{lem}
\label{lem:correctness_backward} Upon success of the rejection sampling
in the backward step, the accepted sample is distributed according
to $\pi^{X|Y=y}$.
\end{lem}

\begin{proof}
Consider the proposal $\mu_{y}$ with success probability $p_{y}>0$
such that $X_{i}\sim\mu_{y}$ for each retry $i\in[\tau]$ satisfies
$\law(X_{i}\,|\,\text{accepted})=\pi^{X|Y=y}$. Let $X$ be the random
sample given by the backward step. For any event $A$ in $\Rd$, we
have
\begin{align*}
\P(X\in A\,|\,\text{accepted}) & =\frac{\P(X\in A,\,\text{success})}{\P(\text{success})}\\
 & =\frac{\sum_{i=1}^{\tau}\P(\text{fail in the first }(i-1)\text{ trials},\,X_{i}\ \text{accepted},\,X_{i}\in A)}{\sum_{i=1}^{\tau}\P(\text{success in the }i\text{-th trial})}\\
 & =\frac{\sum_{i=1}^{\tau}(1-p_{y})^{i-1}\,\P(X_{i}\in A,\,X_{i}\ \text{accepted})}{\sum_{i=1}^{\tau}(1-p_{y})^{i-1}p_{y}}=\frac{\P(X_{1}\in A,\,X_{1}\ \text{accepted})}{p_{y}}\\
 & =\P(X_{1}\in A\,|\,\text{accepted})=\pi^{X|Y}(A)\,,
\end{align*}
which completes the proof.
\end{proof}
However, the law of output conditioned on success introduces a small
bias, which is bounded roughly by the failure probability of the algorithm.
\begin{lem}
[Bias from success, \cite{KVZ26INO}] \label{lem:bias-from-failure}
Let $S$ be the event that the failure is not declared over the required
$N$ iterations, and suppose $\P(S)\geq1-\eta$. Let $\pi_{N}^{X}$
be the law of the uncapped $\psunif$ output after $N$ iterations,
and $\hat{\pi}_{N}^{X}$ be the law of $\psunif$ conditioned on $S$
(i.e., $\hat{\pi}_{N}^{X}:=\law(X_{N}\mid S)$). Then,
\begin{align*}
\frac{\D\hat{\pi}_{N}^{X}}{\D\pi_{N}^{X}} & \leq\frac{1}{1-\eta}\quad\textup{so}\quad\eu R_{\infty}(\hat{\pi}_{N}^{X}\mmid\pi_{N}^{X})\leq\log\frac{1}{1-\eta}\,,\\
\eu R_{q}(\hat{\pi}_{N}^{X}\mmid\pi^{X}) & \leq\eu R_{q}(\pi_{N}^{X}\mmid\pi^{X})+\frac{q}{q-1}\log\frac{1}{1-\eta}\qquad\text{for }q>1\,.
\end{align*}
\end{lem}

\paragraph{(2) Restart-until-success.}

A practical \emph{restart-until-success} strategy repeats independent
attempts until $S$ occurs. It outputs a sample with (conditional)
law $\hat{\pi}_{N}^{X}$ and terminates almost surely whenever $\P(S)>0$.
Moreover, if $\mathbb{P}(S)\ge1-\eta$, then the expected number of
attempts is at most $(1-\eta)^{-1}$ and 
\[
\E[C\mid S]\leq\frac{\E C}{\P(S)}\leq\frac{\E C}{1-\eta}\,,\qquad\E C_{\text{restart}}=\frac{\E C}{\P(S)}\leq\frac{\E C}{1-\eta}\,.
\]
It therefore suffices to bound the one-attempt success probability
$\P(S)$ and the unconditional expected cost $\E C$.

\subsubsection{Per-step analysis}

The analysis follows \cite{KV25faster} closely, which establishes
the query complexity of $\psunif:\eu R_{\Otilde(1)}\to\eu R_{2}$.
The only difference lies in the required number $N$ of iterations,
namely the number needed for contraction in $\eu R_{q}$. Thus, we
recall the relevant results from \S2.1 of \cite{KV25faster} and
adapt them minimally for our setting. With the parameter choices $Z=\frac{16NM_{2}}{\eta}$,
$c=\frac{\log\log Z}{2\log Z}$, $\tau=Z^{2}\log^{3}Z$, and $h=\frac{c}{d^{2}}$,
the failure probability satisfies $\E_{\pi_{k}^{Y}}[(1-\ell)^{\tau}]\leq\frac{\eta}{N}$.
The per-step complexity is bounded as follows: setting $q=1+\alpha$
with $\alpha\ge1\vee\log\tau$, it follows that $\E_{\pi_{k}^{Y}}[\frac{1}{\ell}\wedge\tau]\leq5M_{q}\log^{4}Z$.

We now prove the order-preserving complexity guarantees of $\psunif$
under an $O(1)$-warm start and sufficient condition on $q$.
\begin{prop}
\label{prop:unif-sampling} Let $\pi$ be the uniform distribution
over a convex body $\K\subset\Rd$ with $B_{1}\subset\K$ and $\Lambda=\norm{\cov\pi}$,
given access to a membership oracle for $\K$. Given $\veps>0$ ,
$\eta\in(0,1/2)$, $q\geq2$, and an initial distribution $\pi_{0}$
with $M_{q}=\norm{\nicefrac{\D\pi_{0}}{\D\pi}}_{L^{q}(\pi)}$, initialize
$\psunif$ with $\pi_{0}$ and iterate it $N=\Otilde(qd^{2}\Lambda\log\frac{1}{\veps}\log\frac{1}{\eta})$
times with $h=(2d^{2}\log Z)^{-1}$ and $\tau=Z^{2}\log^{3}Z$ for
$Z:=\frac{16NM_{2}}{\eta}$. If $q\geq1+\log\tau$ and $M_{q}\leq10$,
then
\begin{itemize}
\item The sufficient condition on $q$ is satisfied when $q\geq2\vee O(\log\frac{N}{\eta})$.
\item With probability at least $1-\eta$, $\psunif$ iterates $N$ times
without failure. Conditioned on this success, the output $Z\sim\nu$
satisfies $\eu R_{q}(\nu\mmid\pi)\leq\veps+2\log\frac{1}{1-\eta}$,
with $\Otilde(qd^{2}\Lambda\log\frac{1}{\veps}\log^{5}\frac{1}{\eta})$
membership queries in expectation.
\end{itemize}
\end{prop}

\begin{proof}
The uncapped $\psunif$ can ensure $\eu R_{q}(\pi_{N}^{X}\mmid\pi^{X})\leq\veps$
after iterating 
\[
N\gtrsim q\,\bpar{1\vee h^{-1}\cpi(\pi)}\log\frac{M_{q}}{\veps}\,.
\]
Under the choice of $h=cd^{-2}=\frac{\log\log Z}{2d^{2}\log Z}$
with $Z\asymp\frac{NM_{2}}{\eta}$, the required number $N$ of iterations
must satisfy an inequality of the form 
\[
N\gtrsim A\log(BN)\,,
\]
which can be fulfilled if $N\gtrsim A\log(AB)$. Therefore, the required
number of iterations is of order 
\[
N=\Otilde\bpar{qd^{2}\cpi(\pi)\log\frac{M_{q}}{\veps}\log\frac{M_{q}}{\eta}}=\Otilde\bpar{qd^{2}\Lambda\log\frac{1}{\veps}\log\frac{1}{\eta}}\,.
\]
where $\Lambda=\norm{\cov\pi}$, and the last equality follows from
the assumption $M_{q}\leq10$.

We now find a sufficient condition for ensuring $q\geq1+\log\tau$.
Observe that when $q\geq2$,
\[
1+\log\tau=1+\frac{q}{q-1}\log(Z\log^{4}Z)\lesssim\log Z\lesssim\log\frac{N}{\eta}\,.
\]
Hence, when $q\geq2\vee O(\log\frac{N}{\eta})$, then expected total
query complexity of $\psunif:\eu R_{q}\to\eu R_{q}$ from an $O(1)$-warm
start is 
\[
\Otilde\bpar{qd^{2}\Lambda\log\frac{1}{\veps}\log\frac{1}{\eta}\times\log^{4}Z}=\Otilde\bpar{qd^{2}\Lambda\log\frac{1}{\veps}\log^{5}\frac{1}{\eta}}\,.
\]
Lastly, the success probability satisfies $\P(S)\geq\eta$ by construction,
and by Lemma~\ref{lem:bias-from-failure}, the output law $\hat{\pi}_{N}^{X}$
conditioned on success satisfies 
\[
\eu R_{q}(\hat{\pi}_{N}^{X}\mmid\pi^{X})\leq\veps+2\log\frac{1}{1-\eta}\,,
\]
which completes the proof.
\end{proof}
Using this proposition, we prove Theorem~\ref{thm:unif-warm}.
\begin{proof}
[Proof of Theorem~\ref{thm:unif-warm}] Take $\eta=\veps<0.01$,
and restart the algorithm upon failure. Note that the expected number
of restarts is at most $0.02$. To invoke the proposition above, $q$
should satisfy $q\geq2\vee O(\log N)$ for $N=\Otilde(qd^{2}\Lambda\log^{2}\frac{1}{\veps})$.
To ensure 
\[
q\geq\log\Otilde\bpar{qd^{2}\Lambda\log^{2}\frac{1}{\veps}}(\geq\log N)\,,
\]
let us solve the inequality of the form $q\gtrsim A\log qB$. A sufficient
condition for this is simply $q\gtrsim A\log AB=\Otilde(A\log B)$,
so it suffices to enforce 
\begin{equation}
q\geq2\vee\Otilde\bpar{\log d^{2}\Lambda+\log\log\frac{1}{\veps}}\,.\label{eq:q-suff-unif}
\end{equation}
Hence, by Proposition~\ref{prop:unif-sampling}, the restart version
of $\psunif$ uses $\Otilde(qd^{2}\Lambda\log^{6}\frac{1}{\veps})$
total queries in expectation, and its output is distributed as $\hat{\pi}_{N}^{X}$
satisfying $\eu R_{q}(\hat{\pi}_{N}^{X}\mmid\pi^{X})\leq3\veps$.
Retaking $\veps\gets\veps/3$, we complete the proof.
\end{proof}

\section{Hypercontractivity under the simultaneous heat flow\label{sec:hypercontractivity}}

In this section, we establish hypercontractivity under simultaneous
heat flow, which will serve as a stepping stone for relaying a R\'enyi
warmness in annealing in the next section.

\subsection{Adjoint of the heat semigroups and the proximal sampler\label{subsec:intro-semigroup}}

We first recap the Markov semigroup theory for a gentle introduction
to heat semigroup and its adjoint. Interested readers may want to
refer to \cite{van14probability,BGL14analysis,chewi25log} for further
details on the semigroup theory. Meanwhile, we take detour to $\PS$,
since it sets up a context in which the introduction of heat semigroups
and its adjoint is natural.

\paragraph{Markov semigroup theory.}

Given a (time-homogeneous) Markov process $(X_{t})_{t\geq0}$ on a
domain $E$, its associated \emph{Markov semigroup} $(P_{t})_{t\geq0}$
is defined as a family of functional operators such that
\[
(P_{t}f)(x):=\E[f(X_{t})\mid X_{0}=x]\quad\text{for every bounded measurable function }f:E\to\R\,.
\]
One could readily check that $P_{t}$ is a linear operator with $P_{0}=\id$,
$P_{t+s}=P_{t}\circ P_{s}$ for any $t,s\geq0$ (semigroup property),
$P_{t}\ind=\ind$ (mass conservation), and $P_{t}f\geq0$ for any
$f\geq0$ (positivity preserving). Below, we omit $E$ in the integral
sign.

A measure $\mu$ is said to  be \emph{invariant} (or stationary) for
$P_{t}$ if $\int P_{t}f\,\D\mu=\int f\,\D\mu$ for every $t\geq0$.
In addition, it is said to be \emph{reversible} for $P_{t}$ if $\int P_{t}f\,g\,\D\mu=\int f\,P_{t}g\,\D\mu$
for any $f,g\in L^{2}(\mu)$ and $t\geq0$. For any $p\geq1$, it
follows from Jensen's inequality that $P_{t}$ is a contraction operator
on the bounded functions in $L^{p}(\mu)$ (i.e., $\norm{P_{t}f}_{L^{p}(\mu)}\leq\norm f_{L^{p}(\mu)}$).
This furnishes the extension of $P_{t}$ to $L^{p}(\mu)$ by a density
argument.

\paragraph{The proximal sampler.}

As examined in \cite{CCSW22improved} (with further details in \cite{chewi25log}),
the forward step (i.e., $y_{k+1}\sim\pi^{Y|X=x_{k}}=\mc N(x_{k},hI_{d})$)
can be viewed as simulating Brownian motion for time $h$ started
at $x_{k}$; when $\D X_{t}=\D B_{t}$ with $X_{0}=x_{k}$, we have
$X_{h}\sim\law y_{k+1}$. The associated Markov semigroup of Brownian
motion is often called the \emph{heat semigroup}, which admits the
formula of $P_{t}f=f*\gamma_{t}=:f_{t}$. Based on the classical heat
equation $\de_{t}P_{t}f=\half\,\Delta P_{t}f$, rich theory for the
heat semigroup facilitates one to understand the forward step in a
better way (e.g., contraction under simultaneous heat flow).

This prompts a natural question whether one could also introduce a
semigroup structure to the backward step. It turns out that the backward
step can be interpreted as an SDE $(Z_{t}^{\leftarrow})_{t\geq0}$,
given by 
\[
\D Z_{t}^{\leftarrow}=\nabla\log(\pi^{X}*\gamma_{h-t})\,(Z_{t}^{\leftarrow})\,\D t+\D\bar{B}_{t}\qquad\text{for another Brownian motion }\bar{B}_{t}\,,
\]
so that if $Z_{0}^{\leftarrow}\sim\pi^{Y}$, then $Z_{h}^{\leftarrow}\sim\pi^{X}$.
Though it is non-trivial to justify, one could guess that $Z_{h}^{\leftarrow}\sim\pi^{X|Y=y}$
if $Z_{0}^{\leftarrow}=y$. Hence, one may be able to introduce a
semigroup structure to the backward step via, for any suitable test
function $f$, 
\[
\overline{Q}_{h}f(y):=\E[f(Z_{h}^{\leftarrow})\mid Z_{0}^{\leftarrow}=y]=\E_{\pi^{X|Y=y}}f\,.
\]
As seen shortly, Klartag and Putterman \cite{KP23spectral} defined
the same object as the $L^{2}(\pi)$-adjoint of the heat semigroups,
providing solid ground where similar rigorous theory can be developed
for the backward step.

\paragraph{$L^{2}$-adjoint of heat semigroups.}

Given a probability measure $\pi$, \cite{KP23spectral} viewed the
heat semigroup $P_{t}$ as a contraction operator from $L^{2}(\pi_{t})$
to $L^{2}(\pi)$:
\[
\norm{P_{t}f}_{L^{2}(\pi)}^{2}=\int(P_{t}f)^{2}\,\pi\leq\int P_{t}(f^{2})\,\pi=\int f^{2}\,P_{t}\pi=\norm f_{L^{2}(\pi_{t})}^{2}\,,
\]
where the equality in the middle follows simply from reversibility
of the heat semigroups. This perspective led them to define the \emph{adjoint
operator} $Q_{t}:L^{2}(\pi)\to L^{2}(\pi_{t})$ as $Q_{0}=\id$ and
\[
Q_{t}f=\frac{P_{t}(f\pi)}{P_{t}\pi}=\frac{P_{t}(f\pi)}{\pi_{t}}\quad\text{for }t>0\,.
\]

This formula can be naturally derived, though not mentioned in the
original manuscript, from the definition of adjoint: for any $f\in L^{2}(\pi)$
and $g\in L^{2}(\pi_{t})$, 
\[
\inner{f,P_{t}g}_{L^{2}(\pi)}=\inner{Q_{t}f,g}_{L^{2}(\pi_{t})}\,.
\]
For $\D Z_{t}=\D B_{t}$, one could rewrite the LHS as follows:
\[
\inner{f,P_{t}g}_{L^{2}(\pi)}=\int f\,P_{t}g\,\D\pi=\E\bbrack{f(Z_{0})\,\E[g(Z_{t})|Z_{0}]}=\E[f(Z_{0})\,g(Z_{t})]=\E_{Z_{t}\sim\pi_{t}}\bbrack{g(Z_{t})\,\E[f(Z_{0})|Z_{t}]}\,.
\]
Equating this with $\inner{Q_{t}f,g}_{L^{2}(\pi_{t})}$, we can deduce
that
\[
Q_{t}f(y)=\E[f(Z_{0})|Z_{t}=y]\,.
\]
Since $\law(Z_{0}|Z_{t}=y)=\pi^{X|Y=y}$ with $\pi^{X}=\pi$,
\[
Q_{t}f(y)=\int f\,\D\pi^{X|Y=y}=\frac{\int f(x)\,\pi^{X}(x)\,\gamma_{t}(y-x)\,\D x}{\int\pi^{X}(x)\,\gamma_{t}(y-x)\,\D x}=\frac{(f\pi)*\gamma_{t}}{\pi*\gamma_{t}}=\frac{P_{t}(f\pi)}{\pi_{t}}\,.
\]
We note that this heat adjoint could be viewed as the backward step
of $\PS$: $\overline{Q}_{t}f=Q_{t}f$. Also, it follows from the
above formula that the adjoint operator satisfies mass preservation
($Q_{t}\ind=\ind$) and positivity preserving ($Q_{t}f\geq0$ if $f\geq0$).
Moreover, it is a contraction operator; for $f\in L^{2}(\pi)\backslash\{0\}$,
by Cauchy-Schwarz
\[
(Q_{t}f)^{2}\,\pi_{t}=\frac{\bpar{P_{t}(f\pi)}^{2}}{\pi_{t}}\leq\frac{P_{t}(f^{2}\pi)\,P_{t}\pi}{\pi_{t}}=P_{t}(f^{2}\pi)\,.
\]
and integration of both sides yields $\norm{Q_{t}f}_{L^{2}(\pi_{t})}^{2}\leq\norm f_{L^{2}(\pi)}^{2}$.
We refer readers to \cite[\S B.2]{KV25localization}.

\subsection{Hypercontractivity of an adjoint of the heat semigroups}

We now show something stronger than mere contraction of $Q_{t}$ from
$L^{2}(\pi)$ to $L^{2}(\pi_{t})$, called hypercontractivity. To
begin with, we recall hypercontractivity of Markov semigroups under
\eqref{eq:lsi}:
\begin{prop}
[Gross]\label{prop:Gross-hypercontractivity} For $C>0$ and $p\in(1,\infty)$,
let $q(t)=1+(p-1)\,e^{2t/C}$ for $t\geq0$, and $P_{t}$ be a reversible
Markov semigroup with stationary measure $\nu$. Then, $\nu$ satisfies
\eqref{eq:lsi} with parameter $C$ if and only if 
\[
\norm{P_{t}f}_{L^{q(t)}(\nu)}\leq\norm f_{L^{p}(\nu)}\quad\text{for all }p\in(1,\infty),t\geq0,\text{ and }f\in L^{p}(\nu)\,.
\]
\end{prop}

It means that when a stationary measure satisfies \eqref{eq:lsi},
the semigroup $P_{t}$ not only contracts in any $L^{p}$-spaces but
also improves integrability (recall $L^{q}\subset L^{p}$ for $p<q$
due to monotonicity of $\norm{\cdot}_{L^{p}}$ in $p$). We now prove
Theorem~\ref{thm:hypercontractivity-heat-flow}, hypercontractivity
under the simultaneous heat flow with different $q(t)=1+(p-1)\,(1+\nicefrac{t}{\clsi(\nu)})$;
\[
\norm{Q_{t}f}_{L^{q(t)}(\nu_{t})}\leq\norm f_{L^{p}(\nu)}\quad\text{for all }p\in(1,\infty),t\geq0,\text{ and }f\in L^{p}(\nu)\,.
\]
Note that the reference measure $\nu_{t}$ \emph{also} evolves over
time $t$.
\begin{proof}
[Proof of Theorem~\ref{thm:hypercontractivity-heat-flow}]We first
prove the claim for $f\in C_{c,+}^{\infty}(\Rd):=\{f+c:f\in C_{c}^{\infty}(\Rd),f\geq0,c>0\}$.
Let $\bar{f}$ be the normalization of $f$ with respect to $\nu$
(i.e., $\bar{f}=f/\int f\,\D\nu$).

Denoting a new probability measure $\mu$ by $\D\mu=\bar{f}\,\D\nu$,
we define $\bar{f}_{t}:=Q_{t}\bar{f}=\D\mu_{t}/\D\nu_{t}$ and
\[
\mc F(t,q):=\int\bpar{\frac{\D\mu_{t}}{\D\nu_{t}}}^{q}\,\D\nu_{t}=\int\bar{f}_{t}^{q}\,\D\nu_{t}\,,\qquad\mc G(t):=\frac{1}{q(t)}\,\log\mc F\bpar{t,q(t)}\,.
\]
It then suffices to show that $\de_{t}\mc G\leq0$ for a suitable
choice of $q(t)$. Note that 
\[
\de_{t}\mc G=-\frac{q'(t)}{q^{2}(t)}\,\log\mc F+\frac{\de_{t}\mc F+q'(t)\,\de_{q}\mc F}{q(t)\,\mc F}\,.
\]
With assumptions of $q(t),q'(t)>0$ for a moment, $\de_{t}\mc G\leq0$
is equivalent to 
\[
\frac{1}{q'(t)}\,\de_{t}\mc F+\de_{q}\mc F-\frac{1}{q(t)}\,\mc F\log\mc F\leq0\,.
\]

We now compute $\de_{t}\mc F$ and $\de_{q}\mc F$. By general de
Bruijn's identity in Proposition~\ref{prop:KO-contraction} (note
that $\E_{\nu}[|\cdot|^{4}]<\infty$ since \eqref{eq:lsi} implies
\eqref{eq:pi}),
\[
\de_{t}\mc F=-\frac{q\,(q-1)}{2}\,\E_{\nu_{t}}\bbrack{\bpar{\frac{\D\mu_{t}}{\D\nu_{t}}}^{q-2}\,\babs{\nabla\frac{\D\mu_{t}}{\D\nu_{t}}}^{2}}=-\frac{q\,(q-1)}{2}\,\E_{\nu_{t}}[\bar{f}_{t}^{q-2}\,\abs{\nabla\bar{f}_{t}}^{2}]\,.
\]
As for $\de_{q}\mc F$, we have
\begin{align*}
\de_{q}\mc F & =\de_{q}\int\bar{f}_{t}^{q}\,\D\nu_{t}\underset{(i)}{=}\int\bar{f}_{t}^{q}\log\bar{f}_{t}\,\D\nu_{t}=\frac{1}{q}\,\Bpar{\ent_{\nu_{t}}(\bar{f}_{t}^{q})+\int\bar{f}_{t}^{q}\,\D\nu_{t}\,\log\int\bar{f}_{t}^{q}\,\D\nu_{t}}\\
 & \leq\frac{1}{q}\,\Bpar{2\clsi(\nu_{t})\,\E_{\nu_{t}}[\abs{\nabla(\bar{f}_{t}^{q/2})}^{2}]+\int\bar{f}_{t}^{q}\,\D\nu_{t}\,\log\int\bar{f}_{t}^{q}\,\D\nu_{t}}\\
 & =\frac{1}{q}\,\bpar{\frac{q^{2}\clsi(\nu_{t})}{2}\,\E_{\nu_{t}}[\bar{f}_{t}^{q-2}\,\abs{\nabla\bar{f}_{t}}^{2}]+\mc F\log\mc F}\,.
\end{align*}
Differentiation under the integral sign in $(i)$ can be justified
as follows: $\de_{q}\bar{f}_{t}^{q}=(Q_{t}\bar{f})^{q}\log Q_{t}\bar{f}$
is continuous in $q$, and we can take a small $\delta>0$ such that
for any $q'\in(q-\delta,q+\delta)$ with $q-\delta>1$,
\[
\bar{f}_{t}^{q'}\log\bar{f}_{t}=(\bar{f}_{t}\,\ind[\bar{f}_{t}\geq1]+\bar{f}_{t}\,\ind[\bar{f}_{t}<1])^{q'}\log\bar{f}_{t}\leq2^{q+\delta-1}\,(\bar{f}_{t}^{q+\delta}+\bar{f}_{t}^{q-\delta})\log\bar{f}_{t}\,.
\]
Hence, it amounts to showing integrability of the RHS. To this end,
we first show  that $\bar{f}_{t}\in L^{r}(\nu_{t})$ for any $r>1$.
By Jensen's inequality,
\[
\norm{\bar{f}_{t}}_{L^{r}(\nu_{t})}^{r}=\E_{\nu_{t}}[(Q_{t}\bar{f})^{r}]\leq\E_{\nu_{t}}[Q_{t}(\bar{f}^{r})]=\E_{\nu}[\bar{f}^{r}]=\norm{\bar{f}}_{L^{r}(\nu)}^{r}\,,
\]
and $\bar{f}\in C_{c,+}^{\infty}(\Rd)$ ensures that $\norm{\bar{f}}_{L^{r}(\nu)}<\infty$.
Thus, integrability of the RHS follows from $L^{s}\subset L^{r}\log L$
for any $r<s$.

Combining the bounds on $\de_{t}\mc F$ and $\de_{q}\mc F$,
\begin{align*}
 & \frac{1}{q'(t)}\,\de_{t}\mc F+\de_{q}\mc F-\frac{1}{q(t)}\,\mc F\log\mc F\\
 & \leq-\frac{q\,(q-1)}{2\,q'(t)}\,\E_{\nu_{t}}[\bar{f}_{t}^{q-2}\,\abs{\nabla\bar{f}_{t}}^{2}]+\frac{1}{q}\,\bpar{\frac{q^{2}\clsi(\nu_{t})}{2}\,\E_{\nu_{t}}[\bar{f}_{t}^{q-2}\,\abs{\nabla\bar{f}_{t}}^{2}]+\mc F\log\mc F}-\frac{1}{q(t)}\,\mc F\log\mc F\\
 & =\bpar{\clsi(\nu_{t})-\frac{q-1}{q'(t)}}\times\frac{q}{2}\,\E_{\nu_{t}}[\bar{f}_{t}^{q-2}\,\abs{\nabla\bar{f}_{t}}^{2}]\,.
\end{align*}
Hence, it suffices to enforce
\[
\frac{(q-1)'}{q-1}\leq\frac{1}{\clsi(\nu_{t})}\,.
\]
As $\clsi(\nu_{t})\leq\clsi(\nu)+t$, we can guarantee this by solving
\[
\frac{(q-1)'}{q-1}\leq\frac{1}{\clsi(\nu)+t}\,.
\]
Integrating both sides over $[0,t]$, and setting $q(0)=p$,
\[
\log\frac{q(t)-1}{p-1}\leq\log\frac{\clsi(\nu)+t}{\clsi(\nu)}\,,
\]
and thus we can simply take
\[
\frac{q(t)-1}{p-1}=1+\frac{t}{\clsi(\nu)}\,,
\]
which clearly satisfies $q,q'\geq0$. Thus, $\norm{\bar{f}_{t}}_{L^{q(t)}(\nu_{t})}\leq\norm{\bar{f}}_{L^{p}(\nu)}$
(so $\norm{f_{t}}_{L^{q(t)}(\nu_{t})}\leq\norm f_{L^{p}(\nu)}$),
and hypercontractivity holds on $C_{c,+}^{\infty}(\Rd)$.

Every $f\in C_{c,+}^{\infty}(\Rd)$ can be written as $f=f_{\delta}=g+\delta$
for some non-negative $g\in C_{c}^{\infty}(\Rd)$ and $\delta>0$.
We then apply the dominated convergence theorem to $\norm{Q_{t}f_{\delta}}_{L^{q(t)}(\nu_{t})}\leq\norm{f_{\delta}}_{L^{p}(\nu)}$
as $\delta\to0$, extending the main claim from $C_{c,+}^{\infty}(\Rd)$
to $C_{c,\geq0}^{\infty}(\Rd):=C_{c}^{\infty}(\Rd)\cap\{f\geq0\}$. 

Now, for any non-negative $f\in L^{p}(\nu)$, we can take a sequence
$\{f_{n}\}$ of functions in $C_{c,\geq0}^{\infty}(\Rd)$ such that
$f_{n}\to f$ in $L^{p}$. Moreover, we can take a subsequence $\{n_{k}\}_{k}$
such that $f_{n_{k}}\to f$ almost everywhere. Since $Q_{t}f_{n}\geq0$
due to positivity-preserving, Fatou's lemma implies that
\[
\norm{Q_{t}f}_{L^{q(t)}(\nu_{t})}^{q(t)}=\int\liminf_{k\to\infty}(Q_{t}f_{n_{k}})^{q(t)}\,\D\nu_{t}\leq\liminf_{k\to\infty}\int(Q_{t}f_{n_{k}})^{q(t)}\,\D\nu_{t}\,,
\]
and thus,
\[
\norm{Q_{t}f}_{L^{q(t)}(\nu_{t})}\leq\liminf_{k\to\infty}\,\norm{Q_{t}f_{n_{k}}}_{L^{q(t)}(\nu_{t})}\leq\liminf_{k\to\infty}\,\norm{f_{n_{k}}}_{L^{p}(\nu)}=\norm f_{L^{p}(\nu)}\,.
\]

Lastly, for any $f\in L^{p}(\nu)$, note that $\abs f=f_{+}+f_{-}$
is also in $L^{p}(\nu)$. Thus, 
\[
\norm{Q_{t}f}_{L^{q(t)}(\nu_{t})}=\norm{\abs{Q_{t}f}}_{L^{q(t)}(\nu_{t})}\leq\norm{Q_{t}(\abs f)}_{L^{q(t)}(\nu_{t})}\leq\norm{\abs f}_{L^{p}(\nu)}=\norm f_{L^{p}(\nu)}\,,
\]
which finishes the proof.
\end{proof}
For $f:=\D\mu/\D\nu\in L^{p}(\nu)$ (which is equivalent to $\eu R_{p}(\mu\mmid\nu)<\infty$),
hypercontractivity of $Q_{t}$ implies that
\[
\Bnorm{\frac{\D\mu_{t}}{\D\nu_{t}}}_{L^{q(t)}(\nu_{t})}=\norm{Q_{t}f}_{L^{q(t)}(\nu_{t})}\leq\norm f_{L^{p}(\nu)}=\Bnorm{\frac{\D\mu}{\D\nu}}_{L^{p}(\nu)}\,.
\]
Then, $\eu R_{q(t)}(\mu_{t}\mmid\nu_{t})\leq\eu R_{p}(\mu\mmid\nu)$
follows from
\[
\eu R_{p}(\mu\mmid\nu)=\frac{p}{p-1}\log\,\Bnorm{\frac{\D\mu}{\D\nu}}_{L^{p}(\nu)}\geq\frac{p}{p-1}\,\log\,\Bnorm{\frac{\D\mu_{t}}{\D\nu_{t}}}_{L^{q(t)}(\nu_{t})}\geq\frac{q(t)}{q(t)-1}\,\log\,\Bnorm{\frac{\D\mu_{t}}{\D\nu_{t}}}_{L^{q(t)}(\nu_{t})}\,.
\]
This looks similar to the DPI for R\'enyi divergence, but we remark
that the R\'enyi order of the RHS is larger that that of the LHS.

\paragraph{Improved convergence rate of PS under LSI.}

Just as in the SDPI under \eqref{eq:pi}, the SDPI under \eqref{eq:lsi}
was established in \cite{CCSW22improved,KO25strong}: for $q>1$ and
two probability measures $\mu,\nu$ with $\mu\ll\nu$,
\begin{equation}
\eu R_{q}(\mu_{t}\mmid\nu_{t})\leq\frac{\eu R_{q}(\mu\mmid\nu)}{(1+t/\clsi(\nu))^{1/q}}\,.\label{eq:PS-LSI}
\end{equation}
Combining this contraction with hypercontractivity, we can derive
a stronger mixing result of $\PS$: when $\pi_{k}$ denoting the law
of the $k$-th sample returned by $\PS$ with initial $\pi_{0}$ and
target $\pi$, for any given $\veps>0$ and $p,q>1$, we have  $\eu R_{q}(\pi_{N}\mmid\pi)\leq\veps$
after $N\lesssim h^{-1}\clsi(\pi)\log\frac{q\eu R_{2\wedge p}(\pi_{0}\mmid\pi)}{(p-1)\,\veps}$
iterations.
\begin{proof}
[Proof of Corollary~\ref{cor:PS-mixing-LSI}] We first iterate $\PS$
$N_{0}\asymp(2\wedge p)\,h^{-1}\clsi(\pi)\log\frac{\eu R_{(2\wedge p)}(\pi_{0}\mmid\pi)}{\veps}$
times to bring $\pi_{N_{0}}$ within $\veps$-distance to $\pi$ in
$\eu R_{2\wedge p}$. Next, we iterate it $N_{1}\asymp h^{-1}\clsi(\pi)\log_{2}\frac{q-1}{1\wedge(p-1)}$
more times, where we invoke DPI and hypercontractivity each iteration
as follows: for $i\leq N_{1}$ and $\pi_{i}^{X}=\pi_{i}$,
\[
\eu R_{q_{i+1}}(\pi_{N_{0}+i+1}\mmid\pi)\leq\eu R_{q_{i+1}}(\pi_{N_{0}+i+1}^{Y}\mmid\pi^{Y})\le\eu R_{q_{i}}(\pi_{N_{0}+i}\mmid\pi)\,,
\]
where $(q_{i+1}-1)/(q_{i}-1)=1+h/\clsi(\pi)$. Then,
\[
\frac{q_{N_{1}}-1}{1\wedge(p-1)}=\frac{q_{N_{1}}-1}{q_{0}-1}=\prod_{i=1}^{N_{1}}\frac{q_{i}-1}{q_{i-1}-1}=\bpar{1+\frac{h}{\clsi(\pi)}}^{N_{1}}\geq\frac{q-1}{1\wedge(p-1)}\,.
\]
Therefore, $\PS$ suffices to iterate $N=N_{0}+N_{1}\lesssim h^{-1}\clsi(\pi)\log\frac{q\eu R_{2\wedge p}(\pi_{0}\mmid\pi)}{(1\wedge(p-1))\,\veps}$
times to attain the desired guarantee.
\end{proof}
As a simple corollary, this result leads to a mixing time of $\PS$
in the well-conditioned setting, which improves the linear dependence
of $\PS$ on $q$ \cite{CCSW22improved}.
\begin{prop}
[Zeroth-order sampling in well-conditioned setting]\label{prop:mixing-PS-smooth}
Let $\pi\propto e^{-V}$ with $\alpha$-strongly convex and $\beta$-smooth
$V$ over $\Rd$, given access to an evaluation and proximal oracle
for $V$. Let $\pi_{k}$ denote the law of the $k$-th sample returned
by $\PS$. Given $q>1$ and $\veps>0$, we can attain $\eu R_{q}(\pi_{N}\mmid\pi)\leq\veps$,
using $\Otilde(\kappa d\log\frac{q}{\veps})$ queries to both oracles,
where $\kappa=\beta/\alpha$.
\end{prop}

\begin{proof}
One can find an initial distribution $\pi_{0}$ such that $\eu R_{\infty}(\pi_{0}\mmid\pi)\leq d\log\kappa$
(see \cite[Lemma 18]{LST21structured}). Using $\eu R_{2}\leq\eu R_{\infty}$,
$\clsi(\pi)\leq\alpha^{-1}$, and Corollary~\ref{cor:PS-mixing-LSI}
with $q_{0}=2$, we can conclude that $\PS$ suffices to iterate $\Otilde(\kappa d\log\frac{q}{\veps})$
times. It is also known from \cite{LST21structured} that the backward
step can be implemented via rejection sampling for step size $h=\nicefrac{1}{\beta d}$,
using $\O(1)$ many evaluation queries (in expectation) and one proximal
query per iteration. Combining these together finishes the proof.
\end{proof}

\section{Order-preserving annealing via simpler warm-start generation\label{sec:Unif-warmstart}}

In this section, we streamline prior warm-start generation algorithms
and their proofs. Above all, we interweave an overall analysis with
hypercontractivity under the heat flow, eventually establishing the
complexity of generating $\eu R_{q}$-warmness.

\paragraph{Simpler warm-start generation.}

Let $\K$ be a convex body with $B_{1}(0)\subset\K$ (by translation),
$\E_{\pi}[\abs{\cdot}^{2}]\leq R^{2}$, and $\Lambda=\norm{\cov\pi}$.
We simplify prior annealing approaches in \cite{KV25faster}, moving
along a sequence $\{\mu_{i}\}_{i\in[m]}$ of distributions with $\mu_{i}=\gamma_{\sigma_{i}^{2}}|_{\K}=\pi\gamma_{\sigma_{i}^{2}}$.

This prior algorithm consists of the following several steps: (i)
\emph{Preprocessing}: Truncate $\K$ to the ball of radius $\O(R)$
such that $\pi(\bar{\K})\geq0.99$ where $\bar{\K}:=\K\cap B_{\O(R)}(0)$,
(ii) \emph{Initialization}: Sample from $\mu_{1}=\gamma_{d^{-1}}|_{\bar{\K}}$
by rejection sampling with proposal $\gamma_{d^{-1}}$, (iii) \emph{Annealing}:
While $\sigma_{i}^{2}\lesssim R^{2}$, use $\PS$ to sample from $\mu_{i+1}$
with initial distribution $\mu_{i}$, where $\sigma_{i+1}^{2}=\sigma_{i}^{2}\,(1+\frac{\sigma_{i}}{q^{1/2}R})$
for $q=\polylog\nicefrac{dR}{\veps}$. In particular, there are two
modes: (iii-1) \emph{normal mode}: when $\sigma_{i}^{2}\lesssim qR\Lambda^{1/2}$,
iterate $\PS$ $\Otilde(d^{2}\sigma_{i}^{2})$ times, (iii-2) \emph{faster
mode}: when $\sigma_{i}^{2}\gtrsim qR\Lambda^{1/2}$, iterate $\PS$
$\Otilde(d^{2}R\Lambda^{1/2})$ times. (iv) \emph{Termination}: Run
$\PS$ to sample from $\pi$ with initial distribution $\gamma_{R^{2}}|_{\bar{\K}}$.

We show that the preprocessing step is redundant and also that the
faster mode can be replaced by early-stopping: using $\psgauss$ (\S\ref{subsec:PS-gauss})
below,
\begin{enumerate}
\item Sample from $\mu_{1}=\pi\gamma_{d^{-1}}$ by rejection sampling with
proposal $\gamma_{d^{-1}}$, and set $\bar{\mu}_{1}:=\mu_{1}$.
\item While $\sigma_{i}^{2}\lesssim qR\Lambda^{1/2}\log^{1/2}d$, move from
$\bar{\mu}_{i}$ to $\bar{\mu}_{i+1}$ by $\psgauss$ such that $\eu R_{\Otilde(1)}(\bar{\mu}_{i+1}\mmid\mu_{i+1})=O(1)$,
where 
\[
\sigma_{i+1}^{2}=\sigma_{i}^{2}\,\bpar{1+\frac{\sigma_{i}}{R}}\,.
\]
\end{enumerate}
We need two ingredients to obtain a final guarantee: (1) the query
complexity of $\PS$ for truncated Gaussians ($\psgauss$) and (2)
closeness of two annealing distributions. The former was already studied
in previous work, but we will refine the analysis to obtain an \emph{order-boosting}
guarantee via hypercontractivity in \S\ref{subsec:PS-gauss}. The
latter was proven for $\chi^{2}$-divergence by a localization argument
\cite{CV18Gaussian} and then was extended to $\eu R_{q}$ in \cite{KV25faster}.
We show in \S\ref{subsec:FI-to-closeness} that this is in fact a
simple consequence of \eqref{eq:lsi}. Combining these together, we
establish the query complexity of generating $\eu R_{q}$-warmness.

\subsection{Proximal sampler for truncated Gaussian\label{subsec:PS-gauss}}

We study $\PS$ for $\mu^{X}:=\pi\gamma_{\sigma^{2}}$, Gaussian distributions
truncated to a convex body $\K$ containing $B_{1}(0)$. Its prototype
was already studied in \cite{KZ25Renyi,KV25faster}, and $\psgauss$
can be readily deduced from the general framework of $\PS$.

Starting with $x_{0}\sim\mu_{0}^{X}$, step size $h>0$, threshold
$\tau$, and $s:=\sigma^{2}/(h+\sigma^{2})\in(0,1)$,
\begin{itemize}
\item {[}Forward{]} $y_{k+1}\sim\mu^{Y|X=x_{k}}=\gamma_{h}(\cdot-x_{k})$.
\item {[}Backward{]} $x_{k+1}\sim\mu^{X|Y=y_{k+1}}\propto\gamma_{sh}(\cdot-sy_{k+1})\,\ind_{\K}(\cdot)$.
To this end, draw $x_{k+1}\sim\mc N(sy_{k+1},shI_{d})$ repeatedly
until $x_{k+1}\in\K$. If this rejection loop exceeds $\tau$ trials,
declare failure and restart from scratch with a new sample $x_{0}\sim\mu_{0}^{X}$.
\end{itemize}
As in the uniform-sampling case, its analysis is split into two parts:
a mixing analysis and per-iteration complexity. The failure probability
and per-iteration analysis will be proceeded in a similar fashion,
but the mixing analysis can be improved due to the hypercontractivity
of $\PS$.

\paragraph{Mixing analysis.}

We already established the convergence rate of $\PS$ under \eqref{eq:lsi}
(see Corollary~\ref{cor:PS-mixing-LSI}). It is classical that $\clsi(\mu^{X})\leq\sigma^{2}$
(e.g., use the Bakry--\'Emery criterion, and the convex boundary
term does not harm in satisfying this criterion; see \cite{Klartag09Berry,BJ23spectral}).
Thus, in order to be $\veps$-close to $\mu^{X}$ in $\eu R_{q}$,
for $q_{0}>1$, $\psgauss$ suffices to iterate 
\[
N\lesssim h^{-1}\sigma^{2}\log\frac{q\eu R_{2\wedge q_{0}}(\mu_{0}^{X}\mmid\mu^{X})}{(2\wedge q_{0}-1)\,\veps}\qquad\text{iterations}\,.
\]

\paragraph{Per-step analysis.}

The analysis is analogous to \S\ref{subsec:uniform-query-complexity},
with $\pi_{k}^{Y}$ replaced by $\mu_{k}^{Y}$ and \S2.1 of \cite{KV25faster}
replaced by \S2.2. With the parameter choices $Z=\frac{16NM_{2}}{\eta}$,
$c=\frac{\log\log Z}{10\log Z}$, $h=\frac{c}{d^{2}}$, $\tau=Z^{2}\log^{3}Z$,
the failure probability satisfies $\E_{\mu_{k}^{Y}}[(1-\ell)^{\tau}]\leq\frac{\eta}{N}$,
and enforcing $q\geq1+(1\vee\log\tau)$ yields $\E_{\mu_{k}^{Y}}[\frac{1}{\ell}\wedge\tau]\leq5M_{q_{0}}\log^{4}Z$.
\begin{prop}
\label{prop:gauss-sampling} Let $\K\subset\Rd$ be a convex body
with $B_{1}(0)\subset\K$, given access to a membership oracle for
$\K$, and $\mu\propto\gamma_{\sigma^{2}}\cdot\ind_{\K}$ be the Gaussian
truncated to $\K$. Given $\veps>0$ , $\eta\in(0,1/2)$, $q\geq q_{0}\geq2$,
and an initial distribution $\mu_{0}$ with $M_{q_{0}}=\norm{\nicefrac{\D\mu_{0}}{\D\mu}}_{L^{q_{0}}(\mu)}$,
initialize $\psgauss$ with $\mu_{0}$ and iterate it $N=\Otilde(d^{2}\sigma^{2}\log\frac{q}{\veps}\log\frac{1}{\eta})$
times with $h=(10d^{2}\log Z)^{-1}$ and $\tau=Z^{2}\log^{3}Z$ for
$Z=\frac{16NM_{2}}{\eta}$. If $q_{0}\geq1+\log\tau$ and $M_{q_{0}}\leq10$,
then
\begin{itemize}
\item The sufficient condition on $q_{0}$ is satisfied when $q_{0}\geq2\vee O(\log\frac{N}{\eta})$.
\item With probability at least $1-\eta$, $\psgauss$ iterates $N$ times
without failure. Conditioned on this success, the output $Z\sim\nu$
satisfies $\eu R_{q}(\nu\mmid\mu)\leq\veps+2\log\frac{1}{1-\eta}$,
with $\Otilde(d^{2}\sigma^{2}\log\frac{q}{\veps}\log^{5}\frac{1}{\eta})$
membership queries in expectation.
\end{itemize}
\end{prop}

\begin{proof}
For $q\geq q_{0}\geq2$ and $(M_{2}\leq)M_{q_{0}}\leq10$, the uncapped
$\psgauss$ (with the choice of $h=cd^{-2}=\frac{\log\log Z}{10d^{2}\log Z}$
and $Z\asymp NM_{2}/\eta$) returns a sample $X_{N}$ such that $\eu R_{q}(\law X_{N}\mmid\mu^{X})\leq\veps$
after iterating
\[
h^{-1}\sigma^{2}\log\frac{q\log M_{2}}{\veps}\quad\text{times}\,.
\]
Repeating the similar argument as in the uniform sampling case, the
required number of iterations is of order
\begin{equation}
N=\Otilde\bpar{d^{2}\sigma^{2}\log\frac{q}{\veps}\log\frac{1}{\eta}}\,.\label{eq:psgauss-iter-comp}
\end{equation}

A sufficient condition for ensuring $q_{0}\geq2\vee\log\tau$ is simply
\begin{equation}
1+\log\tau\lesssim\log\frac{N}{\eta}\lesssim\log\Otilde\bpar{\frac{d^{2}\sigma^{2}}{\eta}\log\frac{q}{\veps}\log\frac{1}{\eta}}\,,\label{eq:recursive-ineq-Gaussian}
\end{equation}
Under this condition, the per-step complexity is $O(\log^{4}Z)$,
so the expected number of membership queries for $\psgauss:\eu R_{q_{0}}\to\eu R_{q}$
in one attempt is 
\[
\Otilde\bpar{d^{2}\sigma^{2}\log\frac{q}{\veps}\log^{5}\frac{1}{\eta}}\,,
\]
and the success probability satisfies $\P(S)\geq\eta$. As in the
analysis of $\psunif$, the conditional law of the final output is
distributed as $\hat{\mu}_{N}^{X}$ satisfying $\eu R_{q}(\nu\mmid\mu)\leq\veps+2\log\frac{1}{1-\eta}$.
\end{proof}
Combining all these together, we can prove Theorem~\ref{thm:gauss-warm}
on the query complexity of $\psgauss$.
\begin{proof}
[Proof of Theorem~\ref{thm:gauss-warm}] Take $\eta=\veps<0.01$,
and restart the algorithm upon failure. To invoke the proposition
above, $q_{0}$ should satisfy $q_{0}\geq2\vee O(\log N)$ for $N=\Otilde(d^{2}\sigma^{2}\log\frac{q}{\veps}\log\frac{1}{\veps})$.
To ensure $q_{0}\geq\log\Otilde(d^{2}\sigma^{2}\log\frac{q}{\veps}\log\frac{1}{\veps})(\geq\log N)$,
it suffices to enforce 
\[
q_{0}\geq2\vee\Otilde\bpar{\log d^{2}\sigma^{2}+\log\log\frac{q}{\veps}}\,,
\]
so we will set $q_{0}$ to this as a required base order. We then
restrict to $q\geq2\vee\Otilde(\log d^{2}\sigma^{2}+\log\log\frac{1}{\veps})$
so that $q\geq q_{0}$ holds. Therefore, by Proposition~\ref{prop:gauss-sampling},
the restart version of $\psgauss$ uses $\Otilde(d^{2}\sigma^{2}\log\frac{q}{\veps}\log^{5}\frac{1}{\veps})$
 total queries in expectation, and its output is distributed as $\hat{\mu}_{N}^{X}$
satisfying $\eu R_{q}(\hat{\mu}_{N}^{X}\mmid\pi^{X})\leq3\veps$.
Retaking $\veps\gets\veps/3$, we complete the proof.
\end{proof}

\subsection{Functional inequalities to closeness and early-stopping\label{subsec:FI-to-closeness}}

We will show that two consecutive Gaussians are $\O(1)$-close in
$\eu R_{q}$, finding a connection to \eqref{eq:lsi}. We then show
that annealing can be stopped around $\sigma_{i}^{2}\asymp qR\Lambda^{1/2}\log^{1/2}d$,
much earlier than $\sigma_{i}^{2}\asymp R^{2}$, finding a connection
to \eqref{eq:pi}. Moreover, the second result improves a mathematical
result established in \cite{KV25faster} as an immediate corollary.
We then prove Theorem~\ref{thm:complexity-Renyi-warm} on the query
complexity of our annealing algorithm.

\paragraph{Mollification.}

To avoid any regularity issues, we work with mollifications of logconcave
distributions. As detailed in \cite[\S{C.4}]{evans10partial}, the
\emph{standard mollifier} $\eta\in C^{\infty}(\Rd)$ is defined as
\begin{equation}
\eta(x):=\begin{cases}
C\exp\bpar{\frac{1}{\norm x^{2}-1}} & \norm x\leq1\,,\\
0 & \norm x>1\,,
\end{cases}\label{eq:mollifier}
\end{equation}
where the constant $C$ is chosen to ensure $\int\eta=1$. For $\veps>0$,
one can consider
\[
\eta_{\veps}(x):=\frac{1}{\veps^{d}}\,\eta\bpar{\frac{x}{\veps}}\,,
\]
whose support is bounded by $B_{\veps}(0)$. Then, for locally integrable
$f:\Rd\to\R$, its mollification $f_{\veps}:=f*\eta_{\veps}$ satisfies
$f_{\veps}\in C^{\infty}(\Rd)$ and $f_{\veps}\to f$ almost everywhere
as $\veps\to0$.

\paragraph{Streamlined analysis.}

Recall Scheff\'e's lemma that if a sequence $\pi_{n}$ of probability
densities converges to a probability measure $\pi$ almost everywhere,
then $\pi_{n}$ converges to $\pi$ in $\tv$-distance (thus weakly).
Hence, $\pi_{t}:=\pi*\eta_{t}\to\pi$ weakly as $t\to0$ due to the
property of the mollifier. Then, it is a property of weak convergence
that $\int\gamma_{\sigma^{2}}\,\D\pi_{t}\to\int\gamma_{\sigma^{2}}\,\D\pi$
for any Gaussian density $\gamma_{\sigma^{2}}$. Therefore, probability
measures $\pi_{t}\gamma_{\sigma^{2}}$ converge to $\pi\gamma_{\sigma^{2}}$
almost surely (and thus weakly).

A R\'enyi version of LSI was established in \cite[Lemma 5]{VW23rapid}
that for $q>1$ and smooth $\mu/\nu$,
\begin{equation}
\eu R_{q}(\mu\mmid\nu)\leq\frac{q\clsi(\nu)}{2}\,\msf{RFI}_{q}(\mu\mmid\nu)\,,\tag{\ensuremath{\msf{R\text{-}LSI}}}\label{eq:Renyi-LSI}
\end{equation}
where $(\nicefrac{\mu}{\nu})^{q}\log\nicefrac{\mu}{\nu}$ is continuous,
locally uniform, and integrable. 
\begin{lem}
[Closeness] \label{lem:closeness-annealing} Let $\pi$ be a logconcave
probability measure over $\Rd$ with $\E_{\pi}[\abs{\cdot}^{2}]=R^{2}$.
For $\alpha,\sigma^{2}>0$ and $q>1$, probability measures $\pi\gamma_{\sigma^{2}}$
and $\pi\gamma_{\sigma^{2}(1+\alpha)}$ satisfy
\[
\eu R_{q}(\pi\gamma_{\sigma^{2}}\mmid\pi\gamma_{\sigma^{2}(1+\alpha)})\leq\frac{q^{2}\alpha^{2}R^{2}}{2\sigma^{2}}\,.
\]
In particular, they are $\O(1)$-close in $\eu R_{q}$ if $\alpha\leq\nicefrac{\sigma}{qR}$.
\end{lem}

\begin{proof}
Let $\mu\propto\pi\gamma_{\sigma^{2}}$ and $\nu\propto\pi\gamma_{\sigma^{2}(1+\alpha)}$,
which clearly satisfies $\mu\ll\nu$. For $t>0$, consider mollified
probability measures $\mu_{t}\propto\pi_{t}\gamma_{\sigma^{2}}$ and
$\nu_{t}\propto\pi_{t}\gamma_{\sigma^{2}(1+\alpha)}$. On $\supp\nu_{t}$,
\[
\nabla\log\frac{\D\mu_{t}}{\D\nu_{t}}=\nabla\log\frac{\gamma_{\sigma^{2}}}{\gamma_{\sigma^{2}(1+\alpha)}}=\frac{\alpha x}{(1+\alpha)\,\sigma^{2}}\,,
\]
which admits a smooth extension (as $\frac{\alpha x}{(1+\alpha)\sigma^{2}}$)
to $\Rd$. Then, for some $\delta>0$ independent of $t$,
\begin{equation}
\msf{RFI}_{q}(\mu_{t}\mmid\nu_{t})=q\,\frac{\E_{\nu_{t}}[(\frac{\D\mu_{t}}{\D\nu_{t}})^{q}\,\abs{\nabla\log\frac{\D\mu_{t}}{\D\nu_{t}}}^{2}]}{\E_{\nu_{t}}[(\frac{\D\mu_{t}}{\D\nu_{t}})^{q}]}=\frac{q\alpha^{2}}{\sigma^{4}\,(1+\alpha)^{2}}\,\frac{\E_{\nu_{t}}[(\frac{\D\mu_{t}}{\D\nu_{t}})^{q}\,\abs{\cdot}^{2}]}{\E_{\nu_{t}}[(\frac{\D\mu_{t}}{\D\nu_{t}})^{q}]}=\frac{q\alpha^{2}}{\sigma^{4}\,(1+\alpha)^{2}}\,\frac{\E_{\pi_{t}}[\gamma_{\delta}\,\abs{\cdot}^{2}]}{\E_{\pi_{t}}\gamma_{\delta}}\,.\label{eq:RFI-bound}
\end{equation}
Using \eqref{eq:Renyi-LSI} with $\clsi(\nu_{t})\leq\sigma^{2}\,(1+\alpha)$,
\[
\eu R_{q}(\mu_{t}\mmid\nu_{t})\leq\frac{q\clsi(\nu_{t})}{2}\,\msf{RFI}_{q}(\mu_{t}\mmid\nu_{t})\leq\frac{q^{2}\alpha^{2}}{2\sigma^{2}}\,\frac{\E_{\pi_{t}}[\gamma_{\delta}\,\abs{\cdot}^{2}]}{\E_{\pi_{t}}\gamma_{\delta}}\,.
\]
As noted earlier, $\mu_{t}\to\mu$ and $\nu_{t}\to\nu$ weakly. By
\cite[Theorem 2.34]{AFP00functions}, the lower semi-continuity of
$\eu R_{q}$ implies that $\eu R_{q}(\mu\mmid\nu)\leq\lim_{t\downarrow0}\eu R_{q}(\mu_{t}\mmid\nu_{t})$.
Also, $\lim_{t\downarrow0}\E_{\pi_{t}}[\gamma_{\delta}\,\abs{\cdot}^{2}]\,/\,\E_{\pi_{t}}\gamma_{\delta}=\E_{\pi}[\gamma_{\delta}\,\abs{\cdot}^{2}]\,/\,\E_{\pi}\gamma_{\delta}$
due to the weak convergence of $\pi_{t}\to\pi$. Thus,
\[
\eu R_{q}(\mu\mmid\nu)\leq\frac{q^{2}\alpha^{2}}{2\sigma^{2}}\,\frac{\E_{\pi}[\gamma_{\delta}\,\abs{\cdot}^{2}]}{\E_{\pi}\gamma_{\delta}}\,.
\]

We now need a general fact that for any integrable function $\rho:\Rd\to\R_{+}$,
\begin{equation}
\frac{\int\abs{\cdot}^{2}\,\gamma_{\delta}\,\D\rho\,}{\int\gamma_{\delta}\,\D\rho}\leq\frac{\int\abs{\cdot}^{2}\,\D\rho}{\int\D\rho}=\E_{\rho}[\abs{\cdot}^{2}]\,.\label{eq:cov-decrease}
\end{equation}
Observe that $\cov_{\rho}(f,g)=\half\,\E_{(X,Y)\sim\rho^{\otimes2}}[\{f(X)-f(Y)\}\,\{g(X)-g(Y)\}]$.
As $f(\cdot)=\abs{\cdot}^{2}$ increases in $\abs{\cdot}$ while $g=\gamma_{\delta}$
decreases, $0\geq\cov_{\rho}(f,g)=\E_{\rho}[fg]-\E_{\rho}f\,\E_{\rho}g$,
which proves the claim. Using this fact with $\rho=\pi$, we complete
the proof of the main claim.
\end{proof}

\begin{lem}
[Early-stopping]\label{lem:early-stopping} Let $\pi$ be a logconcave
probability measure over $\Rd$ with $\E_{\pi}[\abs{\cdot}^{2}]=R^{2}$.
For $q\geq2$, if $\sigma^{2}\gtrsim qR\,\norm{\cov\pi}^{1/2}\log^{1/2}d$,
then $\eu R_{q}(\pi\gamma_{\sigma^{2}}\mmid\pi)=\O(1)$.
\end{lem}

\begin{proof}
Let $\mu\propto\pi\gamma_{\sigma^{2}}$ be the probability measure,
and $\mu_{t}$ and $\pi_{t}$ the mollifications of $\mu$ and $\pi$.
Recall from \eqref{eq:Renyi-PI} that
\begin{equation}
\eu R_{q}(\mu_{t}\mmid\pi_{t})\leq\log\Bpar{1-\frac{q\cpi(\pi_{t})}{4}\,\msf{RFI}_{q}(\mu_{t}\mmid\pi_{t})}^{-1}\,.\label{eq:RPI-log}
\end{equation}
By \cite[\S5]{milman09role} or \cite[Theorem 3.13]{CG20poincare},
if two logconcave probability measures $\nu_{1}$ and $\nu_{2}$ satisfy
$\norm{\nu_{1}-\nu_{2}}_{\tv}\leq\veps$ for $\veps\in(0,1)$, then
\[
\cpi(\nu_{1})\lesssim\frac{1\vee\log\frac{1}{1-\veps}}{(1-\veps)^{2}}\,\cpi(\nu_{2})\,.
\]
Since $\pi_{t}\to\pi$ in $\tv$-distance (by Scheff\'e), and $\pi_{t}$
is logconcave (by Pr\'ekopa--Leindler), we have $\limsup_{t\downarrow0}\cpi(\pi_{t})\lesssim\cpi(\pi)$.
Next, emulating the proof of Lemma~\ref{lem:closeness-annealing}
with $\alpha=\infty$, \eqref{eq:RFI-bound} yields
\[
\limsup_{t\downarrow0}\msf{RFI}_{q}(\mu_{t}\mmid\pi_{t})=\frac{q}{\sigma^{4}}\,\frac{\E_{\pi}[\gamma_{\sigma^{2}}\,\abs{\cdot}^{2}]}{\E_{\pi}\gamma_{\sigma^{2}}}\leq\frac{q\,\E_{\pi}[\abs{\cdot}^{2}]}{\sigma^{4}}=\frac{qR^{2}}{\sigma^{4}}\,.
\]
Putting these together and using $\cpi(\pi)\lesssim\norm{\cov\pi}\log d$
\cite{Klartag23log},
\[
\limsup_{t\downarrow0}\frac{q\cpi(\pi_{t})}{4}\,\msf{RFI}_{q}(\mu_{t}\mmid\pi_{t})\lesssim\frac{q^{2}R^{2}\,\norm{\cov\pi}\log d}{\sigma^{4}}\,.
\]
Thus, if $\sigma^{2}\gtrsim qR\norm{\cov\pi}^{1/2}\log^{1/2}d$, then
the RHS can be made smaller than (say) $0.1$. Putting this back to
\eqref{eq:RPI-log}, it follows that $\eu R_{q}(\mu\mmid\pi)=\O(1)$.
\end{proof}

\paragraph{Covariance of strongly logconcave distributions.}

\cite{KV25faster} showed that $\norm{\cov\pi\gamma_{\sigma^{2}}}\lesssim\norm{\cov\pi}$
if $h\gtrsim R\norm{\cov\pi}^{1/2}\log^{2}d\log^{2}\nicefrac{R^{2}}{\Lambda}$
for $R=\E_{\pi}\abs{\cdot}$. This can be improved as an immediate
corollary of the early-stopping result. For the barycenter $b_{\pi}$
of $\pi$, any unit vector $v\in\Rd$, and $\mu=\pi\gamma_{\sigma^{2}}$
with $\sigma^{2}\gtrsim R\,\norm{\cov\pi}^{1/2}\log^{1/2}d$,
\[
\Bpar{\int\bpar{v^{\T}(x-b_{\pi})}^{2}\,\D\mu(x)}^{2}\leq\Bnorm{\frac{\D\mu}{\D\pi}}_{L^{2}(\pi)}^{2}\int\bpar{v^{\T}(x-b_{\pi})}^{4}\,\D\pi(x)\lesssim\Bnorm{\frac{\D\mu}{\D\pi}}_{L^{2}(\pi)}^{2}\norm{\cov\pi}^{2}\,,
\]
where we used the reverse H\"older in the last inequality. Thus,
$\norm{\cov\mu}\lesssim\norm{\frac{\D\mu}{\D\pi}}_{L^{2}(\pi)}\norm{\cov\pi}\,.$
Since $\eu R_{2}(\mu\mmid\pi)=\O(1)$, we have
\[
\norm{\cov\mu}\lesssim\norm{\cov\pi}\,.
\]

\subsection{R\'enyi-warmness generation\label{subsec:warm-analysis-uniform}}

Now that all pieces are ready, we show that an $\O(1)$-warm start
in $\eu R_{q}$ (equivalently, $L^{q}=O(1)$) can be obtained by using
$\Otilde(qd^{2}R^{3/2}\Lambda^{1/4})$ membership queries, where $B_{1}(0)\subset\K$,
$\Lambda=\norm{\cov\pi}$, and $R^{2}=\E_{\pi}[\abs{\cdot}^{2}]$.

\paragraph{Algorithm.}

Given $\veps_{\unif}=\eta_{\unif}\in(0,1/10)$ and $q\geq2$, we pick
a ``base'' order $q_{0}(\leq q)$ that will be preserved through annealing.
To get a failure-free algorithm, we restart the following algorithm
from scratch if $\psgauss$ declares failure in the middle: for $m$
the number of total phases,
\begin{enumerate}
\item Sample from $\mu_{1}=\pi\gamma_{d^{-1}}$ by rejection sampling with
proposal $\gamma_{d^{-1}}$, and set $\bar{\mu}_{1}:=\mu_{1}$.
\item While $\sigma_{i}^{2}\lesssim qR\Lambda^{1/2}\log^{1/2}d$, move from
$\bar{\mu}_{i}$ to $\bar{\mu}_{i+1}$ by $\psgauss$ such that $\eu R_{q_{0}}(\bar{\mu}_{i+1}\mmid\mu_{i+2})\leq1$,
where 
\[
\sigma_{i+1}^{2}=\sigma_{i}^{2}\,\bpar{1+\frac{\sigma_{i}}{q_{0}R}}\,.
\]

\begin{enumerate}
\item For $i<m$, run $\psgauss:\eu R_{q_{0}}\to\eu R_{2q_{0}}$ with $N_{i}=\Otilde(d^{2}\sigma_{i}^{2}\log q_{0}\log m)$,
where we set $\veps_{i}=1/10$ and $\eta_{i}=(10m)^{-1}$.
\item For $i=m$, run $\psgauss:\eu R_{q_{0}}\to\eu R_{q}$ with $N_{i}=\Otilde(qd^{2}R\Lambda^{1/2}\log m)$,
where we set $\veps_{i}=1/10$ and $\eta_{i}=(10m)^{-1}$.
\end{enumerate}
\end{enumerate}
Note that $N_{i}$ is chosen to ensure the target accuracies of $\psgauss$
in each phase (see \eqref{eq:psgauss-iter-comp}). As seen shortly,
given $q\geq\Otilde(1)$, we can take $q_{0}=O(\log q)$ so that the
annealing has ``linear'' dependence on $q$.\footnote{If we na\"ively attempt to preserve $q$-warmness, then the last
doubling of $\sigma_{i}^{2}$ incurs (roughly) $qd^{2}\sigma_{m}R\lesssim q^{3/2}d^{2}R^{3/2}\Lambda^{1/4}$
queries (i.e., $q^{3/2}$-dependence).}

\paragraph{Choice of parameters.}

As $m=\Otilde(q_{0}d^{1/2}R\log q)$, we have $N_{i}\leq\Otilde(qd^{2}R\Lambda^{1/2})=:N_{\max}(q)$.
Also, $N_{\unif}:=\Otilde(qd^{2}\Lambda\log\frac{1}{\eta_{\unif}}\log\frac{1}{\veps_{\unif}})$.
Recall that $q_{0}$ and $q$ should satisfy sufficient conditions
to invoke the established sampling guarantees of $\psgauss$ and $\psunif$
(Proposition~\ref{prop:unif-sampling} and~\ref{prop:gauss-sampling}).
Namely, $q\geq2\vee O(\log\frac{N_{\unif}}{\eta_{\unif}})$, $q_{0}\geq2\vee O(\log mN_{\max})$,
and $q\geq q_{0}$. The second condition is implied by $q_{0}\gtrsim\log\Otilde(q_{0}qd^{5/2}R^{2}\Lambda^{1/2}\log q)$,
so we will set $q_{0}=\Otilde(\log qdR\Lambda^{1/4})$. To ensure
$q\geq q_{0}$, we will restrict $q\geq\Otilde(\log dR\Lambda^{1/4})$.
Now, using \eqref{eq:q-suff-unif} for the first condition, we will
assume that
\begin{equation}
q\geq2\vee\Otilde(\log dR\Lambda^{1/4})\vee\Otilde\bpar{\log\frac{d^{2}\Lambda}{\eta_{\unif}}+\log\log\frac{1}{\veps_{\unif}}}\,,\quad q_{0}=\Otilde(\log qdR\Lambda^{1/4})\,.\label{eq:choice-q-q0}
\end{equation}
By abuse of notation, we sometime write $q\geq\Otilde(1)$ and $q_{0}=\Otilde(\log q)$.
We note that $\eta_{i}$ is chosen to bound by $1/10$ the probability
that the failure condition is met until the termination of the annealing.

\paragraph{High-level idea of analysis.}

We consider an annealing process across a sequence of target distributions
$\{\mu_{i}\}_{i=1}^{m}$, where each phase invokes the capped $\psgauss$
designed to move from (approximately) $\mu_{i}$ to (approximately)
$\mu_{i+1}$. Each phase may declare failure due to a capped inner
rejection sampler. 

\noindent Let $S^{(i)}$ be the event that phase $i$ succeeds (no
failure of $\psgauss$ within that phase), and let $S:=\cap_{i=1}^{m-1}S^{(i)}$
be the global success event. Let $C^{(i)}\ge0$ be the query complexity
incurred by phase $i$, and define the one-attempt total cost
\[
C:=\sum_{i=1}^{m-1}\ind\bbrace{\bigcap_{j=1}^{i-1}S^{(j)}}\,C^{(i)}\,.
\]
Clearly, $\E C\leq\sum_{i=1}^{m-1}\E C^{(i)}$, so it suffices to
upper bound $\E C^{(i)}$. Similarly, letting $F^{(i)}$ denote the
event that phase $i$ fails and $q^{(i)}$ denote its cap-hit probability
inside phase $i$, we have
\[
\P(F^{(i)})=\E\bbrack{\ind\bbrace{\cap_{j=1}^{i-1}S^{(j)}}\,q^{(i)}}\leq\E q^{(i)}\,,\qquad\P(S^{c})\leq\sum_{i=1}^{m-1}\E q^{(i)}\,,
\]
which allows us to bound the overall failure probability via unconditional
phase-wise estimates (under the uncapped intermediate laws used in
the analysis).

Similar to the analysis of the restarting version of $\psunif$, a
restart-until-success annealing repeats independent full attempts
(including resampling the initial input from $\mu_{1}$) until $S$
occurs. If $\P(S)\ge1-\eta$, then the expected number of attempts
is at most $(1-\eta)^{-1}$ and 
\[
\E C_{\text{restart}}=\frac{\E C}{\P(S)}\leq\frac{\E C}{1-\eta}\,.
\]
In summary, we first establish (phase-wise) bounds on unconditional
failure probabilities and unconditional expected costs, combine them
to control $\P(S)$ and $\E C$ for one attempt, and then translate
these into guarantees for the returned-sample distribution $\nu$
and for the restart wrapper.
\begin{proof}
[Proof of Theorem~\ref{thm:complexity-Renyi-warm}] Pick $q$ and
$q_{0}$ as in \eqref{eq:choice-q-q0}. The initialization by rejection
sampling requires $\O(1)$ queries in expectation due to 
\[
\frac{\int_{\Rd}e^{-d\,\abs x^{2}/2}\,\D x}{\int_{\K}e^{-d\,\abs x^{2}/2}\,\D x}=\frac{1}{\gamma_{d^{-1}}(\K)}\leq\frac{1}{\gamma_{d^{-1}}(B_{1})}\lesssim1\,.
\]

In the annealing phase, suppose that we are given $\bar{\mu}_{i}$
such that $\eu R_{2q_{0}}(\bar{\mu}_{i}\mmid\mu_{i})\leq\frac{1}{10}$.
By a weak triangle inequality for R\'enyi divergence (with $\lda=2/3$),
\[
\eu R_{q_{0}}(\bar{\mu}_{i}\mmid\mu_{i+1})\leq1.1\eu R_{2q_{0}}(\bar{\mu}_{i}\mmid\mu_{i})+\eu R_{2q_{0}}(\mu_{i}\mmid\mu_{i+1})\leq\frac{1.1}{10}+\half\leq1\,,
\]
where the last inequality follows from Lemma~\ref{lem:closeness-annealing}.
By Theorem~\ref{thm:gauss-warm}, from an $\O(1)$-warmness in $\eu R_{q_{0}}$-divergence
(precisely, $\eu R_{q_{0}}(\bar{\mu}_{i}\mmid\mu_{i+1})\leq1$), $\psgauss$
needs $\Otilde(d^{2}\sigma_{i}^{2}\log^{2}q_{0})$ queries to sample
$X_{i+1}^{*}\sim\bar{\mu}_{i+1}$ such that $\eu R_{2q_{0}}(\bar{\mu}_{i+1}\mmid\mu_{i+1})\leq\frac{1}{10}$.
Also, doubling of given $\sigma_{i}^{2}$ requires $\O(q_{0}R/\sigma_{i})$
phases, so each doubling needs $\Otilde(q_{0}d^{2}\sigma_{i}R)$ queries
in total. Since there are at most $\O(\log qdR\Lambda^{1/2})$ doublings,
and $\sigma_{i}^{2}\leq\sigma_{m}^{2}\asymp qR\Lambda^{1/2}\log^{1/2}d$,
we can obtain a sample $X^{*}\sim\bar{\mu}_{m-1}$ such that $\eu R_{2q_{0}}(\bar{\mu}_{m-1}\mmid\mu_{m-1})\leq\frac{1}{10}$,
using
\[
\O\bpar{q_{0}q^{1/2}d^{2}R^{3/2}\Lambda^{1/4}\log^{1/4}d\times\log qdR\Lambda^{1/2}}=\Otilde(q^{1/2}d^{2}R^{3/2}\Lambda^{1/4})\quad\text{queries}\,.
\]
By hypercontractivity of $\PS$ (Corollary~\ref{cor:PS-mixing-LSI}),
iterating $\psgauss$ started from $X^{*}$ towards $\mu_{m}$ $\Otilde(d^{2}\sigma_{m}^{2}\log q)$
times leads to a new sample $X'$ such that $\eu R_{2q}(\law X'\mmid\mu_{m})\leq\frac{1}{10}$.
Hence, the entire annealing has used
\[
\Otilde\bpar{q^{1/2}d^{2}R^{3/2}\Lambda^{1/4}+qd^{2}R\Lambda^{1/2}}\leq\Otilde(qd^{2}R^{3/2}\Lambda^{1/4})\quad\text{queries (due to }R^{2}\geq\Lambda\text{)}\,.
\]
Note that the early-stopping result (Lemma~\ref{lem:early-stopping})
and the weak triangle inequality imply $\eu R_{q}(\law X'\mmid\pi)\leq1$
as desired.

Lastly, restarting the annealing (in case of failure) increases the
total query complexity multiplicatively by $10/9$, and the conditional
law on success has an extra bias at most $1$ in $\eu R_{\infty}$.
Hence, by $\eu R_{q}(\law(X'\mid\text{success})\mmid\pi)\leq2$.
\end{proof}
We now establish the query complexity of convex-body sampling from
scratch.
\begin{proof}
[Proof of Corollary~\ref{cor:unif-comp-scratch}] We set $\eta_{\unif}=\veps_{\unif}=\veps<1/100$
and consider $q\geq\Otilde(1)$. By Theorem~\ref{thm:unif-warm},
the expected number of restarts of $\psunif$ is at most $O(1)$,
so the warm-start generation will be repeated at most $O(1)$ times.
Combining the complexity results in Theorem~\ref{thm:complexity-Renyi-warm}
and~\ref{thm:unif-warm}, the total complexity of logconcave sampling
from scratch is 
\[
\Otilde\bpar{qd^{2}R^{3/2}\Lambda^{1/4}+qd^{2}\Lambda\log^{6}\frac{1}{\veps}}\,.
\]
When $1\leq q\leq2\vee\Otilde(\log\frac{dR}{\veps})=:q_{*}$, we generate
the $\eu R_{q_{*}}$-warmness and use the monotonicity of $\eu R_{q}$.
\end{proof}

\section{Logconcave sampling under an evaluation oracle\label{sec:complexity-LC}}

To extend our approach to logconcave distributions, we generalize
what we did for uniform sampling, streamlining algorithms in \cite{KV25faster}
and obtaining a guarantee for $\eu R_{q}$-warmness generation.

We begin with analysis for order-preserving logconcave sampling from
an $O(1)$-warm start in \S\ref{subsec:LC-sampling-from-warmstart}
and move to an order-preserving annealing algorithm for warm-start
generation in \S\ref{subsec:LC-warmness-generation}. Overall, we
follow a similar route for the analysis as in the uniform sampling
case.

\subsection{Zeroth-order logconcave sampling with balanced guarantees\label{subsec:LC-sampling-from-warmstart}}

We use the exponential-lifting approach \eqref{eq:exp-lifting} introduced
in \cite{KV25sampling}, obtaining an order-preserving guarantee under
an $O(1)$-warm start. For $z=(x,t)\in\R^{d+1}$ and $y\in\R^{d+1}$,
$\psexp$ alternates $y\sim\pi^{Y|Z=z}=\mc N(z,hI_{d+1})$ and $z\sim\pi^{Z|Y=y}\propto\mc N(y-h\alpha,hI_{d+1})|_{\K}$
for $\alpha=de_{d+1}$, where the second step is implemented via rejection
sampling, using the proposal $\mc N(y-h\alpha,hI_{d+1})$ with threshold
$\tau$. 

\paragraph{A mixing analysis.}

Assume that $\pi_{0}^{X}$ is $M_{q}$-warm with respect to $\pi^{X}$.
Note that one can generate a sample $(X',T')$ whose joint law is
also $M_{q}$-warm with respect to $\pi^{X,T}$ by first drawing $X'\sim\pi_{0}^{X}$
and then sampling $T'\sim\pi^{T|X=X'}\propto e^{-dt}|_{[V(X')/d,\infty)}$.

Since the convergence rate of $\PS$ was already given in \S\ref{subsec:mixing},
we can just reuse Proposition~\ref{prop:mixing-general}. The only
adjustment is that we now require a bound on $\cpi(\pi^{X,T})$, rather
than $\cpi(\pi^{X})$, as we run $\PS$ to sample from $\pi^{X,T}$.
Since $\cpi(\pi^{X,T})\lesssim\norm{\cov\pi^{X,T}}\log d$ due to
log-concavity of $\pi^{X,T}$, and in \cite[Lemma 2.5]{KV25sampling},
$\norm{\cov\pi^{X,T}}$ is bounded as
\[
\norm{\cov\pi^{X,T}}\leq2\,(\norm{\cov\pi^{X}}+160)\lesssim\norm{\cov\pi^{X}}\vee1\,,
\]
we have $\cpi(\pi^{X,T})\lesssim(1\vee\norm{\cov\pi^{X}})\log d$.
Hence, by Lemma~\ref{lem:psunif-mixing}, $\psexp$ with step size
$h$ can ensure $(\eu R_{q}(\pi_{N}^{X}\mmid\pi^{X})\leq)\eu R_{q}(\pi_{N}^{X,T}\mmid\pi^{X,T})\leq\veps$
after 
\[
N\lesssim q\,\bpar{1\vee h^{-1}\,(1\vee\norm{\cov\pi^{X}})\log d}\log\frac{M_{q}}{\veps}\quad\text{iterations}\,.
\]

\paragraph{Per-step analysis.}

It follows from \S5.1 in \cite{KV25faster} that with the parameter
choices $S=\frac{16NM_{2}}{\eta}$, $c=\frac{(\log\log S)^{2}}{13^{2}\log S}$,
$\tau=Z^{2}\log^{2}Z$, and $h=c/(13d)^{2}$, the failure probability
satisfies $\E_{\pi_{k}^{Y}}[(1-\ell)^{\tau}]\leq\frac{\eta}{N}$.
The per-step complexity is bounded as follows: setting $q=1+\alpha$
with $\alpha\ge1\vee\log\tau$, it follows that $\E_{\pi_{k}^{Y}}[\frac{1}{\ell}\wedge\tau]\leq5M_{q}\log S$.
\begin{prop}
\label{prop:LC-sampling}For a convex function $V:\Rd\to\R$, let
$\pi\propto e^{-V}$ be the logconcave distribution over $\Rd$ with
$B_{1}\subset\msf L_{\pi,g}$ and $\Lambda=\norm{\cov\pi}$, given
access to an evaluation oracle for $V$. Given $\veps>0$, $\eta\in(0,1/2)$,
$q\geq2$, and an initial distribution $\pi_{0}$ with $M_{q}=\norm{\nicefrac{\D\pi_{0}}{\D\pi}}_{L^{q}(\pi)}$,
initialize $\psexp$ with $\pi_{0}$ and iterate it $N=\Otilde(qd^{2}\,(1\vee\Lambda)\log\frac{1}{\veps}\log\frac{1}{\eta})$
times with $h=(13d^{2}\log S)^{-1}$ and $\tau=S^{2}\log^{2}S$ for
$S=\frac{16NM_{2}}{\eta}$. If $q\geq1+\log\tau$ and $M_{q}\leq10$,
then 
\begin{itemize}
\item The sufficient condition on $q$ is satisfied when $q\geq2\vee O(\log\frac{N}{\eta})$.
\item With probability at least $1-\eta$, $\psexp$ iterates $N$ times
without failure. Conditioned on this success, the $X$-output $X^{*}\sim\nu$
satisfies $\eu R_{q}(\nu\mmid\pi)\leq\veps+2\log\frac{1}{1-\eta}$,
with $\Otilde(qd^{2}\,(1\vee\Lambda)\log\frac{1}{\veps}\log^{2}\frac{1}{\eta})$
evaluation queries in expectation.
\end{itemize}
\end{prop}

Its proof is almost similar to the analysis of $\psunif$ in Proposition~\ref{prop:unif-sampling}.
Putting all these together, we can establish the query complexity
of $\psexp$ with balanced guarantees.
\begin{proof}
[Proof of Theorem~\ref{thm:lc-warm}] Take $\veps=\eta<0.01$, and
restart $\psexp$ upon failure. Note that the expected number of restarts
is at most $0.02$. In order to ensure $q\geq2\vee O(\log N)$ for
$N=\Otilde(qd^{2}\,(1\vee\Lambda)\log^{2}\frac{1}{\veps})$, it suffices
to restrict 
\[
q\geq2\vee\Otilde\bpar{\log(d^{2}\,(1\vee\Lambda))+\log\log\frac{1}{\veps}}\,.
\]
By the proposition above, the desired claims follow in a similar way
to the proof of Theorem~\ref{thm:unif-warm}.
\end{proof}

\subsection{R\'enyi-warmness generation for logconcave distributions\label{subsec:LC-warmness-generation}}

We streamline a warm-start algorithm in \cite{KV25faster}, which
generates a warm start for the lifted distribution $\pi\propto e^{-dt}|_{\K}$
by annealing through distributions $\mu_{\sigma^{2},\rho}(x,t)\propto\exp(-\frac{1}{2\sigma^{2}}\,\abs x^{2}-\rho t)\,\ind_{\bar{\K}}(x,t)$
for some $\bar{\K}$. In comparison with our warm-start algorithm
for uniform distributions, we can also stop the $\sigma^{2}$-annealing
earlier than \cite{KV25faster}, but we could not remove the preprocessing
step therein, in order to ensure a finite LSI constant of the annealing
distributions. Also, since the complexity of sampling from $\mu_{\sigma^{2},\rho}$
(by $\psann$ in \S\ref{subsec:complexity-tilted-Gaussian}) is roughly
$d^{2}(\sigma^{2}\vee1)$ (from an $\O(1)$-warm start), which has
the additive $d^{2}$ queries, we cannot help but following a convoluted
approach in \cite{KV25faster}.

\paragraph{Warm-start generation.}

For $l=\log4e$, one may assume that $x_{0}=0$ (by translation) and
consider
\begin{equation}
\bar{\K}:=\K\cap\{B_{Rl}(0)\times[-21,13l-6]\}\,,\label{eq:convex-truncation}
\end{equation}
which satisfies $\pi(\R^{d+1}\backslash\bar{\K})\leq1/2$ (see \cite[Convex truncation and \S3.1]{KV25sampling}).
Thus, we can ensure that $\bar{\pi}=\pi|_{\bar{\K}}$ is $\log2$-close
to $\pi$ in $\eu R_{\infty}$. Below, we let $D:=Rl$ and $\mu_{i}:=\mu_{\sigma_{i}^{2},\rho_{i}}$,
and denote by $\bar{\mu}_{i}$ the law of an actual sample such that
$\eu R_{\Otilde(1)}(\bar{\mu}_{i}\mmid\mu_{i})=O(1)$.
\begin{enumerate}
\item Sample from $\mu_{0}\propto e^{-d\,\abs x^{2}/2}|_{\bar{\K}}$ by
rejection sampling with proposal $\gamma_{d^{-1}}\otimes\text{Unif}\,([-21,13l-6])$).
\item {[}Phase I: $\sigma^{2}$-warming{]} While $d^{-1}\leq\sigma_{i}^{2}\leq1$,
move from $\bar{\mu}_{i}$ to $\bar{\mu}_{i+1}$ by $\psann$, where
\[
\sigma_{i+1}^{2}=\sigma_{i}^{2}\,\bpar{1+\frac{1}{d^{1/2}}}\,.
\]
At the end of Phase I, run $\psann$ toward $\mu_{i+1}\propto e^{-\abs x^{2}/2-t}|_{\bar{\K}}$,
obtaining $\bar{\mu}_{i+1}$.
\item {[}Phase II: $\rho$-annealing{]} While $1\leq\rho_{i}\leq d$ (and
$\sigma_{i}^{2}\approx1$), move from $\bar{\mu}_{i}$ to $\bar{\mu}_{i+1}$
by $\psann$, where
\[
\sigma_{i+1}^{2}=\sigma_{i}^{2}\,\bpar{1+\frac{1}{d^{1/2}}}^{-1}\quad\&\quad\rho_{i+1}=\min\bpar{d,\rho_{i}\,\bpar{1+\frac{1}{d^{1/2}}}}\,.
\]
This is followed by inner annealing: until $\sigma_{i+1}^{2}\leq1$,
run $\psann$ toward \emph{new} $\mu_{i+1}$ defined by
\[
\sigma_{i+1}^{2}\gets\sigma_{i+1}^{2}\,\bpar{1+\frac{\sigma_{i+1}}{D}}\,,
\]
\item {[}Phase III: $\sigma^{2}$-annealing{]} While $\sigma_{i}^{2}\lesssim qD\,(\Lambda^{1/2}\vee1)\log^{1/2}d$
(and $\rho_{i}=d$), move from $\bar{\mu}_{i}$ to $\bar{\mu}_{i+1}$
by $\psann$, where 
\[
\sigma_{i+1}^{2}=\sigma_{i}^{2}\,\bpar{1+\frac{\sigma_{i}}{D}}\,.
\]
\end{enumerate}

\subsubsection{Proximal sampler for tilted Gaussian\label{subsec:complexity-tilted-Gaussian}}

Now consider the sampler for the annealing distributions $\mu_{i}$
proposed in \cite{KV25sampling}. For $v:=(x,t)\in\Rd\times\R$ and
$w:=(y,s)\in\Rd\times\R$, consider the augmented target given by
\[
\mu^{V,W}(v,w)\propto\exp\bpar{-\frac{1}{2\sigma^{2}}\,\abs x^{2}-\rho t-\frac{1}{2h}\,\abs{w-v}^{2}}\big|_{\bar{\K}}\,,
\]
and $\psann$ with step size $h$ and threshold $\tau$ alternates
the following: for $r:=\frac{\sigma^{2}}{h+\sigma^{2}}<1$, $y_{r}:=ry$,
and $h_{r}:=rh$,
\begin{itemize}
\item {[}Forward{]} $w_{k+1}\sim\mu^{W|V=v_{k}}=\mc N(v_{k},hI_{d+1})$.
\item {[}Backward{]} $v_{k+1}\sim\mu^{V|W=w_{k+1}}\propto\bbrack{\mc N(ry_{k+1},h_{r}I_{d})\otimes\mc N(s_{k+1}-\rho h,h)}\big|_{\bar{\K}}$.
To this end, draw $v_{k+1}\sim\mc N(ry_{k+1},h_{r}I_{d})\otimes\mc N(s_{k+1}-\rho h,h)$
until $v_{k+1}\in\bar{\K}$. If this rejection loop exceeds $\tau$
trials, declare failure.
\end{itemize}
We establish its query complexity similarly as before.
\begin{prop}
\label{prop:ann-sampling} Consider the annealing distribution $\mu=\mu_{\sigma^{2},\rho}$
above, given access to an evaluation oracle for $V$. Given $\veps>0$,
$\eta\in(0,1/2)$, $q\geq q_{0}\geq2$, and an initial distribution
$\mu_{0}$ with $M_{q_{0}}=\norm{\nicefrac{\D\mu_{0}}{\D\mu}}_{L^{q_{0}}(\mu)}$,
initialize $\psann$ with $\mu_{0}$ and iterate it $N=\Otilde(d^{2}\,(\sigma^{2}\vee1)\log\frac{q}{\veps}\log\frac{1}{\eta})$
times with $h=(24^{2}d^{2}\log S)^{-1}$ and $\tau=2S^{2}\log^{2}S$
for $S=\frac{16NM_{2}}{\eta}$. If $q_{0}\geq1+\log\tau$ and $M_{q_{0}}\leq10$,
then
\begin{itemize}
\item The sufficient condition on $q_{0}$ is satisfied when $q_{0}\geq2\vee O(\log\frac{N}{\eta})$.
\item With probability at least $1-\eta$, $\psann$ iterates $N$ times
without failure. Conditioned on this success, the output $Z\sim\nu$
satisfies $\eu R_{q}(\nu\mmid\mu)\leq\veps+2\log\frac{1}{1-\eta}$,
with $\Otilde(d^{2}\,(\sigma^{2}\vee1)\log\frac{q}{\veps}\log^{2}\frac{1}{\eta})$
evaluation queries in expectation.
\end{itemize}
\end{prop}

\begin{proof}
Overall, we follow the analysis of $\psgauss$ as in Proposition~\ref{prop:gauss-sampling}.
Using the Bakry--\'Emery criterion and bounded perturbation, one
can show that $\clsi(\mu)\lesssim\sigma^{2}\vee1$ (see \cite[Lemma 3.3]{KV25sampling}).
For $q\geq q_{0}>1$, when $\psann$ starts from $\mu_{0}^{V}$, by
Corollary~\ref{cor:PS-mixing-LSI}, in order to be $\veps$-close
to $\mu^{V}$ in $\eu R_{q}$, for $q_{0}>1$, $\psann$ suffices
to iterate
\[
h^{-1}\,(\sigma^{2}\vee1)\log\frac{q\eu R_{2\wedge q_{0}}(\mu_{0}^{V}\mmid\mu^{V})}{(2\wedge q_{0}-1)\,\veps}\qquad\text{iterations}\,.
\]
Also, \cite{KV25faster} showed that under $c=\frac{(\log\log S)^{2}}{4\cdot24^{2}\log S}$,
$h=\frac{c}{24^{2}d^{2}}$, and $\tau=2S^{2}\log^{2}S$ for $S=\frac{16NM_{2}}{\eta}$,
the per-step query complexity is bounded by $O(M_{q}\log S)$, and
the failure probability is bounded by $\eta/N$.

For $q\geq q_{0}\geq2$ and $(M_{2}\leq)M_{q_{0}}\leq10$, the required
number $N$ of iterations (of $\psann$) for $\veps$-closeness in
$\eu R_{q}$ should satisfy $N\gtrsim h^{-1}\,(\sigma^{2}\vee1)\log\frac{q}{\veps}$.
Thus, it suffices to take 
\[
N=\Otilde\bpar{d^{2}\,(\sigma^{2}\vee1)\log\frac{q}{\veps}\log\frac{1}{\eta}}\,.
\]
Also, a sufficient condition for ensuring $q_{0}\geq2\vee\log\tau$
is $q_{0}\geq2\vee O(\log\frac{N}{\eta})$. Under this condition,
the per-step complexity is $O(\log S)$, so the expected number of
queries for $\psann:\eu R_{q_{0}}\to\eu R_{q}$ is 
\[
\Otilde\bpar{d^{2}\,(\sigma^{2}\vee1)\log\frac{q}{\veps}\log^{2}\frac{1}{\eta}}\,,
\]
and the success probability is at least $1-\eta$.
\end{proof}

\subsubsection{Query complexity of R\'enyi-warmness generation\label{subsec:general-LC-warm-generation}}

Overall, we follow the same analysis of the annealing for uniform
distributions; see `High-level idea of analysis' in \S\ref{subsec:warm-analysis-uniform}.

\paragraph{Algorithm.}

Given $\veps_{\exp}=\eta_{\exp}\in(0,1/10)$ and $q\geq2$, we pick
a base order $q_{0}(\leq q)$ that will be preserved through annealing,
and restart the following algorithm from scratch if $\psann$ declares
failure in the middle.

Recall $D=Rl$ (for a large constant $l>0$ such that $\eu R_{\infty}(\bar{\pi}\mmid\pi)\leq2$)
and $\mu_{i}=\mu_{\sigma_{i}^{2},\rho_{i}}$, and $\bar{\mu}_{i}$
is the law of an actual sample such that $\eu R_{2q_{0}}(\bar{\mu}_{i}\mmid\mu_{i})\leq\frac{1}{10}$.
For $m$ the number of total phases, $\psann$ is iterated $N_{i}$
times with $\veps_{i}=1/10$ and $\eta_{i}=(10m)^{-1}$ in each phase
$i\in[m]$.
\begin{enumerate}
\item Sample from $\mu_{0}\propto e^{-d\,\abs x^{2}/2}|_{\bar{\K}}$ by
rejection sampling with proposal $\gamma_{d^{-1}}\otimes\text{Unif}\,([-21,13l-6])$).
\item {[}Phase I: $\sigma^{2}$-warming{]} While $d^{-1}\leq\sigma_{i}^{2}\leq1$,
move from $\bar{\mu}_{i}$ to $\bar{\mu}_{i+1}$ by $\psann:\eu R_{q_{0}}\to\eu R_{2q_{0}}$,
where 
\[
\sigma_{i+1}^{2}=\sigma_{i}^{2}\,\bpar{1+\frac{1}{(4q_{0}d)^{1/2}}}\,.
\]
At the end of Phase I, run $\psann:\eu R_{q_{0}}\to\eu R_{2q_{0}}$
toward $\mu_{i+1}\propto e^{-\abs x^{2}/2-t}|_{\bar{\K}}$, obtaining
$\bar{\mu}_{i+1}$.
\item {[}Phase II: $\rho$-annealing{]} While $1\leq\rho_{i}\leq d$ (and
$\sigma_{i}^{2}\approx1$), move from $\bar{\mu}_{i}$ to $\bar{\mu}_{i+1}$
by $\psann:\eu R_{q_{0}}\to\eu R_{2q_{0}}$, where
\[
\sigma_{i+1}^{2}=\sigma_{i}^{2}\,\bpar{1+\frac{1}{(4q_{0}d)^{1/2}}}^{-1}\quad\&\quad\rho_{i+1}=\min\bpar{d,\rho_{i}\,\bpar{1+\frac{1}{(4q_{0}d)^{1/2}}}}\,.
\]
This is followed by inner annealing: until $\sigma_{i+1}^{2}\leq1$,
run $\psann:\eu R_{q_{0}}\to\eu R_{2q_{0}}$ toward \emph{new} $\mu_{i+1}$
defined by
\[
\sigma_{i+1}^{2}\gets\sigma_{i+1}^{2}\,\bpar{1+\frac{\sigma_{i+1}}{q_{0}D}}\,,
\]
\item {[}Phase III: $\sigma^{2}$-annealing{]} While $\sigma_{i}^{2}\lesssim qD\,(\Lambda^{1/2}\vee1)\log^{1/2}d$
(and $\rho_{i}=d$), move from $\bar{\mu}_{i}$ to $\bar{\mu}_{i+1}$
by $\psann:\eu R_{q_{0}}\to\eu R_{2q_{0}}$, where 
\[
\sigma_{i+1}^{2}=\sigma_{i}^{2}\,\bpar{1+\frac{\sigma_{i}}{q_{0}D}}\,.
\]
When $i=m-1$ (i.e., the last phase), we move from $\bar{\mu}_{m-1}$
to $\bar{\mu}_{m}$ by $\psann:\eu R_{q_{0}}\to\eu R_{q}$.
\end{enumerate}

\paragraph{Choice of parameters.}

Note that the number $m$ of inner phases is bounded as
\[
m\lesssim(q_{0}d)^{1/2}\log d+(q_{0}d)^{1/2}\log d\times\frac{q_{0}^{1/2}D}{d^{1/2}}+q_{0}D\log\bpar{qD\,(\Lambda^{1/2}\vee1)}=\Otilde\bpar{q_{0}d^{1/2}D\log q(1\vee\Lambda^{1/2})}\,.
\]
Hence, $N_{i}=\Otilde(d^{2}\,(\sigma_{i}^{2}\vee1)\log^{2}q_{0}m)\leq\Otilde(qd^{2}D\,(\Lambda^{1/2}\vee1))=:N_{\max}$
and $N_{\exp}:=\Otilde(qd^{2}\,(\Lambda\vee1)\log\frac{1}{\eta_{\exp}}\log\frac{1}{\veps_{\exp}})$. 

We now find sufficient conditions for $q_{0}$ and $q$ to invoke
the established sampling guarantees of $\psann$ and $\psexp$ (see
Proposition~\ref{prop:LC-sampling} and~\ref{prop:ann-sampling});
namely, $q\geq2\vee O(\log\frac{N_{\exp}}{\eta_{\exp}})$, $q_{0}\geq2\vee O(\log mN_{\max})$,
and $q\geq q_{0}$.

Repeating a similar computation as in the uniform-sampling case, we
will take $q_{0}=2\vee\Otilde(\log\{qdD\,(\Lambda\vee1)\})$ and 
\begin{equation}
q\gtrsim2\vee\Otilde\bpar{\log\frac{d^{2}\,(\Lambda\vee1)}{\eta_{\exp}}+\log\log\frac{1}{\veps_{\exp}}}\vee\Otilde\bpar{\log\{dD\,(\Lambda\vee1)\}}\,.\label{eq:q-q0-general}
\end{equation}
Also, due to the choice of $\eta_{i}$, the probability that the failure
condition is met until the termination of the annealing is bounded
by $\frac{1}{10}$.

\paragraph{Complexity analysis.}

We first recall a result on closeness of two logconcave distributions
from \cite[Lemma 4.4]{KV25faster}.
\begin{lem}
\label{lem:global-annealing} Let $\D\nu\propto e^{-V}\,\D x$ and
$\D\mu\propto e^{-(1+\alpha)\,V}\,\D x$ be logconcave probability
measures over $\Rd$. For $q>1$ and $\delta>0$ such that $1-q\delta>0$,
\[
\eu R_{q}(\mu\mmid\nu)\leq\begin{cases}
\frac{qd\alpha^{2}}{2} & \text{if }\alpha\geq0\,,\\
\frac{qd\alpha^{2}}{1-q\delta} & \text{if }\alpha\in[-\delta/2,0]\,.
\end{cases}
\]
\end{lem}

We now bound the complexity of the overall algorithm as $\Otilde(d^{2.5}+qd^{2}\,(R^{3/2}\vee1)(\Lambda^{1/4}\vee1))$.
\begin{proof}
[Proof of Theorem~\ref{thm:complexity-Renyi-warm-LC}] We set $\eta_{\exp}=\veps_{\exp}=\veps<1/100$
and consider $q_{0}$ and $q\geq2\vee\Otilde(\log\frac{dD}{\veps})$
(to satisfy \eqref{eq:q-q0-general}). At initialization, $\mu_{0}$
can be sampled by rejection sampling with $\O(1)$ queries in expectation
(see \cite[Lemma 5.6]{KV25sampling}). In Phase I, any consecutive
distributions satisfy $\eu R_{2q_{0}}(\mu_{i}\mmid\mu_{i+1})\leq1$
by Lemma~\ref{lem:global-annealing}. Since there are $\O(q_{0}^{1/2}d^{1/2}\log d)$
many phases, with each one requiring $\Otilde(d^{2}\log^{2}q_{0}m)$
queries from an $O(1)$-warm start, the query complexity of Phase
I is $\Otilde(q_{0}^{1/2}d^{2.5}\log^{2}q_{0}m)$. Meanwhile, we relay
$\eu R_{q_{0}}$-warmness through the weak triangle inequality (as
done in the proof of Theorem~\ref{thm:complexity-Renyi-warm}).

In Phase II, closeness during the simultaneous annealing (called the
outer annealing hereafter) is justified again by Lemma~\ref{lem:global-annealing},
and $\psann$ requires $\Otilde(d^{2}\log^{2}q_{0}m)$ queries for
each outer annealing. Each outer annealing is followed by $O(q_{0}^{1/2}D/d^{1/2})$
many inner phases, where $\sigma^{2}$ goes back to $1$ again. Closeness
within the inner annealing follows from Lemma~\ref{lem:closeness-annealing},
and $\psann$ requires $\Otilde(d^{2}\log^{2}q_{0}m)$ queries. Thus,
completion of one outer annealing (combined with the inner annealing)
requires $\Otilde((d^{2}+q_{0}^{1/2}d^{1.5}D)\log^{2}q_{0}m)$ queries.
As there are $\O(q_{0}^{1/2}d^{1/2}\log d)$ outer phases, the total
complexity of Phase II is $\Otilde((q_{0}^{1/2}d^{2.5}+q_{0}d^{2}D)\log^{2}q_{0}m)$.

In Phase III, closeness of annealing distributions follow again from
Lemma~\ref{lem:closeness-annealing}. Since doubling of initial $\sigma^{2}$
requires $\O(q_{0}D/\sigma)$ phases, with phase requiring $\Otilde(d^{2}\sigma^{2}\log^{2}q_{0}m)$
queries from an $\O(1)$-warm start, each doubling uses $\Otilde(q_{0}d^{2}\sigma D\log^{2}q_{0}m)$
queries. As $\sigma^{2}$ is increased while $\sigma^{2}\lesssim qD\,(\Lambda^{1/2}\vee1)\log^{1/2}d$,
it takes Phase III $\Otilde(q^{1/2}q_{0}d^{2}D^{3/2}\,(\Lambda^{1/4}\vee1)\log^{2}q_{0}m)$
queries to obtain a sample whose law is $\O(1)$-close to $\mu_{m}$
in $\eu R_{q_{0}}$-divergence. Due to hypercontractivity of $\PS$
again, iterating $\psann$ toward $\mu_{m}$ $\O(d^{2}\sigma_{m}^{2}\log^{2}qm)$
times returns a sample $Z^{*}$ such that $\eu R_{2q}(\law Z^{*}\mmid\mu_{m})\leq1$.
Hence, Phase III uses 
\[
\Otilde\bpar{q^{1/2}q_{0}d^{2}D^{3/2}\,(\Lambda^{1/4}\vee1)\log^{2}q_{0}m+qd^{2}D\,(\Lambda^{1/2}\vee1)\log^{2}qm}=\Otilde\bpar{qd^{2}D^{3/2}(\Lambda^{1/4}\vee1)\log^{2}qm}
\]
 queries in total. Also, by the early-stopping result and weak triangle
inequality, we have $\eu R_{q}(\law X^{*}\mmid\bar{\pi})\leq\frac{1.1}{10}+\half$.
Since $\eu R_{\infty}(\bar{\pi}\mmid\pi)\leq\log2$, we obtain that
$\eu R_{q}(\law Z^{*}\mmid\pi)\leq1$.

As $D\asymp R$, adding up complexities of each phase bound the total
query complexity as 
\[
\Otilde\bpar{d^{2.5}+qd^{2}R^{3/2}\,(\Lambda^{1/4}\vee1)}\,.
\]
Lastly, restarting the annealing (in case of failure) increases the
total query complexity multiplicatively by $10/9$, and the conditional
law on success has an extra bias at most $\log\frac{10}{9}$ in $\eu R_{\infty}$.
Hence, by $\eu R_{q}(\law(X'\mid\text{success})\mmid\pi)\leq2$.
\end{proof}
We establish the query complexity of general logconcave sampling from
scratch.
\begin{proof}
[Proof of Corollary~\ref{cor:general-LC-comp}] We set $\eta_{\exp}=\veps_{\exp}=\veps<1/100$
and consider $q_{0}$ and $q\geq2\vee\Otilde(\log\frac{dD}{\veps})$.
By Theorem~\ref{thm:lc-warm}, the expected number of restarts of
$\psexp$ is at most $O(1)$, so the warm-start generation will be
repeated at most $O(1)$ times. Combining the complexity results in
Theorem~\ref{thm:complexity-Renyi-warm-LC} and~\ref{thm:lc-warm},
the total complexity of logconcave sampling from scratch is 
\[
\Otilde\bpar{d^{2.5}+qd^{2}R^{3/2}\,(\Lambda^{1/4}\vee1)+qd^{2}\,(\Lambda\vee1)\log^{3}\frac{1}{\veps}}\,.
\]
When $1\leq q\leq q_{*}$, we generate the $\eu R_{q_{*}}$-warmness
and use the monotonicity of $\eu R_{q}$, which introduces an extra
factor of $q_{*}=\Otilde(1)$.
\end{proof}

\section{Obstruction to further acceleration }

In \cite{KV25faster}, it was conjectured that for the uniform distribution
over any isotropic convex body $\pi$ in $\Rd$, the largest eigenvalue
$\lda_{t}:=\norm{\cov\pi\gamma_{t}}$ remains as $\Otilde(1)$ for
any $t>0$. It was indeed shown that $\lda_{t}=O(1)$ if $t\leq1$
or $t\gtrsim R\Lambda^{1/2}\log^{2}d\log^{2}\frac{R^{2}}{\Lambda}$.
Note that the latter has been slightly improved to $R\Lambda^{1/2}\log^{1/2}d$
through Lemma~\ref{lem:early-stopping}. The rationale behind the
conjecture is, as $t$ grows, the Gaussian factor $\gamma_{t}$ push
the mass toward the boundary of $\K$, which tends to increase $\lda_{t}$.

To tackle this question, let us define quadratic-tilt stability.
\begin{defn}
[Quadratic-tilt stability] Let $\pi$ be a probability measure in
$\Rd$. For $\lda>0$, we say that $\pi$ is \emph{$\lda$-quadratic-tilt
stable} if for any $t>0$,
\[
\norm{\cov\pi\gamma_{t}}\leq\lda\,.
\]
\end{defn}

Hence, the conjecture is simply to establish $\Otilde(1)$-quadratic-tilt
stability of the uniform distribution over any isotropic convex body.

In \S\ref{subsec:Acceleration-conj}, we show that one can afford
more aggressive annealing under $\Otilde(1)$-quadratic-tilt stability,
obtaining the $d^{2.5}$-complexity of isotropic logconcave sampling
(from scratch). This improves the current complexity of $d^{2.75}$
and matches the $d^{2.5}$ \emph{iteration} complexity of the $\sw$
in \cite{LV24eldan}. However, in \S\ref{subsec:counterexample},
we disprove the conjecture, illustrating the limitation of the current
annealing approach.

\subsection{Acceleration under quadratic-tilt stability\label{subsec:Acceleration-conj}}

We illustrate how the acceleration is feasible under $\lda_{t}=\Otilde(1)$
for any $t>0$. Recall the warm-start generation algorithm: (1) sample
from $\mu_{1}=\pi\gamma_{d^{-1}}$ by rejection sampling, and (2)
while $\sigma_{i}^{2}\lesssim qR\Lambda^{1/2}\log^{1/2}d$, move from
$\bar{\mu}_{i}$ to $\bar{\mu}_{i+1}$ by $\psgauss$ such that $\eu R_{q_{0}}(\bar{\mu}_{i+1}\mmid\mu_{i+2})\leq1$,
where $\sigma_{i+1}^{2}=\sigma_{i}^{2}\,(1+\nicefrac{\sigma_{i}}{q_{0}R})$.
If $\lda_{t}=\Otilde(1)$ for any $t>0$, then the main annealing
can be improved to $\sigma_{i+1}^{2}=\sigma_{i}^{2}\,(1+\nicefrac{\sigma_{i}^{2}}{q_{0}R})$.
Note that this allows faster annealing when $\sigma_{i}^{2}\geq1$.
\begin{prop}
Let $\pi$ be any isotropic logconcave distribution over $\Rd$ with
$\norm{\cov\pi\gamma_{t}}=\Otilde(1)$. For $t>0$, $q\geq2$, and
$\alpha\asymp t/\Otilde(qR)$, the probability measures $\pi\gamma_{t}$
and $\pi\gamma_{t\,(1+\alpha)}$ satisfy 
\[
\eu R_{q}(\pi_{t}\mmid\pi_{t\,(1+\alpha)})\leq\log2\,.
\]
\end{prop}

\begin{proof}
We previously used the R\'enyi version of LSI, but we use \eqref{eq:RPI-log}
this time:
\[
\eu R_{q}(\pi_{t}\mmid\pi_{t\,(1+\alpha)})\leq\log\Bpar{1-\frac{q\cpi(\pi_{t\,(1+\alpha)})}{4}\,\msf{RFI}_{q}(\pi_{t}\mmid\pi_{t\,(1+\alpha)})}^{-1}\,.
\]
Recall that
\[
\msf{RFI}_{q}(\pi_{t}\mmid\pi_{t\,(1+\alpha)})\leq\frac{q\alpha^{2}R^{2}}{t^{2}}\,.
\]
Taking $\alpha\asymp t\,(qR\,\Otilde(1))^{-1}$, we can ensure the
term inside the log is smaller than $1/2$.
\end{proof}
Hence, as each phase requires roughly $d^{2}\sigma^{2}$ queries from
an $\O(1)$-warm start, and each doubling of a given $\sigma^{2}$
only needs $q_{0}R/\sigma^{2}$ (instead of $q_{0}R/\sigma$) when
$\sigma^{2}\leq q_{0}R\approx d^{1/2}$, the complexity of each doubling
would need $d^{2}R$. Since $R^{2}=d$ due to isotropy, the final
complexity is improved to $d^{2.5}$.

We mention that isotropic uniform distributions over hypercube, simplex,
and any convex bodies of revolution (e.g., a cone, ball) actually
satisfy $O(1)$-quadratic-tilt stability. See \S\ref{sec:Quadratic-tilt-stability}.

\subsection{Counterexample\label{subsec:counterexample}}

We show that there exists an isotropic convex body which is not $O(1)$-quadratic-tilt
stable. Consider the following convex body in $\R^{d+1}$ from \cite{Bizeul26logsobolev}:
for a sufficiently large constant $C_{0}>0$,
\[
\K=(\sqrt{3}B_{\infty}^{d}\times\R)\cap\bbrace{(x,\lda)\in\Rd\times\R:\abs x\leq\sqrt{d}+C_{0}\,(1-\abs{\lda})}\,.
\]
Due to the symmetry (e.g., the sign flip of $x$ and $\lda$, and
the permutation of $x$-coordinates), the barycenter is at the origin,
and the covariance of the uniform distribution is of the form
\[
\Sigma:=\cov\pi=\left[\begin{array}{cc}
aI_{d}\\
 & b
\end{array}\right]\qquad\text{for some }a,b>0\,.
\]
Bizeul used this instance to prove the sharpness of an estimate on
sub-Gaussianity of quadratic-tilted logconcave probability measures.
Bizeul \cite{Bizeul26logsobolev} mentions $\norm{\cov\pi\gamma_{t}}\gtrsim\min(t^{-1},1+t^{2}d)$.
Here, we verify that the isotropized version of $\K$ has the peak
variance of $d^{1/3}$, thus disproving the quadratic-tilt stability
conjecture.

\subsubsection{Near-isotropy of $\protect\K$}

Below, we use $c,c',C'$ to denote universal constants, which could
be different in each line. For fixed $\lda$, the slice of $\K$ at
$\lda$ is 
\[
S_{\lda}:=\{x\in[-\sqrt{3},\sqrt{3}]^{d}:\abs x\leq r(\lda)\}\quad\text{for }r(\lda):=\sqrt{d}+C_{0}\,(1-\abs{\lda})\,,
\]
so the $\lda$-marginal of $\pi$ satisfies $\pi^{\Lambda}(\lda)\propto\vol S_{\lda}$.
For $Z\sim\Unif(\sqrt{3}B_{\infty}^{d})$,
\[
\frac{\vol S_{\lda}}{\vol(\sqrt{3}B_{\infty}^{d})}=\P\bpar{\abs Z\leq r(\lda)}\,.
\]
Recall Hoeffding's inequality: for independent random variables $X_{1},\dots,X_{d}$
with $X_{i}\in[a_{i},b_{i}]$ a.s., their sum $S_{d}=\sum_{i\in[d]}X_{i}$
satisfies that
\[
\P(S_{d}-\E S_{d}\geq s)\leq\exp\bpar{-\frac{2s^{2}}{\sum_{i\in[d]}(b_{i}-a_{i})^{2}}}\quad\text{for any }s>0\,.
\]
Using the Hoeffding, for $Z=(Z_{1},\dots,Z_{d})\sim\Unif(\sqrt{3}B_{\infty}^{d})$,
due to $\E[Z_{i}^{2}]=1$,
\[
\P(\abs Z^{2}-d\leq-s)\leq\exp\bpar{-\frac{2s^{2}}{9d}}\,.
\]
Since $\abs Z\leq\sqrt{d}-s$ implies $\abs Z^{2}\leq d-2s\sqrt{d}+s^{2}\leq d-s\sqrt{d}$
when $s\leq\sqrt{d}$, it follows that for any $s\in[0,\sqrt{d}]$,
\begin{equation}
\P(\abs Z\leq\sqrt{d}-s)\leq\exp\bpar{-\frac{2s^{2}}{9}}\quad\text{(and similarly)}\quad\P(\abs Z\geq\sqrt{d}+s)\leq\exp\bpar{-\frac{8s^{2}}{9}}\,.\label{eq:hoeffding}
\end{equation}

\paragraph{Variance in $\Lambda$-direction.}

Let us bound $b:=\var_{\pi}\Lambda$. Since $b=\E[\lda^{2}]\ge\E\bbrack{\lda^{2}\,\ind\{\abs{\lda}\in[\frac{1}{4},\half]\}}\geq\frac{1}{16}\,\P(\frac{1}{4}\leq\abs{\lda}\leq\frac{1}{2})$,
it suffices to show that $\P(\frac{1}{4}\leq\abs{\lda}\leq\frac{1}{2})=\Omega(1)$
to establish $b=\Omega(1)$. Using $\vol\K=\int\vol S_{\lda}\,\D\lda=2\int_{\geq0}\vol S_{\lda}\,\D\lda$,
\[
\P\bpar{\frac{1}{4}\leq\abs{\lda}\leq\frac{1}{2}}\gtrsim\frac{\int_{1/4}^{1/2}\vol S_{\lda}\,\D\lda}{\vol\K}=\frac{\int_{1/4}^{1/2}\vol S_{\lda}\,\D\lda}{2\int_{0}^{1+\sqrt{d}/C_{0}}\vol S_{\lda}\,\D\lda}
\]

We show that both denominator and numerator are within constant factors
of $\vol(\sqrt{3}B_{\infty}^{d})$. Using the Hoeffding \eqref{eq:hoeffding}
and taking $C_{0}$ large enough, for $\abs{\lda}\leq1/2$ (which
implies $r(\lda)\geq\sqrt{d}+C_{0}/2$),
\[
\frac{\vol S_{\lda}}{\vol(\sqrt{3}B_{\infty}^{d})}=\P\bpar{\abs Z\leq r(\lda)}\geq\P\bpar{\abs Z\leq\sqrt{d}+\frac{C_{0}}{2}}\geq1-\exp\bpar{-\Theta(C_{0}^{2})}\geq0.99\,.
\]
When $1\leq\abs{\lda}\lesssim\sqrt{d}$,
\[
\frac{\vol S_{\lda}}{\vol(\sqrt{3}B_{\infty}^{d})}=\P\bpar{\abs Z\leq\sqrt{d}-C_{0}\,(\abs{\lda}-1)}\lesssim\exp(-c\lda^{2})\,.
\]
Hence, $b=\Omega(1)$ follows from
\[
\int_{0}^{1+\sqrt{d}/C_{0}}\vol S_{\lda}\,\D\lda=\int_{0}^{1}\cdot+\int_{1}^{1+\sqrt{d}/C_{0}}\cdot\leq\vol(\sqrt{3}B_{\infty}^{d})\,\Bigl(1+\int_{1}^{1+\sqrt{d}/C_{0}}e^{-c\lda^{2}}\,\D\lda\Bigr)\lesssim\vol(\sqrt{3}B_{\infty}^{d})\,.
\]

Also, $b=O(1)$ follows from
\begin{align*}
\E[\lda^{2}] & =\E[\lda^{2}\,\ind\{\abs{\lda}\leq1\}]+\E[\lda^{2}\,\ind\{\abs{\lda}\geq1\}]\leq1+\frac{\int_{1}^{1+\sqrt{d}/C_{0}}\lda^{2}\vol S_{\lda}\,\D\lda}{\int_{0}^{1+\sqrt{d}/C_{0}}\vol S_{\lda}\,\D\lda}\\
 & \leq1+\frac{\int_{1}^{1+\sqrt{d}/C_{0}}\lda^{2}\vol S_{\lda}\,\D\lda}{\int_{0}^{1/2}\vol S_{\lda}\,\D\lda}\leq1+\frac{C'\int_{1}^{\infty}\lda^{2}\,\vol(\sqrt{3}B_{\infty}^{d})\,e^{-c\lda^{2}}\,\D\lda}{\vol(\sqrt{3}B_{\infty}^{d})}=O(1)\,.
\end{align*}

\paragraph{Variance in $X$-direction.}

Let us bound $a:=\var_{\pi}X_{i}$ for any $i\in[d]$. Since $\abs{x_{i}}\leq\sqrt{3}$,
we have $a=\E[X_{i}^{2}]\leq3$. For the lower bound, observe that
\begin{align*}
da=\E_{\pi}[\abs X^{2}] & \geq\E_{\pi}[\abs X^{2}\,\ind\{\abs{\lda}\leq1/2\}]=\E_{\pi^{\Lambda}}\bbrack{\ind\{\abs{\lda}\leq1/2\}\,\E_{\pi^{X|\Lambda=\lda}}[\abs X^{2}]}\,.
\end{align*}
For $\abs{\lda}\leq1/2$ (which is equivalent to $r(\lda)\geq\sqrt{d}+C_{0}/2$),
\begin{align*}
\E_{\pi^{X|\Lambda=\lda}}[\abs X^{2}] & =\int_{S_{\lda}}\abs x^{2}\,\frac{\D x}{\vol S_{\lda}}=\frac{\vol(\sqrt{3}B_{\infty}^{d})}{\vol S_{\lda}}\int_{\sqrt{3}B_{\infty}^{d}}\abs x^{2}\,\frac{\ind[x\in S_{\lda}]}{\vol(\sqrt{3}B_{\infty}^{d})}\,\D x\geq\int_{\sqrt{3}B_{\infty}^{d}}\abs x^{2}\,\frac{\ind[x\in S_{\lda}]}{\vol(\sqrt{3}B_{\infty}^{d})}\,\D x\\
 & =\E_{Z}[\abs Z^{2}\,\ind\{Z\in S_{\lda}\}]\geq\E_{Z}[\abs Z^{2}\,\ind\{\abs Z\leq\sqrt{d}+C_{0}/2\}]\\
 & =\E[\abs Z^{2}]-\E[\abs Z^{2}\,\ind\{\abs Z\geq\sqrt{d}+C_{0}/2\}]\\
 & \geq d-3d\,\P(\abs Z\geq\sqrt{d}+C_{0}/2)\geq d-3d\exp(-c'C_{0}^{2})\geq d/2\,,
\end{align*}
where we took $C_{0}$ large enough in the last line. Therefore, $a=\Omega(1)$
follows from
\[
da=\E_{\pi}[\abs X^{2}]\gtrsim d\,\P_{\lda}(\abs{\lda}\leq1/2)\gtrsim d\,.
\]

\subsubsection{Lower bound on the variance in $\protect\lda$-direction}

Let us denote Bizeul's convex body by $\K_{0}$ in place of $\K$.
Define the isotropic scaling 
\[
D=\mathrm{diag}(a^{-1/2}I_{d},b^{-1/2})\,,\qquad\K:=D\K_{0}\,.
\]
Then $\mu=D_{\#}\pi$ is isotropic, and $\K$ can be written as 
\[
\K=\bpar{(L\sqrt{3})B_{\infty}^{d}\times\R}\cap\{(x,\lambda):\abs x\le R(\lambda)\}\,,
\]
where $L\asymp1$, $R(\lambda)=Ld^{1/2}+c\,(1-q^{-1}\abs{\lda})$,
$q\asymp1$, and $c\asymp1$. Note that $R(\lambda)=Ld^{1/2}-c\,(\abs{\lda}/q-1)$
for $\abs{\lda}\ge q$.

\paragraph{Quadratic tilting.}

For $t>0$, define the tilted measure on $\K$ by $\mu_{t}(\D x,\D\lambda)\propto e^{-t\,(\norm x^{2}+\lda^{2})}\,\D x\D\lda$.
When $(X_{t},\Lambda_{t})\sim\mu_{t}$, it holds that $\E\Lambda_{t}=0$
(by symmetry) and $\mathrm{Var}\Lambda_{t}=\E[\Lambda_{t}^{2}]$.
\begin{prop}
\label{prop:counter-example}$\var\Lambda_{t}\gtrsim d^{1/3}$ for
$t\asymp d^{-1/3}$. Precisely, 
\[
\var\Lambda_{t}\gtrsim\begin{cases}
t^{2}d & \text{if }d^{-1/2}\lesssim t\lesssim d^{-1/3}\,,\\
t^{-1} & \text{if }d^{-1/3}\lesssim t\ll1\,.
\end{cases}
\]
\end{prop}

Let $C=[-L,L]^{d}$, and define the slice $S_{\lambda}=\{x\in C:\abs x\le R(\lambda)\}$
for $\lda\geq0$. Define the product probability measure on $C$ as
\[
Q_{t}(\D x)=\frac{\ind[x\in C]\,e^{-t\,\abs x^{2}}}{\int_{C}e^{-t\,\abs u^{2}}\,\D u}\,\D x\propto\prod_{i=1}^{d}(e^{-tx_{i}^{2}}\,\ind[\abs{x_{i}}\leq L])\,,
\]
and its $\lda$-marginal density of $\mu_{t}$ satisfies $p_{t}(\lambda)\propto e^{-t\lambda^{2}}\,Q_{t}(S_{\lda})$.

\paragraph{Concentration of the base measure $Q_{t}$.}

Let $Y=\abs X^{2}=\sum_{i=1}^{d}X_{i}^{2}$ (w.r.t. $X\sim Q_{t}$),
$m_{1}(t)=\E_{Q_{t}}[X_{1}^{2}]$, and $m(t):=\E_{Q_{t}}Y=dm_{1}(t)$.
We also define
\[
I_{0}(t)=\int_{-L}^{L}\exp(-tx_{1}^{2})\,\D x_{1}\,,\qquad Z_{d}(t)=\int_{C}\exp(-t\,\abs x^{2})\,\D x\,.
\]
Note that
\[
m_{1}(t)=-\de_{t}\log I_{0}(t)\,,\qquad m_{1}'(t)=-\var_{Q_{t}}(X_{1}^{2})\,.
\]
When $t=0$, we have $\E[X_{1}^{2}]=L^{2}$ and $\E[X_{1}^{4}]=9L^{4}/5$
, so $\var_{Q_{0}}(X_{1}^{2})=4L^{4}/5\asymp1$. Since the $X_{1}$-marginal
of $Q_{t}$ has density proportional to $\exp(-tx_{1}^{2})\,\ind[\abs{x_{1}}\leq L]$,
there exists a small constant $t_{0}$ such that for any $t\leq t_{0}$,
\[
\var_{Q_{t}}(X_{1}^{2})=\Theta(1)\,.
\]
As $m_{1}(t)-L^{2}=-\int_{0}^{t}\var_{Q_{t}}(X_{1}^{2})\,\D t=-\Theta(t)$,
we have 
\[
m(t)=dL^{2}-\Theta(dt)\,.
\]

Since $X_{i}^{2}\in[0,3L^{2}]$ are i.i.d., Hoeffding yields that
for any $s\geq0$,
\[
Q_{t}\bpar{Y-m(t)\leq-sd^{1/2}}\leq\exp\bpar{-\frac{2s^{2}}{9L^{4}}}\,,\qquad Q_{t}\bpar{Y-m(t)\geq sd^{1/2}}\leq\exp\bpar{-\frac{2s^{2}}{9L^{4}}}\,.
\]
Let $M(t)$ be a median of $Y$. Then, the Hoeffding above implies
$\abs{M(t)-m(t)}\lesssim d^{1/2}$. Since $m(t)=dL^{2}-\Theta(dt)$,
we have
\begin{equation}
M(t)=L^{2}d-\Theta(dt)\pm O(d^{1/2})\,.\label{eq:Mt-order}
\end{equation}

\paragraph{Median cutoff at $\protect\lda^{*}$.}

In the paragraph above, we have dealt with the order of the median
of $\abs X^{2}$ w.r.t.\ the base measure $Q_{t}$ (not $\mu_{t}$).
We now examine how $Q_{t}$ interacts with each slice $S_{\lda}$.

Hereafter, we assume that $t\gtrsim d^{-1/2}$, for which $M(t)\leq L^{2}d=R(q)^{2}$.
By the definition of $R(\lda)$ and $S_{\lda}$, as we increase $\abs{\lda}(\geq q)$,
the slice $S_{\lda}$ is truncated by a ball of \emph{smaller} radius
$R(\lda)$. Hence, there exists a unique $\lda^{*}$ such that 
\[
M(t)=R(\lda^{*})^{2}\,.
\]
That is, we are attempting to find the point $\lda^{*}$, where the
median of $Q_{t}$ is cut off.

Letting $x=c\,(\abs{\lda}/q-1)$, we should solve
\begin{align*}
x & =\sqrt{L^{2}d}-\sqrt{L^{2}d-\Theta(dt)\pm O(d^{1/2})}=\frac{\Theta(dt)\pm O(d^{1/2})}{\sqrt{L^{2}d}+\sqrt{L^{2}d-\Theta(dt)+O(d^{1/2})}}\\
 & =\frac{\Theta(dt)}{\sqrt{L^{2}d}+\sqrt{L^{2}d-\Theta(dt)+O(d^{1/2})}}\pm O(1)\,.
\end{align*}
When $t\ll1$ (so that $\Theta(dt)\ll L^{2}d$), we have $x\pm O(1)\asymp\frac{\Theta(dt)}{2\sqrt{L^{2}d}}\asymp t\sqrt{d}$.
Therefore, when $t\ll1$,
\[
\lda^{*}\asymp1+t\sqrt{d}\,.
\]

\paragraph{Gaussian tail of $p_{t}(\protect\lda)$ after $\protect\lda\protect\geq\protect\lda^{*}$.}

The definition of $\lda^{*}$ implies that for $\abs{\lda}\leq\lda^{*}$,
\[
Q_{t}\bpar{Y\leq R(\lda)^{2}}\geq Q_{t}\bpar{Y\leq M(t)}=\half\,.
\]
When $\lda=\lda^{*}+s$ for $s>0$, the endpoint $\lda_{\max}$ (i.e.,
$R(\lda_{\max})=0$) satisfies that $\lda_{\max}=q+qc^{-1}Ld^{1/2}$.
Thus,
\[
0\leq s\leq s_{\max}:=\lda_{\max}-\lda^{*}\leq\lda_{\max}-q\leq qc^{-1}Ld^{1/2}\,.
\]
Using $s\leq s_{\max}\leq qc^{-1}Ld^{1/2}$ and $R(\lda^{*}+s)-R(\lda^{*})=-q^{-1}cs$,
\begin{equation}
R(\lda^{*})^{2}-R(\lda^{*}+s)^{2}=q^{-1}cs\,\bpar{2R(\lda^{*})-q^{-1}cs}\geq q^{-1}cs\,\bpar{2R(\lda^{*})-Ld^{1/2}}\,.\label{eq:coreineq1}
\end{equation}
Taking $t\lesssim L^{2}$, we have
\[
R(\lda^{*})^{2}=M(t)\underset{\eqref{eq:Mt-order}}{=}L^{2}d-\Theta(dt)\pm O(d^{1/2})\geq\frac{9L^{2}d}{16}\,.
\]
Putting this back to \eqref{eq:coreineq1},
\[
R(\lda^{*})^{2}-R(\lda^{*}+s)^{2}\geq\frac{csLd^{1/2}}{2q}\,.
\]
Using $\abs{m(t)-M(t)}\lesssim\sqrt{d}$ with $R(\lda^{*})^{2}=M(t)$,
\[
R(\lda^{*}+s)^{2}\leq R(\lda^{*})^{2}-\frac{csLd^{1/2}s}{2q}=m(t)\pm O(d^{1/2})-\frac{csLd^{1/2}}{2q}\,.
\]
Hence, for $1\lesssim s\leq s_{\max}$ and some constant $c'>0$,
\[
Q_{t}\bpar{Y\leq R(\lda^{*}+s)^{2}}\leq Q_{t}\bpar{Y\leq m(t)-sd^{1/2}}\leq\exp(-c's^{2})\,.
\]

\paragraph{First regime: $d^{-1/2}\lesssim t\lesssim d^{-1/3}$.}

Let $\bar{p}_{t}(\lda)=e^{-t\lambda^{2}}\,Q_{t}(S_{\lda})$ (i.e.,
an unnormalized density of $p_{t}(\lda)$). The result above implies
that for $\lda=\lda^{*}+s$ with $s\gtrsim1$,
\[
\bar{p}_{t}(\lda^{*}+s)\leq e^{-t\,(\lda^{*}+s)^{2}}e^{-c's^{2}}\leq e^{-c's^{2}}\,.
\]
Observe that
\[
\var\Lambda_{t}=\E[\Lambda_{t}^{2}]=\frac{\int_{\geq0}\lda^{2}\bar{p}_{t}(\lda)\,\D\lda}{\int_{\geq0}\bar{p}_{t}(\lda)\,\D\lda}\,.
\]
As for the denominator,
\[
\int_{\geq0}\bar{p}_{t}(\lda)\,\D\lda=\int_{0}^{\lda^{*}+\Theta(1)}+\int_{\lda^{*}+\Theta(1)}^{\infty}\leq\lda^{*}+\Theta(1)\,.
\]
As for the numerator, since $t\,(\lda^{*})^{2}\asymp t^{3}d\lesssim1$,
we have $\exp(-t\lda^{2})\gtrsim1$ and $Q_{t}(Y\le R(\lambda)^{2})\geq1/2$
on $\lda\in[0,\lda^{*}]$. Hence,
\[
\int_{\geq0}\lda^{2}\,\bar{p}_{t}(\lda)\,\D\lda\geq\int_{0}^{\lda^{*}}\lda^{2}\,\bar{p}_{t}(\lda)\,\D\lda\gtrsim\int_{0}^{\lda^{*}}\lda^{2}\,\D\lda\asymp(\lda^{*})^{3}\,.
\]
Therefore, in the regime of $t\lesssim d^{-1/3}$,
\[
\var\Lambda_{t}\gtrsim\frac{(\lda^{*})^{3}}{1+\lda^{*}}\,.
\]
Since $\lda^{*}\asymp1+td^{1/2}\gtrsim1$ (due to $t\gtrsim d^{-1/2}$),
the variance is bounded lower by $(\lda^{*})^{2}\asymp t^{2}d$.

\paragraph{Second regime: $d^{-1/3}\lesssim t\ll1$.}

We first show that in this regime, $R(\lda)^{2}\geq M(t)$ when $\abs{\lda}\lesssim t^{-1/2}$.
For $r:=c\,(\abs{\lda}/q-1)$,
\begin{align*}
M(t) & =m(t)\pm O(d^{1/2})=L^{2}d-\Theta(dt)\pm O(d^{1/2})\,,\\
R(\lda)^{2} & =L^{2}d-2Lrd^{1/2}+r^{2}\geq L^{2}d-2rLd^{1/2}\,.
\end{align*}
Hence, $R(\lda)^{2}\geq M(t)$ holds if $\Theta(dt)-O(d^{1/2})\geq2Lrd^{1/2}$
(equivalently, $r\leq\Theta(d^{1/2}t)-O(1)$ and $\abs{\lda}\lesssim1+d^{1/2}t$).
When $d^{-1/3}\lesssim t$, we have $t^{-1/2}\lesssim d^{1/2}t$.
Then, the claim follows from
\[
\abs{\lda}\lesssim\frac{1}{t^{1/2}}\lesssim d^{1/2}t\leq1+d^{1/2}t\,.
\]

Let us lower-bound the variance. As for the denominator, due to $Q_{t}(\cdot)\leq1$,
\[
\int_{\geq0}\bar{p}_{t}(\lda)\,\D\lda\leq\int_{\geq0}e^{-t\lda^{2}}\,\D\lda\asymp\frac{1}{t^{1/2}}\,.
\]
As for the numerator, $Q_{t}(Y\leq R(\lda)^{2})\geq1/2$ on $\abs{\lda}\lesssim t^{-1/2}$.
Hence, for some constant $c'>0$, 
\[
\int\lda^{2}\,\bar{p}_{t}(\lda)\,\D\lda\gtrsim\int_{-c'/t^{1/2}}^{c'/t^{1/2}}\lda^{2}\,e^{-t\lda^{2}}\,\D\lda\asymp\frac{1}{t^{3/2}}\,.
\]
Therefore,
\[
\var\Lambda_{t}=\frac{\int\lda^{2}\bar{p}_{t}(\lda)\,\D\lda}{2\int_{\geq0}\bar{p}_{t}(\lda)\,\D\lda}\gtrsim\frac{1}{t}\,.
\]

\paragraph{}
\begin{acknowledgement*}
Part of this work was done while YK was a student researcher at Google
Research. The authors thank Sinho Chewi and Andre Wibisono for helpful
comments. This work was supported in part by NSF Award CCF-2106444
and a Simons Investigator award.
\end{acknowledgement*}
\bibliographystyle{alpha}
\bibliography{main}

\appendix

\section{Quadratic-tilt stability of structured distributions\label{sec:Quadratic-tilt-stability}}

\subsection{Hypercube }

Consider the isotropic hypercube, $\K=\otimes_{i=1}^{d}[-\sqrt{3},\sqrt{3}]$.
Then, any truncated Gaussian $\pi\gamma_{t}$ (here $\pi\propto\ind_{\K}$)
does not violate the conjecture. Since $\pi\gamma_{t}$ can be written
as $d$-tensor product of one-dimensional probability measures, $\cov\pi\gamma_{t}$
is a diagonal matrix. Clearly, each diagonal (i.e., variance along
any coordinate) is $\O(1)$ since the support has diameter of $\O(1)$.

\subsection{Simplex}

Consider the regular $d$-simplex in isotropic position. Let us embed
the simplex in the ambient $(d+1)$-dimensional space:
\[
\Delta_{d}=\Bbrace{x\in\R^{d+1}:x_{i}\geq0,\sum_{i=1}^{d+1}x_{i}=1}\,,
\]
centered at $\bar{x}=\frac{1}{d+1}\,(1,\dots,1)$, so that $y:=x-\bar{x}$
lies in the hyperplane $H:=\{y\in\R^{d+1}:\sum_{i}y_{i}=0\}$.

Now, consider the quadratically-tilted measure $\pi\gamma_{t}$ over
the isotropized simplex. Using the coordinate symmetry (i.e., permuting
coordinates), we can write $C:=\cov\pi\gamma_{t}=\alpha I_{d+1}+\beta\,(J_{d+1}-I_{d+1})$
for some $\alpha,\beta\in\R$, where $J_{d+1}$ is the all-ones matrix.
Since $y\in H$ satisfies $y\perp\mb 1$, we have $J_{d+1}y=0$. Thus,
the restriction of $C$ to $H$ satisfies
\[
Cy=(\alpha-\beta)\,y\quad\text{for all }y\in H\,,
\]
so $\cov(\pi\gamma_{t})|_{H}=(\alpha-\beta)\,I_{d}=:cI_{d}$. Using
the similar argument in \eqref{eq:cov-decrease}, 
\[
\E_{\pi\gamma_{t}}[\abs Y^{2}]=\frac{\E_{\pi}[\abs Y^{2}\,e^{-\abs Y^{2}/(2t)}]}{\E_{\pi}[e^{-\abs Y^{2}/(2t)}]}\leq\E_{\pi}[\abs Y^{2}]=d\,,
\]
where the last equality holds since $\pi$ is isotropic. Since $\E_{\pi\gamma_{t}}Y=0$
by symmetry, we have $\tr(\cov\pi\gamma_{t})=\E_{\pi\gamma_{t}}[\abs Y^{2}]\leq d$.
As $\tr(cI_{d})=cd$, it follows that $c\leq1$. Hence, $\norm{\cov\pi\gamma_{t}}=c=O(1)$.

\subsection{Isotropic convex body of revolution}

For an interval $I\subset\R$ and a concave function $r:\R\to\R_{\geq0}$,
let 
\[
\K=\{(s,y)\in\R\times\R^{d-1}:s\in I,\,\abs y\le r(s)\}\,,
\]
where the uniform distribution over $\K$ is isotropic. Consider a
quadratically-tilted measure over $\K$:
\[
\pi_{t}(\D x)\propto\exp\bpar{-\frac{\abs x^{2}}{2t}}\,\ind_{\K}(x)\,\D x\qquad\text{for }x=(s,y)\,.
\]

\subsubsection{$y$-direction}

Let $\sigma_{\perp}^{2}$ be the variance of $\pi_{t}$ in the $y$-direction.
Using the rotational symmetry and the similar argument in \eqref{eq:cov-decrease},
\[
\sigma_{\perp}^{2}\leq\frac{1}{d-1}\,\E_{\pi_{t}}[\abs Y^{2}]\leq\frac{1}{d-1}\,\E_{\pi}[\abs Y^{2}]\leq\frac{d}{d-1}=O(1)\,.
\]

\subsubsection{$s$-direction}

Conditionally on $S=s$, the radial coordinate $U=\abs Y$ has density,
for $R=r(s)$
\[
\nu_{t,R}(\D u)\propto u^{d-2}\exp\bpar{-\frac{u^{2}}{2t}}\,\ind_{[0,R]}(u)\,\D u\,.
\]
Hence, integrating $y\in\R^{d-1}$ gives the following $s$-marginal:
\[
p_{t}(s)\propto\exp\bpar{-\frac{s^{2}}{2t}}\,H_{t}\bpar{r(s)}\qquad\text{for }H_{t}(R):=\int_{0}^{R}u^{d-2}e^{-\frac{u^{2}}{2t}}\,\D u\,.
\]

\begin{lem}
\label{lem:U-average} Let $\eta_{R}$ be the probability measure
on $[0,R]$ with density proportional to $u^{d-2}$. If $f$ is increasing
and $g$ is decreasing on $[0,R]$, then $\E_{\eta_{R}}[fg]\le\E_{\eta_{R}}f\,\E_{\eta_{R}}g$.
In particular, for $f(u)=u^{2}$ and $g(u)=e^{-u^{2}/(2t)}$, 
\[
\E_{\nu_{t,R}}[U^{2}]\le\E_{\eta_{R}}[U^{2}]=\frac{d-1}{d+1}\,R^{2}\,.
\]
\end{lem}

\begin{proof}
The first claim can be deduced  from the definition of $\cov_{\eta_{R}}(f,g)$
(see \eqref{eq:cov-decrease}). For $f(u)=u^{2}$ and $g(u)=e^{-u^{2}/(2t)}$,
we have
\[
\E_{\eta_{R}}[u^{2}e^{-u^{2}/(2t)}]\leq\E_{\eta_{R}}[u^{2}]\,\E_{\eta_{R}}e^{-u^{2}/(2t)}\,.
\]
Equivalently,
\[
\int_{0}^{R}u^{d-2}\,\D u\int_{0}^{R}u^{2}u^{d-2}e^{-u^{2}/(2t)}\,\D u\leq\int_{0}^{R}u^{2}u^{d-2}\,\D u\int_{0}^{R}u^{d-2}e^{-u^{2}/(2t)}\,\D u\,,
\]
which implies that 
\[
\E_{\nu_{t,R}}[U^{2}]\le\E_{\eta_{R}}[U^{2}]=\frac{\int_{0}^{R}u^{d}\,\D u}{\int_{0}^{R}u^{d-2}\,\D u}=\frac{R^{d+1}/(d+1)}{R^{d-1}/(d-1)}=\frac{d-1}{d+1}\,R^{2}\,.
\]
This completes the proof.
\end{proof}

\paragraph{Maximum of $r(s)$.}

From above, $\E[\abs Y^{2}\mid S=s]=\frac{d-1}{d+1}\,r(s)^{2}$. Define
the axial density $p_{0}(s)\propto r(s)^{d-1}\,\ind_{I}(s)$ and recall
that $p_{0}$ is isotropic logconcave on $\R$. Taking expectation
with respect to $s\sim p_{0}$, 
\[
d-1=\E_{\pi}[\abs Y^{2}]=\frac{d-1}{d+1}\,\E_{p_{0}}[r(s)^{2}]\,,
\]
so 
\[
\frac{\int_{I}r(s)^{d+1}\,\D s}{\int_{I}r(s)^{d-1}\,\D s}=\E_{p_{0}}[r(s)^{2}]=d+1\,.
\]

Using this and concavity of $r(s)$, we show the following:
\begin{lem}
\label{lem:max-r} In the setting above,
\[
r_{\max}^{2}=\max_{s\in I}r(s)^{2}\leq16d\,.
\]
\end{lem}

\begin{proof}
Let $R=\max_{I}r(s)$, achieved at $s^{*}$. Then, it follows from
the concavity of $r(s)$ that the set $J:=\{s:r(s)\geq R/2\}$ has
length at least $\abs I/2$. To see this, consider the chord between
the left endpoint $(a,r(a))$ and $(s^{*},R)$. Concavity of $r(s)$
implies that $r$ is above $R/2$ (as $r(a)\geq0$). Repeating the
same argument for the right endpoint, the claim follows.

Let us show that 
\[
\frac{\int_{I}r(s)^{d+1}\,\D s}{\int_{I}r(s)^{d-1}\,\D s}\gtrsim R^{2}\,.
\]
To this end, using $\abs J\geq\abs{J^{c}}$,
\[
\int_{I}r(s)^{d-1}\,\D s=\int_{J}+\int_{J^{c}}\leq\int_{J}+\int_{J}=2\int_{J}r(s)^{d-1}\,\D s\,.
\]
Also,
\[
\int_{I}r(s)^{d+1}\,\D s\geq\int_{J}r(s)^{d+1}\,\D s\geq\frac{R^{2}}{4}\int_{J}r(s)^{d-1}\,\D s\,.
\]
Hence,
\[
(d+1=\E_{p_{0}}[r(s)^{2}]=)\frac{\int_{I}r(s)^{d+1}\,\D s}{\int_{I}r(s)^{d-1}\,\D s}\geq\frac{R^{2}}{8}\,,
\]
which completes the proof.
\end{proof}

\paragraph{Main proof.}

We need some identity:
\begin{lem}
For $H_{t}(R)=\int_{0}^{R}u^{d-2}e^{-u^{2}/(2t)}\,\D u$, it holds
that
\[
\frac{d-1}{R}\,\bpar{1-\frac{R^{2}}{t\,(d+1)}}\leq\frac{H_{t}'(R)}{H_{t}(R)}=\frac{d-1}{R}-\frac{1}{tR}\,\E_{\nu_{t,R}}[U^{2}]\leq\frac{d-1}{R}\,.
\]
In particular, if $R^{2}\le t\,(d+1)/2$, then $\frac{H_{t}'(R)}{H_{t}(R)}\ge\frac{d-1}{2R}$.
\end{lem}

\begin{proof}
Using IBP and $H_{t}'(R)=R^{d-2}e^{-R^{2}/2t}$,
\[
H_{t}(R)=e^{-\frac{R^{2}}{2t}}\,\frac{R^{d-1}}{d-1}+\frac{1}{d-1}\int_{0}^{R}\frac{u^{d}}{t}\,e^{-\frac{u^{2}}{2t}}\,\D u=\frac{R}{d-1}\,H_{t}'(R)+\frac{1}{d-1}\int_{0}^{R}\frac{u^{d}}{t}\,e^{-\frac{u^{2}}{2t}}\,\D u\,.
\]
Dividing both sides by $RH_{t}(R)/(d-1)$,
\[
\frac{H_{t}'(R)}{H_{t}(R)}=\frac{d-1}{R}-\frac{1}{tR}\frac{\int_{0}^{R}u^{d}e^{-\frac{u^{2}}{2t}}\,\D u}{\int_{0}^{R}u^{d-2}e^{-\frac{u^{2}}{2t}}\,\D u}=\frac{d-1}{R}-\frac{1}{tR}\,\E_{\nu_{t,R}}[U^{2}]\,.
\]
The claim follows from Lemma~\ref{lem:U-average}.
\end{proof}
Since $p_{0}$ is isotropic, we can take a mode within $[-10,10]$
(see Claim~\ref{claim:mode-iso}). By Claim~\ref{claim:lc-universal},
there exist universal $A,c>0$ such that $(\log p_{0})'(s)\le-c$
for all $s\ge A$. Differentiating $\log p_{t}(s)=-\frac{s^{2}}{2t}+\log H_{t}(r(s))+\mathrm{const}$
yields 
\[
(\log p_{t})'(s)=-\frac{s}{t}+r'(s)\,\frac{H_{t}'\bpar{r(s)}}{H_{t}\bpar{r(s)}}\,.
\]
Using Lemma~\ref{lem:max-r}, we take $t\geq32$ to satisfy $\frac{H_{t}'(r(s))}{H_{t}(r(s))}\ge\frac{d-1}{2r(s)}$.
On the right tail (where $r'(s)\le0$ by concavity), we have 
\[
(\log p_{t})'(s)\le-\frac{s}{t}+\frac{d-1}{2}\,\frac{r'(s)}{r(s)}=-\frac{s}{t}+\frac{1}{2}\,(\log p_{0})'(s)\le-\frac{c}{2}\qquad\text{for }s\geq A\,.
\]
Thus, $p_{t}$ has a universal exponential right tail, which implies
that second moment contributed by the right tail (after $A=\Omega(1)$)
is $O(1)$. We can handle the left tail in a similar manner. Also,
the second moment coming from $[-10,10]$ is $O(1)$. Hence, $\E_{p_{t}}[S^{2}]=O(1)$
and $\mathrm{Var}_{p_{t}}S=O(1)$ for $t\in[32,d^{1/2}]$. We can
handle the case of $t\in[1,32]$ by the Brascamp--Lieb inequality
inequality.

\paragraph{Missing details.}
\begin{claim}
\label{claim:mode-iso}The mode of an isotropic logconcave distribution
in $\R$ is within $[-10,10]$.
\end{claim}

\begin{proof}
Let $g$ be the pdf of the isotropic logconcave distribution, and
$m$ a mode of $g$. If $m>10$, then log-concavity implies that $g$
is non-decreasing on $(-\infty,m]$. Thus, on $[0,10]$, we have $g(x)\geq g(0)\geq1/8$
from \cite{LV07geometry},
\[
1=\int g(x)\,\D x\geq\int_{0}^{10}g(x)\,\D x\geq\frac{10}{8}\,,
\]
which is contradiction. Repeating the same argument for $m<-10$,
we conclude that $\abs m\leq10$.
\end{proof}
Let $A\geq11$. Then, $g$ is non-increasing on $[A-1,\infty)$ since
$m\leq10$. Hence,
\[
g(A)\leq\int_{A-1}^{\infty}g(u)\,\D u=\P(X\geq A-1)\leq\P(\abs X\geq A-1)\leq e^{-(A-1)+1}=e^{-A+2}\,.
\]
Hence,
\[
\frac{g(A)}{g(0)}\leq8e^{-A+2}\,.
\]

\begin{claim}
\label{claim:lc-universal} There exist universal $A,c>0$ such that
$(\log p_{0})'(s)\le-c$ for all $s\ge A$.
\end{claim}

\begin{proof}
Take the concave function $\phi(s)=\log g(s)$, where $g(s)$ is the
pdf of the isotropic logconcave $p_{0}$. Then, for $A\geq11$,
\[
\phi'(A)\leq\frac{\phi(A)-\phi(0)}{A}=\frac{1}{A}\,\log\frac{g(A)}{g(0)}\leq\frac{1}{A}\,(\log8-A+2)\,.
\]
For $A=\Omega(1)$, $\phi'(A)\leq-1/2$. Since $\phi'$ is non-increasing
due to concavity, for all $x\geq A$,
\[
\phi'(x)=(\log g)'(x)\leq\phi'(A)\leq-1/2\,,
\]
which completes the proof.
\end{proof}

\section{Related work\label{subsec:Related-work}}

We survey the literature on applications of zeroth-order logconcave
sampling.

\subsection{More on uniform sampling}

 In early work, researchers used many pioneering ideas to obtain
polynomial-time algorithms. In \cite{DFK91random}, the input body
$\K$ is first mollified and then sampled using the $\msf{Grid}\text{ }\msf{walk}$,
a random walk on a suitably fine lattice over the body. This algorithm
is then analyzed through the conductance of the resulting Markov chain,
the well-known framework developed by Jerrum and Sinclair \cite{JS88conductance}.
To use this, they applied an isoperimetric inequality of Milman and
Schechtman \cite{MS86asymptotic}, showing that the $\msf{Grid}\text{ }\msf{walk}:\eu R_{\infty}\to\tv$
requires $\poly(d,D,\nicefrac{M_{\infty}}{\veps})$ many queries to
achieve $\veps$-$\tv$ distance.

Lov\'asz and Simonovits \cite{LS90mixing} generalized the conductance
framework to continuous domains, introducing \emph{$s$}-conductance
that lower bounds the conductance of the chain for subsets of measure
at least $s>0$. To prove lower bounds on $s$-conductance, they developed
the well-known localization lemma, proving a Cheeger isoperimetric
inequality for $\pi_{\K}$, with $\cch(\pi_{\K})\gtrsim D^{-1}$ (which
was originally established by Payne and Weinberger \cite{PW60optimal}).
These new tools also improved the complexity of the $\msf{Grid}\text{ }\msf{walk}$,
although it still has dependence of $\poly(d,D,\nicefrac{M_{\infty}}{\veps})$.

As mentioned in the main body of the paper, Lov\'asz and Simonovits
\cite{LS93random} used the $\bw$ to develop an indirect approach
with improved complexity $\poly(d,D,\log\nicefrac{M_{\infty}}{\veps})$.
It was achieved by refining their earlier conductance framework and
localization machinery; they proved that any logconcave distribution
$\pi$ supported on a convex body of diameter $D$ also satisfies
$\cch(\pi)\gtrsim D^{-1}$. \cite{KLS97random} provided a refined
version of these results with improved isoperimetry, an efficient
rounding algorithm, the speedy walk and other extensions. The study
of logconcave sampling was initiated by Applegate and Kannan \cite{AK91sampling}.
Frieze and Kannan \cite{FK99lsi} used a log-Sobolev inequality together
with a grid walk to give an improved complexity for logconcave sampling. 

Recently, $\chr$ ($\msf{CHR}$), a special variant of $\har$, has
been shown to mix in polynomial time with respect to the dimension
and ($\ell_{2}$ or $\ell_{\infty}$) diameter \cite{NS22char,LV23char,fernandez23ell0},
and with a logarithmic dependence on the initial warmness \cite{NRS25sampling}.
This thus results in $\msf{CHR}:\eu R_{2}\to\eu R_{2}$, although
the mixing time is a larger polynomial than that of hit-and-run.

\subsection{KLS conjecture}

The Kannan-Lov\'asz-Simonovits (KLS) conjecture \cite{KLS95isop}
posits that the Cheeger constant of any logconcave probability measure
is bounded below by $\Omega(\norm{\cov\pi}^{-1/2})$. Starting from
$\cch^{2}\gtrsim\tr^{-1}(\cov)$, subsequent work \cite{eldan13thin,LV24eldan,Chen21almost,KL22Bourgain,Klartag23log}---notably
through Eldan's stochastic localization---has improved it to $\cch^{2}\gtrsim(\norm{\cov\pi}\log d)^{-1}$
\cite{Klartag23log}. 

\subsection{Rounding}

A notable application is the rounding of convex bodies, or more generally,
logconcave functions. It aims to make the geometry of the input object
more regular---typically by applying an affine transformation that
brings it into a nearly-isotropic position. This step is often a necessary
preprocessing phase for achieving a fully polynomial-time algorithm,
since worst-case shapes (e.g., highly skewed bodies) can cause sampling
algorithms to suffer from poor mixing.

Earlier approaches to rounding relied on deterministic methods based
on the ellipsoid algorithm \cite{DFK91random,LS90mixing}, which guarantees
$B_{1}(0)\subset\K\subset B_{d^{3/2}}(0)$ after a suitable affine
transformation. To improve the sandwiching ratio, randomized rounding
methods were pioneered in~\cite{LS93random} and further developed
in~\cite{KLS97random,LV06simulated,JLLV21reducing,JLLV26reducing}.
These methods use sampling to estimate the covariance matrix and then
apply affine transformations to round the body. The best known complexity
for isotropic rounding of a convex body is $\Otilde(d^{3.5}\polylog R)$
\cite{JLLV26reducing}, with the analysis streamlined in~\cite{KZ24covariance}.
These rounding techniques have been extended to logconcave distributions~\cite{LV06fast,KV25sampling},
while preserving the query complexity.

\subsection{Volume computation and integration}

One of the most well-known applications of uniform sampling is volume
computation, a problem with a long and rich history in theoretical
computer science and convex geometry; given $\veps>0$, the goal is
to estimate the volume $V$ of a convex body within a relative $\veps$-error,
$(1-\veps)\vol\K\leq V\leq(1+\veps)\vol\K$.

Computing the volume of a high-dimensional convex body is known to
be intractable via deterministic algorithms~\cite{elekes86geometric,BF87computing}.
The breakthrough result by Dyer, Frieze, and Kannan~\cite{DFK91random}
showed that this task admits a randomized polynomial-time algorithm
using random walks. The standard technique introduced therein, referred
to as the \emph{multi-phase Monte Carlo} method, constructs a sequence
of convex bodies such that each successive pair is close in some sense.
Rewriting the volume of $\K$ as a telescoping product of volume ratios
between successive bodies, one can efficiently estimate each ratio
using samples obtained via uniform sampling over convex bodies.

This framework initiated a long line of research focused on improving
the mixing time, oracle complexity, and practical implementability
of such algorithms~\cite{lovasz90compute,AK91sampling,LS90mixing,LS93random,KLS95isop,KLS97random,LV06hit,CV15bypass,CV18Gaussian,JLLV21reducing}.
The current fastest algorithm for well-rounded convex bodies (a weaker
condition than near-isotropic) is due to \cite{CV18Gaussian} with
complexity $\Otilde(d^{3}/\veps^{2})$. A natural generalization of
volume computation is the integration of a logconcave function. This
general problem was studied in~\cite{LV06fast,KV25sampling}, establishing
the same complexity $\Otilde(d^{3}/\veps^{2})$ when the corresponding
distribution is well-rounded; this can be applied after rounding for
an arbitrary logconcave function. See \cite[\S5.2]{KV25sampling}
for a streamlined analysis.
\end{document}